\pdfoutput=1
\RequirePackage{ifpdf}
\ifpdf 
\documentclass[pdftex]{sigma}
\else
\documentclass{sigma}
\fi

\usepackage{theoremref}
\usepackage[algonl,boxed,norelsize,lined]{algorithm2e}
\usepackage{mathtools}

\usepackage{bbm,arydshln}
\usepackage{tikz,tikz-cd}
\usetikzlibrary{decorations.pathreplacing}
\usetikzlibrary{calc,cd}

\newtheorem{Theorem}{Theorem}[section]
\newtheorem{Lemma}[Theorem]{Lemma}
\newtheorem*{Theorem*}{Theorem}
\newtheorem*{Corollary*}{Corollary}
\newtheorem{Corollary}[Theorem]{Corollary}
\newtheorem{Proposition}[Theorem]{Proposition}
\newtheorem{Theorem-Definition}[Theorem]{Theorem-Definition}

{\theoremstyle{definition}
\newtheorem{Definition}[Theorem]{Definition}
\newtheorem{Example}[Theorem]{Example}

\newtheorem{Notation}[Theorem]{Notation}

\newtheorem{Remark}[Theorem]{Remark}}

\numberwithin{equation}{section}

\allowdisplaybreaks
\DeclareMathOperator{\init}{in}
\DeclareMathOperator{\Gr}{Gr}
\DeclareMathOperator{\GF}{GF}

\DeclareMathOperator{\trop}{Trop}

\DeclareMathOperator{\val}{\mathfrak v}

\DeclareMathOperator{\univ}{univ}

\definecolor{caribbeangreen}{rgb}{0.0, 0.8, 0.6}
\definecolor{capri}{rgb}{0.0, 0.75, 1.0}
\begin{document}

\newcommand{\arXivNumber}{2007.14972}

\renewcommand{\PaperNumber}{059}

\FirstPageHeading

\ShortArticleName{Families of Gr\"obner Degenerations, Grassmannians and Universal Cluster Algebras}

\ArticleName{Families of Gr\"obner Degenerations, Grassmannians\\ and Universal Cluster Algebras}

\Author{Lara BOSSINGER~$^{\rm a}$, Fatemeh MOHAMMADI~$^{\rm bc}$ and Alfredo N\'AJERA CH\'AVEZ~$^{\rm ad}$}

\AuthorNameForHeading{L.~Bossinger, F.~Mohammadi and A.~N\'ajera Ch\'avez}

\Address{$^{\rm a)}$~Instituto de Matem\'aticas UNAM Unidad Oaxaca,\\
\hphantom{$^{\rm a)}$}~Le\'on 2, altos, Oaxaca de Ju\'arez, Centro Hist\'orico, 68000 Oaxaca, Mexico}
\EmailD{\href{mailto:lara@im.unam.mx}{lara@im.unam.mx}, \href{mailto:najera@matem.unam.mx}{najera@matem.unam.mx}}
\URLaddressD{\url{https://www.matem.unam.mx/~lara/}, \url{https://www.matem.unam.mx/~najera/}}

\Address{$^{\rm b)}$~Department of Mathematics: Algebra and Geometry, Ghent University, 9000 Gent, Belgium}
\EmailD{\href{mailto:fatemeh.mohammadi@ugent.be}{fatemeh.mohammadi@ugent.be}}
\URLaddressD{\url{https://www.fatemehmohammadi.com}}

\Address{$^{\rm c)}$~Department of Mathematics and Statistics, UiT -- The Arctic University of Norway,\\
\hphantom{$^{\rm c)}$}~9037 Troms\o, Norway}

\Address{$^{\rm d)}$~Consejo Nacional de Ciencia y Tecnolog\'ia, Insurgentes Sur 1582,\\
\hphantom{$^{\rm d)}$}~Alcald\'ia Benito Ju\'arez, 03940 CDMX, Mexico}

\ArticleDates{Received October 21, 2020, in final form May 29, 2021; Published online June 10, 2021}

\Abstract{Let $V$ be the weighted projective variety defined by a weighted homogeneous ideal $J$ and $C$ a maximal cone in the Gr\"obner fan of $J$ with $m$ rays. We construct a~flat family over $\mathbb A^m$ that assembles the Gr\"obner degenerations of $V$ associated with all faces of $C$. This is a multi-parameter generalization of the classical one-parameter Gr\"obner degeneration associated to a weight. We explain how our family can be constructed from Kaveh--Manon's recent work on the classification of toric flat families over toric varieties: it is the pull-back of a toric family defined by a Rees algebra with base $X_C$ (the toric variety associated to $C$) along the universal torsor $\mathbb A^m \to X_C$. We apply this construction to the Grassmannians ${\rm Gr}(2,\mathbb C^n)$ with their Pl\"ucker embeddings and the Grassmannian ${\rm Gr}\big(3,\mathbb C^6\big)$ with its cluster embedding. In each case, there exists a unique maximal Gr\"obner cone whose associated initial ideal is the Stanley--Reisner ideal of the cluster complex. We show that the corresponding cluster algebra with universal coefficients arises as the algebra defining the flat family associated to this cone. Further, for ${\rm Gr}(2,\mathbb C^n)$ we show how Escobar--Harada's mutation of Newton--Okounkov bodies can be recovered as tropicalized cluster mutation.}


\Classification{13F60; 14D06; 14M25; 14M15; 13P10}

\section{Introduction}

The theory of Gr\"obner fans, introduced by Mora and Robbiano~\cite{mora1988grobner}, is a popular tool in commutative algebra, and one of its modern applications is degenerating ideals into simpler ones such as monomial, binomial or toric ideals.
More precisely, let $\mathbb K$ be an algebraically closed field and $J\subseteq \mathbb{K}[x_1,\dots,x_n]$ a ${\bf d}$-weighted homogeneous ideal for some ${\bf d}\in \mathbb Z_{>0}^n$.
The Gr\"obner fan of $J$ is a complete fan in $\mathbb{R}^n$ whose elements represent weight vectors on the variables $x_1,\dots,x_n$.
Two weight vectors lie in the same open cone if and only if they give rise to the same initial ideal of $J$, see~\thref{def:init}.
The ideal $J$ defines a weighted projective variety $V$ inside the weighted projective space $\mathbb P({\bf d})$.
Every open cone in the Gr\"obner fan gives rise to a one-parameter flat family degenerating $V$ to the variety defined by the associated initial ideal of $J$.
This construction is realized by choosing a weight in the relative interior of the cone.

We modify the classical one-parameter construction as follows: for a maximal cone $C$ in the Gr\"obner fan of $J$ we choose integral generators of its rays $r_1,\dots,r_m$ and denote by ${\bf r}$ the $(m\times n)$-matrix whose rows are $r_1,\dots,r_m$.
For an element $f=\sum_{\alpha\in \mathbb Z_{\ge 0}^n} c_\alpha{\bf x}^\alpha$ of $J$,
we define $\mu(f)\in \mathbb Z^m$ as the vector whose $i$-th entry is $\min_{c_\alpha\not=0}\{r_i\cdot\alpha\}$.
Then the {\it lift} of $f$ is
\begin{displaymath}
 \tilde f:=f\big({\bf t}^{{\bf r}\cdot e_1}x_1,\dots, {\bf t}^{{\bf r}\cdot e_n}x_n\big) {\bf t}^{-\mu(f)},
\end{displaymath}
where ${\bf t}^{\bf a}$ for ${\bf a}\in \mathbb Z^m$ denotes the monomial $\prod_{i=1}^mt_i^{a_i}$.
The {\it lifted ideal} $\tilde J \subseteq \mathbb K[t_1,\dots,t_m][x_1,\dots,x_n]$ is the ideal generated by the lifts of all polynomials in $J$.
We prove that $\tilde J$ is generated by the lifts of elements of the reduced Gr\"obner basis for $J$ and $C$, see~\thref{prop:lifted generators}.
The lifts of these elements are independent of the choice of ${\bf r}$ and homogeneous with respect to the ${\bf d}$-grading on $x_i$'s, see~\thref{prop:change rays}.
Consequently, $\tilde J$ is independent of the choice of ${\bf r}$ and it defines a~variety inside $\mathbb P({\bf d})\times \mathbb A^m$. Our first main result is the following.

\begin{Theorem}[\thref{thm:family}]
Let $J$ be a weighted homogeneous ideal, $C$ a maximal cone in the Gr\"obner fan of $J$ and ${\bf r}$ an $(m\times n)$-matrix whose rows are integral ray generators of $C$.
Then the algebra $\tilde A:=\mathbb K[t_1,\dots,t_m][x_1,\dots,x_n]/\tilde J$ is a free $\mathbb K[t_1,\dots,t_m]$-algebra.
It defines a flat family
 \begin{displaymath}
 \begin{tikzcd}[ampersand replacement=\&,cramped]
 {\rm Proj}\big(\tilde A\big) \ar[r,hook]\ar[d,"\pi"] \& \mathbb P({\bf d})\times \mathbb A^m \ar[ld,twoheadrightarrow]\\
 \mathbb A^m \&
 \end{tikzcd}
 \end{displaymath}
 such that for every face $\tau$ of $C$ there exists ${\bf a}_\tau\in \mathbb A^m$ with fiber
 $\pi^{-1}({\bf a}_\tau)$ isomorphic to the variety defined by the initial ideal associated to $\tau$.
 In particular, generic fibers are isomorphic to~${\rm Proj}(A)$, where $A= \mathbb{K}[x_1, \dots , x_n]/J$, and there exist special fibers for every proper face $\tau\subset C$.
\end{Theorem}

Next, we explain how the algebra $\tilde{A} $ arises in Kaveh--Manon's recent work on the classification of affine toric flat families of finite type over toric varieties~\cite{KM-toricbundles}.
Consider a fan $\Sigma$ defining a toric variety $X_\Sigma$ which contains a dense torus $T$.
Then a {\it toric family} is a $T$-equivariant flat sheaf $\mathcal A$ of positively graded algebras of finite type over $X_\Sigma$ such that: $(i)$ the relative spectrum of $\mathcal A$ has reduced fibers and $(ii)$ its generic fibers are isomorphic to the spectrum of some positively graded $\mathbb K$-algebra $A$.
Such families are classified by so-called PL-quasivaluations on $A$ whose codomain is the semifield of piecewise linear functions on the intersection of $\Sigma$ with the cocharacter lattice of $T$ (see~\cite[Section~1.1]{KM-toricbundles} or Section~\ref{sec:toric families} below).
Given a PL-quasivaluation, Kaveh--Manon construct a sheaf of Rees algebras on $X_\Sigma$ that is a toric family.
In the special case when $\Sigma$ is a~cone, their construction yields a single Rees algebra rather than a sheaf.
We prove the following result:

\begin{Theorem}[\thref{thm:KM max cone}]
Let $J\subseteq\mathbb K[{x_1,\dots,x_n}]$ be a weighted homogeneous ideal and~$C$ a~maximal cone in the Gr\"obner fan of $J$.
Let $\mathcal R_C(A)$ be the Rees algebra associated to the PL-quasivaluation on $A=\mathbb K[{x_1,\dots,x_n}]/J$ defined by $C$ and let $\psi\colon \operatorname{Spec}(\mathcal R_C(A))\to X_C$ be the corresponding toric family.
The morphism $\pi\colon \operatorname{Spec}\big(\tilde A\big)\to \mathbb A^m$ fits into a pull-back diagram as follows:
 \begin{displaymath}
 \begin{tikzcd}[ampersand replacement=\&,cramped]
 \operatorname{Spec}(\mathcal R_C (A))\ar[d,"\psi"'] \&
 \operatorname{Spec}\big(\tilde A\big)\ar[d,"\pi"] \ar[l]\\
 X_{C}\& \mathbb A^m. \ar[l,"p_C"]
 \end{tikzcd}
 \end{displaymath}
Here $p_C\colon \mathbb A^m\to X_C$ is the universal torsor of $X_C$ $($see~\eqref{eq:Cox C} for details$)$ and $m$ is the number of rays of $C$.
Both flat families, the one defined by $\psi$ and the one defined by $\pi$, are direct sums of line bundles indexed by the standard monomial basis associated to the maximal cone $C$.
\end{Theorem}

Toric degenerations are a particular class of flat families with a single special fiber that is a toric variety.
They are of significant interest and have been widely studied in recent years.
Among the faces of the maximal cone $C$ one may find binomial prime initial ideals that hence define toric Gr\"obner degenerations of $V$.
Such faces lie in the tropicalization of $J$ denoted $\trop(J)$, which by definition is the subfan of the Gr\"obner fan consisting of cones whose associated initial ideal does not contain monomials.
Let $\Sigma$ be the intersection of $C$ with the tropicalization of $J$.
Then the affine space $\mathbb A^m$ contains the universal torsor $\mathcal T_\Sigma$ of the toric variety $X_\Sigma$ and all fibers of $\pi^{-1}(\mathcal T_\Sigma)\to \mathcal T_\Sigma$ correspond to initial ideals of cones in $\Sigma$.
In~\thref{thm:KM vs grobner} we show how the family defined by $\pi\vert_{\pi^{-1}(\mathcal T_\Sigma)}$ also arises as a pull-back from a toric family defined by a~sheaf of Rees algebras on $X_\Sigma$.

\medskip

\noindent{\bf Grassmannians and cluster algebras.}
The interest of studying toric degenerations in the context of cluster algebras has grown in the last years (see for example~\cite{BFMN,GHKK}).
Therefore, we would like to understand the framework introduced above from the perspective of cluster algebras.
As a first step in this direction, we analyze in depth the situation for the Grassmannians $\Gr(2,\mathbb C^n)$ and $\Gr\big(3,\mathbb C^6\big)$ whose coordinate rings are cluster algebras.

\looseness=-1 For $\Gr(2,\mathbb C^n)$ we choose its Pl\"ucker embedding and obtain the homogeneous coordinate ring~$A_{2,n}$ as a quotient of the polynomial ring in Pl\"ucker variables $\mathbb C[p_{ij}\colon 1\le i<j\le n]$ by the Pl\"ucker ideal $I_{2,n}$.
It was shown in~\cite{FZ02} that $A_{2,n}$ 
is a cluster algebra in which the cluster variables are in one-to-one correspondence with Pl\"ucker coordinates $\bar p_{ij}\in A_{2,n}$ and the
seeds are in one-to-one correspondence with triangulations of the $n$-gon.
Every triangulation of the $n$-gon gives rise to a toric degeneration of $\Gr(2,\mathbb C^n)$ obtained by adding {\it principal coefficients} to $A_{2,n}$ at the corresponding seed~\cite{GHKK}.
Principal coefficients were introduced by Fomin and Zelevinsky in~\cite{FZ_clustersIV} and are a core concept in the theory of cluster algebras.
A central piece of the construction is endowing every cluster variable (i.e.,~every Pl\"ucker coordinate) with~a~so-called ${\bf g}$-vector depending on the fixed seed.
We want to understand the toric degenerations coming from principal coefficients in the context of Gr\"obner theory:
to achieve this first fix a~triangulation $T$.
We use ${\bf g}$-vectors for Pl\"ucker coordinates to construct a weight vector ${\bf w}_T$, see~\thref{def:wt of T}.
In~\thref{thm:toric wt ideals} we prove that the initial ideal with respect to ${\bf w}_T$ is binomial and prime, hence ${\bf w}_T$ lies in the tropicalization of $I_{2,n}$.
Moreover, the central fiber of the Gr\"obner degeneration induced by ${\bf w}_T$ is isomorphic to the central fiber of the toric degeneration induced by endowing~$A_{2,n}$ with principal coefficients at the seed determined by $T$, see~\thref{cor:central fibre}.


There is a single object in cluster theory that simultaneously encodes principal coefficients at all seeds and thus the ${\bf g}$-vectors of all Pl\"ucker coordinates with respect to all triangulations.
Namely, the cluster algebra with {\it universal coefficients} $A^{\rm univ}_{2,n}$ associated to $A_{2,n}$, see~\thref{def:Plucker univ}.
In spirit, this algebra is very similar to the algebra defining the flat family associated to a maximal cone in the Gr\"obner fan of $I_{2,n}$ as it encodes various toric degenerations of $\Gr(2,\mathbb C^n)$ at the same time.
It is therefore natural to ask if the cluster algebra with universal coefficients fits into the above framework.
The following result answers this question 
for $\Gr(2,\mathbb C^n)$:

\begin{Theorem}
There exists a maximal cone $C$ in the Gr\"obner fan of $I_{2,n}$ whose rays are in bijection with diagonals of the $n$-gon. Moreover, the unique cone $C$ has the following properties:
\begin{itemize}\itemsep=0pt
 \item[$(i)$] The standard monomial basis for $A_{2,n}$ associated with $C$ coincides with the basis of cluster monomials $($\thref{cluster vs SM basis}$)$.
 \item[$(ii)$] For every triangulation $T$ of the $n$-gon the weight vector ${\bf w}_T$ lies in the boundary of $C$ and the intersection of $C$ with the tropicalization of $I_{2,n}$ is the totally positive part of $\trop(I_{2,n})$ $($\thref{prop:unique mono ideal} and~\thref{thm:trop+}$)$.
 \item[$(iii)$] Let $\tilde I_{2,n}$ be the lift of $I_{2,n}$ with respect to $C$ and denote by $\tilde A_{2,n}$ the quotient $\mathbb C[t_{ij}\colon 2\le i+1<j\le n][p_{ij}\colon 1\le i<j\le n]/\tilde I_{2,n}$.
 Then there exists a canonical isomorphism between the cluster algebra with universal coefficients $A_{2,n}^{\rm univ}$ and $\tilde A_{2,n}$ $($\thref{thm:A univ as quotient}$)$.
 \item[$(iv)$] The monomial initial ideal of $I_{2,n}$ with respect to $C$ is squarefree and coincides with the Stanley--Reisner ideal of the cluster complex $($\thref{cor:Stanley-Reisner}$)$.
\end{itemize}
\end{Theorem}

Similarly, we describe the cluster algebra with universal coefficients for $\Gr\big(3,\mathbb C^6\big)$ from the viewpoint of Gr\"obner theory.
In this case, we fix its {\it cluster embedding} obtained as follows: consider the Pl\"ucker embedding and let $A_{3,6}$ be the homogeneous coordinate ring.
As a cluster algebra of type $\mathtt D_4$, $A_{3,6}$ has 22 cluster variables, 20 of which are the Pl\"ucker coordinates.
The~addi\-tional two cluster variables are homogeneous binomials in Pl\"ucker coordinates of deg\-ree~2.
Hence, we can present $A_{3,6}$ as the quotient of a polynomial ring in 22 variables by a weighted homogeneous ideal denoted by $I^{\rm ex}$.
This yields an embedding of $\Gr\big(3,\mathbb C^6\big)$ in the weighted projective space $\mathbb P(1,\dots,1,2,2)$, where the Pl\"ucker coordinates correspond to the coordinates of weight one, and the additional cluster variables to those of weight two, see Section~\ref{sec:Gr(3,6)}.
Our~main result is the following:

\begin{Theorem}
There exists a unique maximal cone $C$ in the Gr\"obner fan of $I^{\rm ex}$ such that
\begin{itemize}\itemsep=0pt
 \item[$(i)$] The algebra $\tilde A_{3,6}$ is canonically isomorphic to the cluster algebra with universal coeffici\-ents~$A^{\rm univ}_{3,6}$, where rays of $C$ are identified with mutable cluster variables of $A_{3,6}$ $($Theo\-rem~$\ref{thm:Gr36}(ii))$.
 \item[$(ii)$] For every seed of $A_{3,6}$ there exists a face of $C$ whose associated initial ideal is a totally positive binomial prime ideal. More precisely, the intersection of $C$ with $\trop(I^{\rm ex})$ is the totally positive part of $\trop(I^{\rm ex})$ $($\thref{thm:Gr36}$(iii))$.
 \item[$(iii)$] The monomial initial ideal $\init_C(I^{\rm ex})$ associated to $C$ is squarefree and coincides with the Stanley--Reisner ideal of the cluster complex. In particular, the basis of standard monomials associated to $C$ coincides with the basis of cluster monomials $($\thref{cor:Stanley-Reisner}$)$.
\end{itemize}
\end{Theorem}

As mentioned before, the toric fibers of the above families are of particular interest.
They arise from maximal cones in the tropicalization whose initial ideal is binomial and prime; such cones are called {\it maximal prime cones}.
The corresponding projective toric varieties (respectively, their normalizations) have associated polytopes.
Following~\cite{KM16} these polytopes can be realized as Newton--Okounkov bodies of full rank valuations constructed from maximal prime cones.
In~the recent preprint~\cite{EH-NObodies} Escobar and Harada study wall-crossing phenomena of Newton--Okounkov bodies associated to maximal prime cones that intersect in a facet.
They give piecewise linear maps called {\it flip} and {\it shift} that relate the Newton--Okounkov bodies.
For cluster algebras like~$A_{3,6}$ and~$A_{2,n}$ it has been shown in several cases that Newton--Okounkov bodies can equivalently be obtained from the cluster structure (see, e.g.,~\cite{BCMN,BF19,FO20,RW17}).
Hence, one might wonder how Escobar--Harada's wall-crossing formulas can be understood in the context of cluster algebras.
In~the particular case of $A_{2,n}$ each Newton--Okounkov body arising from $\trop(I_{2,n})$ is unimodularly equivalent to the convex hull of {\bf g}-vectors of Pl\"ucker coordinates for some triangulation.
More precisely, we obtain the following result:

\begin{Corollary}[\thref{cor:mutate NObodies}]
Let $\sigma_1$ and $\sigma_2$ be two maximal prime cones in $\trop(I_{2,n})$ that intersect in a facet. Then their associated Newton--Okounkov bodies
are $($up to unimodular equivalence$)$ related by a shear map obtained from tropicalizing a cluster mutation.
\end{Corollary}

In combination with Escobar--Harada's results about the flip map for $\Gr(2,\mathbb C^n)$ this corollary implies that it is of cluster nature (for details see Section~\ref{sec:NO}).

We would like to remark that this paper is not the first to make a connection between cluster algebras and Gr\"obner theory.
In~\cite{MRZ18} Muller, Rajchgot and Zykoski obtained presentations for lower bounds of cluster algebras using Gr\"obner theory.
Further, it is worth noticing that the theory of universal coefficients for cluster algebras is particularly well-developed for finite and surface type cluster algebras, see~\cite{Rea_surface}.
We believe that the above results can be extended to projective varieties containing a cluster variety of finite type.
It is further an interesting and challenging problem to extend the results of this paper to cluster algebras of (infinite) surface type as it would involve Gr\"obner theory of non-Noetherian algebras.

\medskip

\noindent{\bf Structure of the paper.} In Section~\ref{sec:prep} we recall the background on weighted projective varieties and Gr\"obner basis theory.
In Section~\ref{sec:flat families} we introduce the construction of the flat families and prove the main theorem.
We relate to Kaveh--Manon's work in Section~\ref{sec:toric families}.
In Section~\ref{sec:Gr} we turn to the Grassmannians $\Gr(2,\mathbb C^n)$ and $\Gr\big(3,\mathbb C^6\big)$.
We recall the background on cluster algebras and universal coefficients in Section~\ref{sec:cluster}.
We explain in detail how toric degenerations from cluster algebras arise as Gr\"obner degenerations for $\Gr(2,\mathbb C^n)$ in Section~\ref{sec: toric deg}.
Then we apply the main construction to $\Gr(2,\mathbb C^n)$ in Section~\ref{sec:monomial_degeneration} and afterwards to $\Gr\big(3,\mathbb C^6\big)$ in Section~\ref{sec:Gr(3,6)}.
We~explore further connections to Escobar--Harada's work in Section~\ref{sec:NO}.
Finally, in the Appendix~\ref{sec:data Gr36} we present computational results used for the application to $\Gr\big(3,\mathbb C^6\big)$.

\section{Preliminaries}\label{sec:prep}
We first fix our notation throughout the note. Let $\mathbb{K}$ be an algebraically closed field. We are mainly interested in the case when $\mathbb{K}= \mathbb{C}$. In the polynomial ring $\mathbb K[x_1, \dots , x_n]$, we fix the notation ${\bf x}^\alpha$ with $\alpha=(a_1,\ldots,a_n)\in \mathbb Z_{\ge 0}^n$ denoting the monomial $x_1^{a_1}\dots x_n^{a_n}$.
Throughout when we write $f = \sum_{\alpha\in \mathbb{Z}^n_{\geq 0}} c_{\alpha} {\bf x}^{\alpha}$ for $f \in \mathbb K[x_1, \dots , x_n]$ we refer to the expression of $f$ in the basis of~monomials.

\subsection{Weighted projective varieties}
Fix a vector $ {\bf d}=(d_1, \dots , d_n)\in \mathbb R^n$.
Let $\mathbb K_{\bf d}[x_1, \dots , x_n]$ be the polynomial ring $\mathbb K [x_1, \dots , x_n]$ endowed with the $\mathbb R$-grading determined by setting $\text{deg} (x_i)=d_i$ for $i = 1, \dots , n$.

\begin{Definition}\thlabel{def:d homog}
For $c \in \mathbb K^{\ast}$ and $ \alpha=(a_1,\dots , a_n) \in \mathbb Z_{\geq 0}^n$ the {\it weight} of the monomial $ c {\bf x}^{\alpha}\in \mathbb K_{\bf d}[x_1, \dots , x_n]$ is ${\bf d}\cdot \alpha = \sum_{i=1}^nd_ia_i$.
We extend this definition and say $f \in \mathbb K_{\bf d}[x_1, \dots , x_n]$ is {\bf d}-{\it homogeneous} if $f$ can be expressed as a sum of monomials of the same weight.
Similarly, an~ideal $J\subseteq \mathbb K_{\bf d}[x_1, \dots , x_n]$ is {\bf d}-{\it homogeneous} if it is generated by ${\bf d}$-homogeneous elements.
\end{Definition}

Assume ${\bf d}\in \mathbb Z^n_{> 0}$. In this case $\mathbb K_{\bf d}[x_1, \dots , x_n]$ is $\mathbb Z$-graded and we can define the {\it weighted projective space} $\mathbb P (d_1, \dots , d_n)$ as the quotient of $\mathbb K^n$ under the equivalence relation
\begin{displaymath}
(x_1,\dots , x_n) \sim \big(\lambda^{d_1}x_1, \dots ,\lambda^{d_n} x_n\big) \qquad\text{for}\quad\lambda \in \mathbb K^{\ast}.
\end{displaymath}
For $\mathbb K = \mathbb C$ it is well-known that $\mathbb P (d_1, \dots , d_n)$ is a projective toric variety~\cite{CLS}.
Moreover, as a scheme $\mathbb P(d_1,\dots,d_n)$ is the homogeneous spectrum (or Proj) of $\mathbb K_{\bf d}[x_1,\dots,x_n]$~\cite[Notation~1.1]{Dolgachev}.
For simplicity, we denote $\mathbb P (d_1, \dots , d_n)$ by $\mathbb P ({\bf d})$ and let $[x_1:\dots : x_n]$ denote the class of $(x_1, \dots , x_n)$ in $\mathbb P ({\bf d})$.
Given a {\bf d}-homogeneous ideal $J \subseteq \mathbb K_{\bf d}[x_1, \dots , x_n]$ its set of zeros is
\begin{displaymath}
V(J)=\big\{ [x_1:\dots : x_n]\in \mathbb P ({\bf d}) \colon f(x_1,\dots , x_n)=0\text{ for all }f\in J\big\}.
\end{displaymath}
Conversely, given a subset $X \subseteq \mathbb P (d_1,\dots,d_n)$, its associated ideal $I(X)$ is generated by polynomials:
\begin{displaymath}
\big\{ f \in \mathbb K_{\bf d}[x_1,\dots , x_n] \colon f \text{ is }{\bf d}\text{-homogeneous and }f(p)=0\text{ for all }p\in X \big\}.
\end{displaymath}
Subsets of $\mathbb P({\bf d})$ of the form $V(J)$ for some weighted projective ideal $J$ are called weighted projective varieties.
These sets are the closed sets of a Zariski-type topology on $ \mathbb P({\bf d})$.
The~the\-ory of weighted projective varieties is very similar to the theory of usual projective varieties and background on this theory can be found for example in~\cite{Dolgachev,Hosgood,Reid}. In particular, the projective space $\mathbb{P}^{n-1}$ can be realized as $\mathbb{P}
^{n-1}=\mathbb P ({\bf 1})$, where ${\bf 1}=(1, \dots , 1)\in \mathbb Z^n$. So every projective variety can be considered as a weighted projective variety of $\mathbb P ({\bf 1})$.

\begin{Definition}
The {\it weighted homogeneous coordinate ring} of a weighted projective variety $X\subseteq \mathbb P ({\bf d})$ is defined as
\begin{displaymath}
S(X):= \mathbb{K}_{\bf d}[x_1,\dots , x_n]/ I(X).
\end{displaymath}
\end{Definition}
By construction, $S(X)$ is a positively graded ring. Moreover, the weighted homogeneous coordinate ring of a projective variety (considered as a weighted projective variety) coincides with its homogeneous coordinate ring.
There is a natural notion of morphism between weighted projective varieties which we will not need in this work.
It will be sufficient to recall that if~$X\subseteq \mathbb P(d_1, \dots , d_n)$ and $Y\subseteq \mathbb P(d'_1, \dots , d'_m)$ are weighted projective varieties, then an isomorphism between the graded rings $S(X)$ and $S(Y)$ induces an isomorphism of the weighted projective varieties $X$ and $Y$.

\subsection{Gr\"obner basis theory}
We first review some results in Gr\"obner basis theory 
in order to fix our notation and
keep the paper self-contained. Most of the material here is well-known and we refer to Buchberger's thesis~\cite{buchberger2006bruno} and standard books (e.g.,~\cite{loustaunau1994introduction,herzog2011monomial,kreuzer2005computational,GB_Bernd}) for proofs and more details.

\begin{Definition}\thlabel{def:init}
Let $f = \sum_{\alpha\in \mathbb{Z}^n_{\geq 0}} c_{\alpha} {\bf x}^{\alpha} \in \mathbb K[x_1, \dots , x_n]$.
Given a {\it weight vector} $w\in \mathbb R^n$ the {\bf initial form} of $f$ with respect to $w$ is defined as
\begin{displaymath}
\init_w(f):=\sum_{\alpha\colon w\cdot \alpha=a} c_{\alpha}{\bf x}^{\alpha},
\end{displaymath}
where $a= \min \{w\cdot \alpha \colon c_{\alpha}\not =0\}$. For an ideal $J\subseteq \mathbb K[x_1, \dots , x_n]$ its {\it initial ideal} with respect to $w$ is defined as $\init_w(J)\colon =\langle \init_w(f)\colon f\in J\rangle$.
A finite set $\mathcal{G}=\{g_1,\ldots,g_r\}\subseteq J$ is called a {\it Gr\"obner basis} for $J$
with respect to $w$ if $\init_w(J)=\langle \init_w(g_1),\ldots,\init_w(g_r)\rangle$.
\end{Definition}

\begin{Definition}\thlabel{def:mono order}
A {\it monomial term order} on $\mathbb K[x_1, \dots , x_n]$ is a total order $<$ on the set of monic monomials in $\mathbb K[x_1, \dots , x_n]$ such that for every $\alpha,\beta,\gamma$ in $\mathbb Z_{\ge 0}^n$ we have that
\[
(i)\quad 1\le {\bf x}^\alpha,\qquad \text{and}\qquad
(ii)\quad \text{if}\quad {\bf x}^\alpha<{\bf x}^\beta, \quad \text{then}\quad {\bf x}^{\alpha+\gamma}<{\bf x}^{\beta+\gamma}.
\]
The {\it initial monomial} of an element $f=\sum_{\alpha\in \mathbb{Z}^n_{\geq 0}} c_{\alpha} {\bf x}^{\alpha}\in \mathbb K[x_1, \dots , x_n]$ with respect to $<$ is $\init_<(f):=c_\beta{\bf x}^{\beta}$, where ${\bf x}^\beta=\max_<\{{\bf x}^\alpha\colon c_\alpha\not=0\}$.
The {\it initial ideal} of an ideal $J\subseteq \mathbb K[x_1, \dots , x_n]$ with respect to $<$ is defined as $\init_<(J):=\langle \init_<(f)\colon f\in J\rangle$.
\end{Definition}

Recall that given an ideal $J\subseteq \mathbb K[x_1, \dots , x_n]$ and a monomial term order $<$ there always exists a weight vector $-w\in \mathbb{Z}_{\geq 0}^n$
such that $\init_w(J)=\init_<(J)$, see, e.g.,~\cite[Theorem~3.1.2]{herzog2011monomial} (Note the sign switch here, this is due to our $\min$-convention for initial ideals with respect to weight vectors and $\max$-convention for initial ideals with respect to monomial term orders.)
On the other hand, if $\init_w(J)$ is generated by monomials and there exists a monomial term order $<$ such that $\init_w(J)=\init_{<}(J)$, then
we say that $w$ is {\it compatible} with $<$.

\begin{Definition}\thlabel{def:SMT}
Let $\init_<(J)$ be a monomial initial ideal of the ideal $J$ for some monomial term order $<$ on $\mathbb{K}[x_1,\ldots,x_n]$. 
Then the set $\mathbb B_{<}:=\{\bar{\bf x}^\alpha \colon {\bf x}^\alpha\not \in \init_<(J)\}$ is a vector space basis of~$\mathbb K[x_1, \dots , x_n]/J$ (and $\mathbb K[x_1, \dots , x_n]/\init_<(J)$) called {\it standard monomial basis}.
\end{Definition}

{\samepage\begin{Theorem}[{the division algorithm}]\thlabel{division alg}
Fix a monomial term order $<$ on $\mathbb K[x_1,\dots,x_n]$ and let $g_1,\dots,g_s$ be non-zero polynomials in $\mathbb K[x_1,\dots,x_n]$. Given $0\not=f\in\mathbb K[x_1,\dots,x_n]$ there exist $f_1,\dots,f_s,f'\in \mathbb K[x_1,\dots,x_n]$ with $f= f_1g_1+\dots+f_sg_s+f'$, such that
\begin{itemize}\itemsep=0pt
\item[$(i)$] if $f'\neq0$ then no monomial of $f'$ is divisible by any of the monomials $\init_<(g_1),\dots,\init_<(g_s);$
\item[$(ii)$] if $f_i\not=0$ then $\init_<(f)\ge\init_<(f_ig_i)$ for all $i$.
\end{itemize}
We say that $f$ reduces to $f'$ with respect to $\{g_1,\dots,g_s\}$.
\end{Theorem}

}

\begin{Definition}
Let $<$ be a monomial term order on $\mathbb K[x_1, \dots , x_n]$ and $\mathcal{G}=\{g_1,\ldots,g_s\}$ a finite generating set of an ideal $J$.
Then the {\it $S$-polynomial} of $g_i$ and $g_j$ is defined as
\begin{displaymath}
S(g_i,g_j):=\frac{{\rm lcm}\big(\init_{<}(g_i),\init_{<}(g_j)\big)}{\init_{<}(g_i)}g_i-\frac{{\rm lcm}\big(\init_{<}(g_i),\init_{<}(g_j)\big)}{\init_{<}(g_j)}g_j.
\end{displaymath}
We say that {\it Buchberger's criterion} holds if for all $1\le i<j\le s$, the $S$-pairs reduce to zero with respect to $\{g_1,\dots,g_s\}$.
If Buchberger's criterion holds, then $\mathcal{G}$ forms a Gr\"obner basis for~$J$ with respect to $< $.
Moreover, a Gr\"obner basis $\mathcal{G}$ for $J$ with respect to $<$ is {\it reduced} if
\begin{itemize}\itemsep=0pt
 \item[$(i)$] $\init_<(g_i)$ is monic for all $1\leq i\leq s$, and
\item[$(ii)$] for $i\neq j$ no monomial in $g_i$ is divisible by $\init_<(g_j)$.
\end{itemize}
A reduced Gr\"obner basis for $J$ with respect to $<$ always exists and is unique, see, e.g.,~\cite[Theorem~2.2.7]{herzog2011monomial}.
We let $\mathcal{G}_{<}(J)$ denote the reduced Gr\"obner basis for $J$ with respect to $<$.
\end{Definition}

Studying all possible initial ideals of a given ideal leads to the notion of Gr\"obner fan (see, e.g.,~\cite[Proposition 2.4]{GB_Bernd}):

\begin{Definition}\thlabel{def:Gfan}
Let $J\subseteq \mathbb K[x_1, \dots , x_n]$ be an ideal.
The {\it Gr\"obner region} of $J$ denoted by $\text{GR}(J)$ is the set of $w\in \mathbb R^n$ such that there exists a monomial term order $<$ with $\init_<(\init_w(J))=\init_<(J)$.
The {\it Gr\"obner fan} of $J$ denoted by $\GF(J)$ is a fan with support $\text{GR}(J)$ in which a pair of elements $u,w\in \mathbb R^n$ lie in the relative interior of the same cone $C\subset \mathbb R^n$ (denoted by $C^\circ$) if and only if $\init_u(J)=\init_w(J)$.
We introduce the notation $\init_C(J):=\init_w(J)$ for any $w\in C^\circ$.
By~definition of~$\text{GR}(J)$,
for every full-dimensional cone $C$, there exists a monomial term order $<$ such that~$C$ is the topological closure of $\{w\in \mathbb{R}^n\colon \init_w(J)=\init_<(J)\}$.
Moreover, we define the {\it lineality space}~$\mathcal L(J)$ as the linear subspace of $\GF(J)$ that contains all elements $u$ for which $\init_u(J)=J$.
\end{Definition}

Integral weight vectors in $\GF(J)$ can be seen as points in a lattice $N=\mathbb Z^n$ whose dual lattice $M=(\mathbb Z^n)^*$ contains exponent vectors of monomials in $\mathbb K[x_1,\dots,x_n]$.
Consequently, for $w\in N$ and $\alpha\in M$ we denote by $w\cdot \alpha$ the pairing between the two lattices.

\begin{Remark}\label{rem:Eisenbud}
Proposition~15.16 in~\cite{eisenbud2013commutative} gives a criterion for whether a weight vector is compatible with a monomial term order $<$ for $J$, or not. Namely, a weight $w$ is compatible with a monomial term order $<$ if and only if $\init_w(g)=\init_{<}(g)$ for every element of $\mathcal{G}_<(J)$. Let $C$ be the topological closure of the corresponding Gr\"obner cone of $<$, then $w\in C$ if and only if $\init_<(g)=\init_<(\init_w(g))$ for every $g\in \mathcal{G}_<(J)$.
\end{Remark}

\begin{Lemma}\thlabel{lem:wt homog}
Let ${\bf d}\in \mathbb Z_{>0}^n$ and $J\subseteq \mathbb K_{\bf d}[x_1, \dots , x_n]$ be a ${\bf d}$-homogeneous ideal. Then ${\bf d} \in \mathcal{L}(J)$.
More generally, $\mathcal L(J)$ contains all elements ${\bf u}\in \mathbb R^n$ such that $J$ is ${\bf u}$-homogeneous.
\end{Lemma}
\begin{proof}
Observe that for every ${\bf d}$-homogeneous polynomial $f\in \mathbb K_{\bf d}[x_1, \dots , x_n]$ the equality $\init_{\bf d}(f)=f$ holds.
Next, every $f \in J$ can be uniquely written as a finite sum $f=\sum_{i=0}^{d} f_i$ for some $d\in \mathbb Z_{\geq 0}$ and $f_i\in J$ of {\bf d}-degree $i$ (see~\cite[Lemma~3.0.7]{Hosgood}).
Then $\init_{\bf d}(f)=f_j$, where $j=\min\{ i \colon f_i \neq 0\} $.
In particular, for any $f \in J $ we have that $\init_{\bf d}(f) \in J$.
Therefore, $\init_{\bf d}(J )= J$.
\end{proof}

\begin{Lemma}\thlabel{lem:init ell}
Fix an arbitrary monomial term order $<$. Then for every $\ell \in \mathcal L(J)$ and every $g\in \mathcal G_<(J)$ we have $\init_\ell(g)=g$.
In particular, if $J$ is ${\bf d}$-homogeneous then so is every $g\in \mathcal G_<(J)$.
\end{Lemma}
\begin{proof}
Consider a Gr\"obner basis $F:=\{f_1,\dots,f_t\}$ of $J$ with respect to a fixed element \mbox{$\ell\in\mathcal L(J)$}.
Then $J=\init_{\ell}(J)$ is generated by $\init(F):=\{\init_{\ell}(f_1),\dots,\init_\ell(f_t)\}$.
In particular, we have $\init_{\ell}(\init_{\ell}(f))=\init_\ell(f)$ for all $f\in F$.
From the set $F$ we can construct the reduced Gr\"obner basis $\mathcal G_<(J)$ by doing the following three steps:
First, extend $F$ to a Gr\"obner basis $H$ with res\-pect to $<$ by adding $S$-pairs using Buchberger's criterion.
Then if necessary eliminate elements from $H$ until the outcome is a minimal Gr\"obner basis $G$ (see~\cite[Corollary 1.8.3]{loustaunau1994introduction}).
Finally, reduce all elements in $G$ to obtains $\mathcal G_<(J)$ (see~\cite[Corollary 1.8.6]{loustaunau1994introduction}).
We invite the reader to verify that the property $\init_\ell(g)=g$ holds for all $g\in H$, hence for all $g\in G$ and finally for all $g\in \mathcal G_<(J)$.
\end{proof}

Recall from~\cite[Corollary 5.7]{mora1988grobner} that for a ${\bf d}$-homogeneous ideal $J\subseteq \mathbb K_{\bf d}[x_1, \dots , x_n]$ the support of its Gr\"obner fan, i.e.,~$\text{GR}(J)$, is $\mathbb R^n$. In other words, there exists a compatible monomial term order for every maximal cone $C\in \GF(J)$.

\begin{Lemma}\thlabel{lem:facet}
Let $J\subseteq \mathbb K_{\bf d}[x_1, \dots , x_n]$ be a ${\bf d}$-homogeneous ideal, $C\in\GF(J)$ a maximal cone, and $<$ a compatible monomial term order of $C$.
Consider $v\in C\setminus C^\circ$ and $u\in\mathbb{R}^n$ such that $w:=u+v\in C^\circ$. Then for every $g\in\mathcal{G}_<(J)$, we have $\init_<(g)=\init_{w}(g)=\init_u(\init_v(g))$.
\end{Lemma}
\begin{proof}
Consider an element $g$ in $\mathcal{G}_<(J)$. Since $w\in C^\circ$ we have $\init_<(g)=\init_{w}(g)$. On the other hand, since $v\in C$, by Remark~\ref{rem:Eisenbud} we have that $\init_<(g)=\init_{<}(\init_{v}(g))$.
This implies that $\init_{w}(g)$ is a refinement of $\init_{v}(g)$.
In other words, $\init_{v}(g)$ contains the monomial $\init_{w}(g)$ with possibly some extra terms which will disappear after taking its initial form with respect to $u$.
\end{proof}

We note that~\thref{lem:facet} does not hold for arbitrary elements of $J$. See~\thref{exp:lem facet}.

\begin{Lemma}\thlabel{lem:strictly_convex}
For every cone $C\in\GF(J) \subseteq \mathbb R^n$, the cone $\overline C :=C / \mathcal{L}(J)\subseteq \mathbb R^n / \mathcal{L}(J)$ is strongly convex, that is $\overline C\cap (- \overline C)=\{0\}$.
\end{Lemma}
\begin{proof}
As we have chosen $J$ to be ${\bf d}$-homogeneous every cone in $\GF(J)$ is a face of a full-dimensional cone.
Hence, we may assume without loss of generality that $C$ is a cone of dimension~$n$.
Let $<$ be the corresponding monomial term order and $g\in\mathcal{G}_<(J)$.
Let $v$ and $-v$ be in~$C\setminus C^\circ$ and choose $u,u' \in \mathbb{R}^n$ such that $u+v,u'-v \in C^{\circ}$. By~\thref{lem:facet} we have that
$\init_{<}(g)=\init_u(\init_v(g))=\init_{u'}(\init_{-v}(g))$. However, $\init_{<}(g)$ cannot simultaneously be a monomial in~$\init_v(g)$ and $\init_{-v}(g)$ unless $v\in \mathcal L(J)$.
\end{proof}

\begin{Definition}\thlabel{def:trop}
For a ${\bf d}$-homogeneous ideal $J\subseteq \mathbb K_{\bf d}[x_1, \dots , x_n]$ we define its {\it tropicalization}
\begin{displaymath}
\trop(J):=\big\{w\in \mathbb R^n \colon \init_w(J) \text{ does not contain any monomial}\big\}.
\end{displaymath}
In fact, $\trop(J)$ is a $d$-dimensional subfan of $\GF(J)$, where $d$ is the Krull-dimension of $\mathbb K_{\bf d}[x_1,\dots,\linebreak x_n]/J$.
\end{Definition}

\begin{Definition}\thlabel{def:trop+}
An ideal $J\subseteq \mathbb R[x_1,\dots,x_n]$ is called {\it totally positive} if it does not contain any non-zero element of $\mathbb R_{\ge 0}[x_1,\dots,x_n]$.
For a ${\bf d}$-homogeneous ideal $J\subseteq \mathbb R[x_1,\dots,x_n]$ the {\it totally positive part} of its tropicalization is defined as
\begin{displaymath}
\trop^+(J):=\{w\in \trop(J)\colon \init_w(J) \text{ is totally positive}\}.
\end{displaymath}
Due to~\cite{SW05} $\trop^+(J)$ is a closed subfan of $\trop(J)$ (that may be empty).
\end{Definition}

\section{Families of Gr\"obner degenerations}\label{sec:flat families}

In this section, we introduce the main construction of the paper.
Let $J \subseteq \mathbb K_{\bf d}[x_1, \dots , x_n]$ be a~{\bf d}-homogeneous ideal, $C $ a maximal cone in $\GF(J)$ and $A$ the quotient $\mathbb K_{\bf d}[x_1, \dots , x_n]/J$.
Our aim is to construct a flat family of degenerations of ${\rm Proj}(A)$ that contains as fibers the varieties corresponding to the various initial ideals associated to the interior of the faces of $C$.
To achieve this we generalize the following classical construction:
take $w\in C^\circ$ and consider the ideal
\begin{equation}\label{eq:flat}
\hat J_w:=\bigg\langle \sum_{\alpha \in \mathbb Z^n_{\ge 0}} c_\alpha {\bf x}^\alpha t^{w\cdot \alpha -\min\{w\cdot \beta \colon c_\beta \not =0\}} \colon \sum_{\alpha \in \mathbb Z^n_{\ge 0}} c_\alpha {\bf x}^\alpha\in J \bigg\rangle \subseteq \mathbb K[t][x_1,\dots,x_n].
\end{equation}

\begin{Remark}\label{rem:flat}
The ideal $\hat J_w$ induces a flat family $\operatorname{Spec} \big(\mathbb K[t][x_1,\dots,x_n]/\hat J_w\big) \to \operatorname{Spec}(\mathbb{K}[t])$
whose fiber over the closed point $(t)$ is isomorphic to $\operatorname{Spec}(\mathbb K[x_1,\dots,x_n]/\init_w(J))$ and the fiber over any non-zero closed point $ (t-a)$ is isomorphic to $\operatorname{Spec}(A)$, see~\cite[Theorem~15.17]{eisenbud2013commutative}.
In fact,~\eqref{eq:flat} and~\cite[Theorem~15.17]{eisenbud2013commutative} hold more generally for arbitrary cones in $\GF(J)$.
However, for the following generalization we focus on maximal cones for simplicity.
\end{Remark}

To generalize the construction~\eqref{eq:flat} we fix vectors $r_1, \dots , r_m \in C$ such that $\{ \overline{r}_1,\dots, \overline{r}_m \}$ is the set of primitive ray generators for $\overline{C}$, which is possible due to~\thref{lem:strictly_convex}.
Using~\thref{lem:wt homog} we may assume if necessary that $r_1,\dots, r_m $ are non-negative or positive vectors.
We denote by~${\bf r}$ the $(m\times n)$-matrix whose rows are $r_1,\dots, r_m$.
Additionally, we write $<$ for a monomial term order compatible with $C$ and denote by $\mathcal G$ the associated reduced Gr\"obner basis.

\begin{Definition}\thlabel{def:lift}
For $f=\sum_{\alpha \in \mathbb Z^n_{\ge 0}} c_\alpha {\bf x}^{\alpha} \in J$ set ${\mu}_{\bf r}(f):=\big(\min_{c_\alpha\not =0}\{r_i\cdot \alpha\}\big)_{i=1,\dots,m}\in \mathbb Z^{m\times 1}$, that is we think of $\mu_{\bf r}(f)$ as a column vector with $m$ entries.
We define the {\it lift of $f$} to be the polynomial $ \tilde f_{\bf r} \in \mathbb K[t_1,\dots,t_m][x_1,\dots,x_n]$ given by the following formula
\begin{displaymath}
 \tilde f_{\bf r} := \tilde f_{\bf r} ({\bf t},{\bf x})
 := f({\bf t}^{{\bf r}\cdot e_1}x_1,\dots, {\bf t}^{{\bf r}\cdot e_n}x_n) {\bf t}^{-\mu_{\bf r}(f)}
 =\sum_{\alpha \in \mathbb Z^n_{\ge 0}} c_\alpha {\bf x}^{\alpha} {\bf t}^{{\bf r}\cdot \alpha -\mu_{\bf r}(f)}.
\end{displaymath}
Similarly, we define the {\it lifted ideal} as
\begin{displaymath} 
\tilde J_{\bf r}:=\big\langle \tilde f_{\bf r} \colon f \in J \big\rangle \subseteq \mathbb K[t_1,\dots,t_m][x_1,\dots,x_n],
\end{displaymath}
and the {\it lifted algebra} as the quotient
\begin{equation}\label{eq:lifted algebra}
\tilde A_{\bf r}:=\mathbb K[t_1,\dots,t_m][x_1,\dots,x_n]/\tilde J_{\bf r}.
\end{equation}
\end{Definition}
We proceed by showing that the lifted algebra $\tilde A_{\bf r}$ is independent of the choice of vectors $r_1, \dots , r_m \in C$ that represent the primitive ray generators $\overline{r}_1,\dots , \overline{r}_m$ of $\overline{C}$. For this, we identify generators of the ideal $\tilde J_{\bf r}$ which is a crucial matter for applications.

\begin{Proposition}\thlabel{prop:change rays}
Suppose $r'_1, \dots , r'_m \in C$ are such that $\overline{r}'_i= \overline{r}_i$ in $ \mathbb R^n / \mathcal{L}(J)$.
Let ${\bf r'}$ be the $(m\times n)$-matrix whose rows are $r'_1,\dots, r'_m$. Then for $g\in \mathcal{G}$ we have that $\tilde{g}_{\bf r}=\tilde{g}_{\bf r'}$.
\end{Proposition}

\begin{proof}
We have that $r'_i=r_i+l_i$ for some $l_i \in \mathcal{L}(J)$. Let $L$ be the $(m \times n)$-matrix whose rows are $l_1, \dots , l_m$. In particular, ${\bf r'}={\bf r}+L$.
Write $g=\sum_{\alpha \in \mathbb Z^n_{\geq 0}}c_{\alpha} {\bf x}^{\alpha}$.
Since $g \in \mathcal{G}$ we have by~\thref{lem:init ell} that the value $l_i \cdot \alpha$ is the same for all $\alpha $ with $c_{\alpha}\neq 0$. Let $a_i$ be this common value.
Observe that $\min_{c_\alpha \neq 0}\{r'_i\cdot \alpha\}= \min_{c_\alpha \neq 0}\{r_i\cdot \alpha\}+a_i$. In particular, we have that the column vector $ (a_1, \dots , a_m)$ is equal to $ L \cdot \alpha$ for all $\alpha$ with $c_{\alpha}\neq 0$ and therefore $\mu_{\bf r'}(f)= \mu_{\bf r}(f)+L\cdot \alpha$.
Finally, we compute
\begin{displaymath}
\tilde{g}_{\bf r'}=\sum_{\alpha \in \mathbb Z^n_{\geq 0}}c_{\alpha} {\bf x}^{\alpha}{\bf t}^{{\bf r}'\cdot \alpha-\mu_{\bf r'}(g)}=
\sum_{\alpha \in \mathbb Z^n_{\geq 0}}c_{\alpha} {\bf x}^{\alpha}{\bf t}^{{\bf r}\cdot \alpha+ L\cdot \alpha -(\mu_{\bf r}(g)+L\cdot \alpha)}=
\tilde{g}_{\bf r}.\tag*{\qed}
\end{displaymath}
\renewcommand{\qed}{}
\end{proof}

In the following, when there is no risk of confusion, we write $\tilde J$ for $\tilde J_{\bf r},\tilde f$ for $\tilde f_{\bf r}$ and $\tilde A$ for $\tilde A_{\bf r}$.

{\samepage\begin{Notation}\thlabel{notation}
For ${\bf a}=(a_1, \dots , a_m)\in \mathbb K^m$ denote by $({\bf t}-{\bf a})$ the maximal ideal $\langle t_1-a_1,\dots$, \mbox{$t_m-a_m\rangle$} of $\mathbb K[t_1,\dots,t_m]$.
Then we denote by $\tilde A\vert_{{\bf t}={\bf a}}$ the tensor product $\tilde A\otimes_{\mathbb K[t_1,\dots,t_m]} \mathbb K[t_1,\dots,t_m]$ $/({\bf t}-{\bf a})$.
Additionally, for $f\in \mathbb K[t_1,\dots,t_m][x_1,\dots,x_n]$ let $f\vert_{{\bf t}={\bf a}} \in \mathbb K[x_1,\dots,x_n]$ be the evaluation of $f$ at $t_i = a_i$ for $i \in [m]$.
Similarly, denote by $\tilde J\vert_{{\bf t}={\bf a}}$ the ideal of $\mathbb K[x_1,\dots,x_n]$ generated by the set $\lbrace \tilde{f}\vert_{{\bf t}={\bf a}} \colon f\in J \rbrace$.
For a set $B\subset \mathbb K[t_1,\dots,t_m][x_1,\dots,x_n]$ let $B\vert_{{\bf t}={\bf a}}:=\{b\vert_{{\bf t}={\bf a}}\colon b\in B\}$.
Recall the isomorphism $\mathbb K[t_1,\dots,t_m]/({\bf t}-{\bf a})\to \mathbb K$ that sends an element $\bar f$ to $f\vert_{{\bf t}={\bf a}}$.
It extends to an isomorphism $\tilde A\vert_{{\bf t}={\bf a}}\cong \mathbb K[x_1,\dots,x_n]/\tilde J\vert_{{\bf t}={\bf a}}$, which explains our choice of notation.
\end{Notation}

Denote ${\bf 1}:=(1,\dots,1)\in \mathbb Z^m$.

}

\begin{Lemma}\thlabel{lem:w'homogeneous}
Let $w=(w_1,\dots,w_n)=r_1+\cdots+r_m$, and $w' = (-1,\ldots,-1,w_1,\ldots,w_n)\in \mathbb{Z}^{m+n}$, where $w_i$ is the weight of $x_i$ and $-1$ is the weight of $t_j$ for $i=1,\ldots,n$ and $j=1,\ldots,m$.
Then for every $f\in J$ its lift $\tilde f$ is $w'$-homogeneous.
In particular, the ideal $\tilde{J}$ is ${w'}$-homogeneous.
\end{Lemma}
\begin{proof}
Let $f=\sum_{\alpha\in \mathbb Z_{\ge 0}^n} c_\alpha{\bf x}^{\alpha}$.
Then the lift of $f$ is by definition $\tilde f=\sum c_\alpha{\bf x}^\alpha{\bf t}^{{\bf r}\cdot \alpha-\mu(f)}$.
Notice that the $w'$-weight of a monomial $c_\alpha{\bf x}^\alpha{\bf t}^{{\bf r}\cdot \alpha-\mu(f)}$ in $\tilde f$ is
\begin{displaymath}
w\cdot \alpha -{\bf 1}\cdot ({\bf r}\cdot \alpha-\mu(f)) = \sum_{i=1}^m r_i\cdot \alpha - \sum_{i=1}^m r_i\cdot \alpha + \sum_{i=1}^m \min_{c_\beta\not=0}(r_i\cdot \beta) = \sum_{i=1}^m \min_{c_\beta\not=0}(r_i\cdot \beta).
\end{displaymath}
This implies the first claim. The second claim is immediate as $\tilde J=\langle \tilde f\colon f\in J\rangle$.
\end{proof}

\begin{Remark}\thlabel{rmk:homogeneous weights}
In~\thref{lem:w'homogeneous} more generally we can choose
$v=c_1r_1+\dots+c_mr_m\in C^\circ$ with $c_i\in \mathbb R_{\ge 0}$ and $v'=(-c_1,\dots,-c_m,v_1,\dots,v_n).$
Then $\tilde J$ is $v'$-homogeneous.
\end{Remark}

\begin{Lemma}\thlabel{lem:homogeneous}
Let $h\in \tilde J$ be $w'$-homogeneous with $w'$ as in~\thref{lem:w'homogeneous}.
Then $h={\bf t}^{\bf u}\tilde{f}$ for some $f\in J$ and ${\bf u}\in \mathbb{Z}_{\geq 0}^m$.
\end{Lemma}

\begin{proof}
Let $h=\sum_{i=1}^s{p_i}\tilde{f}_i$ with $f_i\in J$ and $p_i\in\mathbb{K}[t_1,\ldots,t_m][x_1,\ldots,x_n]$ $w'$-homogeneous.
Then
$h\vert_{{\bf t}={\bf 1}}=\sum_{i=1}^s p_i\vert_{{\bf t}={\bf 1}}f_i\in J$. Furthermore, since $h$ is $w'$-homogeneous, $h={\bf t}^{\bf u}\widetilde{h\vert_{{\bf t}={\bf 1}}}$ for some ${\bf u}\in \mathbb{Z}_{\geq 0}^m$. Hence, we can take $h\vert_{{\bf t}={\bf 1}}$ as $f$.
\end{proof}

\begin{Lemma}\thlabel{lem:lifts}
Consider $f=\sum_{\alpha \in \mathbb Z^n_{\ge 0}} c_\alpha{\bf x}^\alpha\in J$ with a unique monomial $c_\beta{\bf x}^\beta\in \init_<(J)$. That is, for every $\alpha\neq\beta$ with $c_\alpha\neq 0$, there exists no $g'\in\mathcal G$ such that ${\bf x}^\alpha$ is divisible by $\init_<(g')$.
Then
\begin{equation}\label{eq:lifted element}
\tilde f = \sum_{\alpha \in \mathbb Z^n_{\ge 0}} c_\alpha {\bf x}^{\alpha} {\bf t}^{{\bf r}\cdot(\alpha -\beta)} = c_\beta {\bf x}^{\beta} + \sum_{\alpha\not=\beta} c_\alpha{\bf x}^{\alpha}{\bf t}^{{\bf r}\cdot (\alpha-\beta)},
\end{equation}
contains a unique monomial with coefficient in $\mathbb K$.
In particular, this is the case for elements of the reduced Gr\"obner basis $\mathcal G$.
\end{Lemma}
\begin{proof}
The assumption that only the monomial $c_\beta{\bf x}^\beta$ of $f$ is contained in $\init_<(J)$ ensures that~$c_\beta{\bf x}^\beta$ is also a monomial appearing in $\init_{r_i}(f)$ for all $i$. To see this, assume there is a ray $r_i$ such that $c_\beta{\bf x}^\beta$ is not a monomial in $\init_{r_i}(f)$ and let $w'=\sum_{j\not=i} r_j$. By~\cite[Lemma~2.4.5]{M-S} we can find an $\epsilon$ so that $\init_{\epsilon w'+r_i}(f)=\init_{w'}(\init_{r_i}(f))$. As we assumed $c_\beta{\bf x}^\beta$ is not a monomial in $\init_{r_i}(f)$, it is also not a monomial in $\init_{\epsilon w'+r_i}(f)$. But $\epsilon w'+r_i\in C^\circ$ and so $\init_{\epsilon w'+r_i}(f)\in \init_{\epsilon w'+r_i}(J)=\init_<(J)$. This however is a contradiction as we assumed that $c_\beta{\bf x}^\beta$ is the only monomial in $f$ that is contained in $\init_<(J)$.
Therefore, $r_i\cdot\beta=\min_{c_\alpha\not =0}\{r_i\cdot \alpha\}$ for all $i$ and ${\bf r}\cdot \beta=\mu_{\bf r}(f)$.
Assume on the contrary that there exists another monomial $c_\gamma{\bf x}^\gamma$ in $f$ with ${\bf r}\cdot\gamma={\bf r}\cdot\beta$.
Then $c_\gamma{\bf x}^\gamma$ is also a~monomial in $\init_{r_i}(f)$ for all $i$.
Moreover, for $w=r_1+\dots+r_m\in C^\circ$ the initial form $\init_w(f)$ contains both monomials $c_\beta{\bf x}^\beta$ and $c_\gamma{\bf x}^\gamma$.
But as $\init_w(J)=\init_<(J)$ this implies that $c_\gamma{\bf x}^\gamma\in \init_<(J)$, a contradiction.
Hence, $c_\beta{\bf x}^\beta$ is the unique monomial in $f$ with $\mu_{\bf r}(f)={\bf r}\cdot \beta$ and we obtain~\eqref{eq:lifted element} as ${\bf r}\cdot(\alpha-\beta)\not=0$ for all $\alpha\not=\beta$. This completes the proof of the first claim.

To prove the second part, assume that $g\in\mathcal{G}$ has a monomial term $c_\gamma{\bf x}^{\gamma}\in \init_<(J)$ with $c_\gamma{\bf x}^{\gamma}\neq \init_<(g)$.
Then by the definition of the reduced Gr\"obner basis, $c_\gamma{\bf x}^{\gamma}$ is not divisible by~$\init_<(g')$ for any $g'\in \mathcal G\setminus \{g\}$.
Hence, it should be divisible by~$\init_<(g)$, but this cannot happen as $g$ is ${\bf d}$-homogeneous by~\thref{lem:init ell}.
\end{proof}

We extend the monomial term order $<$ on $\mathbb K[x_1,\dots,x_n]$ to a monomial term order $\ll$ on the polynomial ring $\mathbb K[t_1,\dots,t_m,x_1,\dots,x_n]$
in such a way that $t_i\ll x_j$ for all $1\le i\le m$ and $1\le j\le n$, and
\begin{gather}\label{eq:<<}
{\bf x}^\alpha{\bf t}^{\lambda}\ll{\bf x}^\beta{\bf t}^{\mu}\quad\text{if and only if}\quad
(i)\quad {\bf x}^\alpha<{\bf x}^\beta \quad \text{or}\quad
(ii)\quad {\bf x}^\alpha={\bf x}^\beta\quad \text{and}\quad {\bf t}^{\lambda}<_{\rm lex}{\bf t}^{\mu}.
\end{gather}

\begin{Proposition}\thlabel{prop:lifted generators}
The set $\tilde{\mathcal{G}} =\{\tilde g\colon g\in \mathcal G\}$ is a Gr\"obner basis for $ \tilde{J}$ with respect to $\ll$. In~par\-ticular, $\tilde{\mathcal{G}} $ is a generating set for $\tilde{J}$.
\end{Proposition}
\begin{proof}
Since $\tilde{J}$ is $w'$-homogeneous by Lemma~\ref{lem:w'homogeneous}, it is enough to show that for every $w'$-homogeneous polynomial $h\in \tilde{J}$ there exists some $\tilde g_i\in\tilde{\mathcal{G}} $ whose initial term divides
$\init_\ll(h)$.
Let~$h\in \tilde{J}$ be $w'$-homogeneous.
By Lemma~\ref{lem:homogeneous} there exist $f\in J$ and ${\bf u}\in \mathbb{Z}_{\geq 0}^m$ such that $h={\bf t}^{\bf u}\tilde{f}$. We compute
\begin{gather*}
\init_{\ll}(h) = {\bf t}^{\bf u}\init_{\ll}(\tilde f) \overset{\text{Def. }\ll}{=} {\bf t}^{{\bf u}+{\bf v}}\init_<(f)
\overset{\exists \ g\in \mathcal G}{=} {\bf t}^{{\bf u}+{\bf v}}c_\alpha{\bf x}^{\alpha}\init_{<}(g) \overset{\text{\thref{lem:lifts}}}{=} {\bf t}^{{\bf u}+{\bf v}}c_\alpha{\bf x}^{\alpha}\init_{\ll}(\tilde g),
\end{gather*}
where $g$ is a suitable element of the Gr\"obner basis $\mathcal G$ for $J$ and ${\bf v}\in \mathbb Z_{\ge 0}^m$.
As $\tilde g\in \tilde{\mathcal{G}}$, this shows that $\tilde{\mathcal{G}}$ is a Gr\"obner basis for $\tilde{J}$.
Hence, by~\cite[Theorem~2.1.8]{herzog2011monomial} $\tilde{\mathcal{G}}$ is a generating set for $\tilde{J}$.
\end{proof}

As a direct consequence of~\thref{prop:change rays} and~\thref{prop:lifted generators} we obtain that the lifted algebra is independent of the choice of vectors $r_1, \dots ,r_m$.
\begin{Corollary}
For ${\bf r}$ and ${\bf r}'$ as in~\thref{prop:change rays} we have that $\tilde{A}_{\bf r} = \tilde{A}_{\bf r '}$.
\end{Corollary}

The algebras $\tilde A\vert_{{\bf t}={\bf a}}$ constitute the fibers of a flat family introduced below.
The following result leads to fibers over different points being isomorphic if they have zero entries in the same positions.

\begin{Proposition}\thlabel{prop:generic fibers}
Consider ${\bf a}=(a_1,\dots,a_m)$ and ${\bf b}=(b_1,\dots,b_m)$ in $\mathbb K^m$ with the property that $a_i=0$ if and only if $b_i=0$.
Then the algebras $\tilde A\vert_{{\bf t}={\bf a}}$ and $\tilde A\vert_{{\bf t}={\bf b}}$ are isomorphic.
\end{Proposition}
\begin{proof}
Let $\varphi$ be the $\mathbb K$-algebra automorphism of $\mathbb K[t_1,\dots,t_m]$ defined by $\varphi(t_i)=\frac{a_i}{b_i}t_i$ if $b_i\not =0$ and $\varphi(t_i)=t_i$ if $b_i=0$.
Then $\varphi({\bf t}-{\bf a})=({\bf t}-{\bf b})$, so the claim follows by~\thref{notation}.
\end{proof}

\begin{Remark}\thlabel{rmk:non primitive rays}
For computational reasons it might be desirable to work with ray generators for $C$ that are not representatives of primitive ray generators for $\overline{C}$.
Let ${\bf r}$ and ${\bf r'}$ be two choices of ray matrices whose rows satisfy $\overline{r'}_j=q_j\overline{r}_j$ with $q_j\in \mathbb Q$ for all $j$.
Then we still have an isomorphism between $\tilde A_{\bf r}\vert_{{\bf t}={\bf a}}$ and $\tilde A_{\bf r'}\vert_{{\bf t}={\bf a}}$ for all ${\bf a}\in \mathbb K^m$.
For the proof it is neces\-sary to extend the polynomial ring $\mathbb K[t_1,\dots,t_m]$ to a ring $\mathbb K\big[t_1^{\mathbb Q},\dots,t_m^{\mathbb Q}\big]$, where the $t_i$'s are allowed to have rational exponents.
The automorphism of $\mathbb K\big[t_1^{\mathbb Q},\dots,t_m^{\mathbb Q}\big]$ defined by $h(t_i)=t_i^{q_i}$ extends to an~auto\-morphism of $\mathbb K\big[t_1^{\mathbb Q},\dots,t_m^{\mathbb Q}\big][x_1,\dots,x_n]$ with the property $h\big(\tilde f_{\bf r}\big)=\tilde f_{\bf r'}$ for all ${\bf d}$-homogeneous polynomials $f\in J$.
The rest follows from~\thref{prop:generic fibers}.
\end{Remark}

Before presenting our main result we explain how the ideal $\tilde J_{\bf r}$ is related to the ideal $\hat J_w$ in~\eqref{eq:flat}.

\begin{Proposition}\thlabel{prop:special fibers}
Consider a face $\tau$ of $C$ spanned by a subset $\{r_{i_1},\dots,r_{i_s}\}$ of $\{r_1,\dots,r_m\}$ and the lineality space $\mathcal L(J)$.
We define ${\bf t}_\tau\in \mathbb K[t]^{m}$ by
\begin{displaymath}
({\bf t}_\tau)_k:= \begin{cases}t& \text{if}\quad k\in\{i_1,\ldots,i_s\},
\\
1 &\text{otherwise,}
\end{cases}
\qquad
\text{for}
\quad k=1,\ldots,m.
\end{displaymath}
Then for $w_\tau=r_{i_1}+\dots +r_{i_s}$ we have $\tilde J\vert_{{\bf t}={\bf t}_\tau}=\hat J_{w_\tau}$.
\end{Proposition}
\begin{proof}
Let $<$ be the monomial term order compatible with $C$. Consider an element $g=\sum_{\alpha \in \mathbb Z^n_{\ge 0}}c_{\alpha}{\bf x}^{\alpha}$ in $\mathcal{G}$ with $\init_<(g)=c_\beta {\bf x}^\beta$. Since the vectors $w_\tau$, $w:=w_\tau-r_{i_s}$ and $r_{i_s}$ are all in $C$, by Remark~\ref{rem:Eisenbud}:
\begin{equation*}
\init_<(\init_{w_\tau}(g))=\init_<(\init_{w}(g))=\init_{<}(\init_{r_{i_s}}(g))=\init_<(g)=c_\beta {\bf x}^\beta.
\end{equation*}
This implies that the initial forms of $\init_{w_\tau}(g)$, $\init_{w}(g)$ and $\init_{r_{i_s}}(g)$ contain 
$c_\beta {\bf x}^\beta$. In other words,
\begin{displaymath}
\min_{c_{\alpha}\not =0}\{w_\tau\cdot \alpha\}=w_\tau\cdot\beta=(w+r_{i_s})\cdot \beta, \qquad \min_{c_{\alpha}\not =0}\{w\cdot \alpha\}=w\cdot\beta,\qquad \min_{c_{\alpha}\not =0}\{r_{i_s}\cdot \alpha\}=r_{i_s}\cdot\beta.
\end{displaymath}
Therefore,
\begin{displaymath}
(w+r_{i_s})\cdot\beta=\min_{c_{\alpha}\not =0}\{(w+r_{i_s})\cdot \alpha\} =\min_{ c_{\alpha}\not =0}\{w\cdot \alpha\} + \min_{c_{\alpha}\not =0} \{r_{i_s}\cdot \alpha\}.
\end{displaymath}
Using the same argument multiple times we obtain that
\begin{displaymath}
\min_{c_{\alpha}\not =0}\Bigg\{ \bigg(\sum_{j=1}^s r_{i_j}\bigg)\cdot \alpha \Bigg\} = \sum_{j=1}^s \min_{c_{\alpha}\not =0}\{r_{i_j}\cdot \alpha\}.
\end{displaymath}
Now the claim follows by comparing the generating sets of $\hat J_{w_\tau}$ and $\tilde J\vert_{{\bf t}={\bf t}_{\tau}}$.
\end{proof}

We are now prepared to present our main theorem:

\begin{Theorem}\thlabel{thm:family}
Let $J$ be a ${\bf d}$-homogeneous ideal in $\mathbb K_{\bf d}[x_1,\dots,x_n]$, $A=\mathbb K_{\bf d}[x_1,\ldots,x_n]/J$, $C$ a~maximal cone in $\GF(J)$ with compatible monomial term order $<$ and ${\bf r}$ an $(m\times n)$-matrix whose rows are representatives of the primitive ray generators of $\overline{C}\subset \mathbb R^n/\mathcal L(J)$. Then:
\begin{itemize}\itemsep=0pt
 \item[$(i)$] The algebra $\tilde A_{\bf r}$ is a free $\mathbb K[t_1,\dots,t_m]$-module with basis $\mathbb B_<$, the standard monomial basis of $A$ with respect to $\init_<(J)$.
 In particular, we have a flat family
 \begin{displaymath}
 \begin{tikzcd}[ampersand replacement=\&,cramped]
 {\rm Proj}(\tilde A) \ar[r,hook]\ar[d,"\pi"] \& \mathbb P({\bf d})\times \mathbb A^m \ar[ld,twoheadrightarrow]\\
 \mathbb A^m. \&
 \end{tikzcd}
 \end{displaymath}
\item[$(ii)$] For every face $\tau$ of $C$ there exists ${\bf a}_\tau\in \mathbb A^m$ such that
 $\pi^{-1}({\bf a}_\tau) = {\rm Proj}(\mathbb K_{\bf d}[x_1,\dots,x_n]/$ $\init_\tau(J))$.
 In particular, generic fibers are isomorphic to ${\rm Proj}(A)$ and there exist special fibers for every proper face $\tau\subset C$.
\end{itemize}
\end{Theorem}
\begin{proof}
$(i)$ We extend $<$ to a monomial term order $\ll$ on $\mathbb K[t_1,\dots,t_m,x_1,\dots,x_n]$ as in~\eqref{eq:<<}.
 Let $\mathbb B_{\ll}$ be the standard monomial basis for $\tilde A_{\bf r}$ induced by $\init_{\ll}\big(\tilde J_{\bf r}\big)$.
 Then $\mathbb B_{\ll}$ is a basis for $\tilde A_{\bf r}$ as a $\mathbb{K}$-vector space and a generating set for $\tilde A_{\bf r}$ as a $\mathbb K[t_1,\dots,t_m]$-module.
 By~\thref{prop:lifted generators} and~\cite[Proposition 1.1.5]{herzog2011monomial} we have
 \begin{gather*}
 \mathbb B_{\ll} = \big\{\bar {\bf t}^\beta\bar {\bf x}^\alpha \colon \init_{\ll}(\tilde g)\not\vert \ {\bf t}^\beta{\bf x}^\alpha \text{ for all }g\in \mathcal G\big\},\qquad
 \mathbb B_{<} = \big\{\bar {\bf x}^\alpha \colon \init_{<}(g)\not\vert \ {\bf x}^\alpha \text{ for all }g\in \mathcal G\big\}.
 \end{gather*}
 Observe that $\mathbb B_{\ll}$ can be reduced to a $\mathbb K[t_1,\dots,t_m]$-basis of $\tilde A_{\bf r}$ by defining
 \begin{displaymath}
 \mathbb B'_{\ll}:= \mathbb B_{\ll} \setminus \big\{\bar{\bf t}^\beta\bar{\bf x}^\alpha \colon \exists \ \bar{\bf t}^{\gamma}\bar{\bf x}^\alpha\in \mathbb B_{\ll} \text{ with } \gamma < \beta\big\},
 \end{displaymath}
 where $\gamma < \beta$ means that $\gamma_i\le \beta_i$ for all $i$ and $\gamma_j<\beta_j$ for some $j$. Note that $\mathbb B_{<}\subseteq \mathbb B'_{\ll}$.
 Now assume there is a monomial $\bar{\bf t}^\beta\bar{\bf x}^\alpha \in \mathbb B'_{\ll}\setminus \mathbb B_<$.
 By~\thref{prop:lifted generators} we have that $\init_{\ll}\big(\tilde J\big)\vert_{{\bf t}={\bf 1}}=\init_<(J)$.
 Hence, $\mathbb B'_{\ll}\vert_{{\bf t}={\bf 1}}=\mathbb B_{\ll}\vert_{{\bf t}={\bf 1}}=\mathbb B_<$.
 In particular, this implies that $\bar{\bf x}^\alpha\in \mathbb B_<$. But as $\mathbb B_<\subset \mathbb B_{\ll}$ this is a contradiction to the definition of $\mathbb B'_{\ll}$.

 For the second claim, let $\mathbb K[t_1,\dots,t_m]_{\bf d}[x_1,\dots,x_n]$ denote the polynomial ring in $x_1,\dots,x_n$ with coefficients in $\mathbb K[t_1,\dots,t_m]$ and grading induced by ${\bf d}$.
 Then by~\thref{lem:lifts} and~\thref{prop:lifted generators} $\tilde J\subseteq \mathbb K[t_1,\dots,t_m]_{\bf d}[x_1,\dots,x_n]$ is ${\bf d}$-homogeneous.
 This yields the embedding ${\rm Proj}\big(\tilde A\big)\hookrightarrow \mathbb P({\bf d})\times \mathbb A^m$.
 The projection $\mathbb P({\bf d})\times \mathbb A^m\twoheadrightarrow \mathbb A^m$ induces the flat morphism $\pi\colon {\rm Proj}(\tilde A)\to \mathbb A^m$ as $\tilde A$ has a $\mathbb K[t_1,\dots,t_m]$-basis by the first claim.

$(ii)$ Every face $\tau$ of $C$ induces a face $\overline{\tau}$ of $\overline{C}$ with primitive ray generators $\overline{r}_{i_1},\dots, \overline{r}_{i_{s}}$ for some $i_1,\dots,i_s\in [m]$.
 We define
 \begin{equation*}
 ({\bf a}_{\tau})_k:=
 \begin{cases}
 0 & \text{if}\quad k\in\{i_1,\dots,i_{s}\},\\
 1 & \text{otherwise},
 \end{cases}
 \qquad
\text{for}
\quad k=1,\ldots,m.
 \end{equation*}
 Then by~\thref{prop:special fibers} we have $\tilde J\vert_{{\bf t}={\bf a}_\tau}=\hat J_{w_\tau}\vert_{t=0}$ which is equal to $\init_{\tau}(J)$ by Remark~\ref{rem:flat}.
\end{proof}

\subsection{Torus equivariant families}\label{sec:toric families}

We explain how the above results are related to Kaveh--Manon's recent work~\cite{KM-toricbundles} on the classification of torus equivariant families.

Consider the lattice $N=\mathbb Z^n$, its dual lattice $M=N^*$ and a fan $\Sigma\subset N\otimes_{\mathbb Z}\mathbb R$.
We write~$X_\Sigma$ for the toric variety associated to $\Sigma$.
Furthermore, we define $\mathcal O_N$ to be the semifield of piecewise linear functions on $N$ and $\mathcal O_{\Sigma}$ the semifield of piecewise linear functions on $\vert\Sigma\vert\cap N$.
For $a,b\in\mathcal O_{\Sigma}$ we have $a\otimes b:=a+b$ and $a\oplus b:=\min\{a,b\}$, where the minimum is taken pointwise.
In particular, $\mathcal O_{\Sigma}$ is partially ordered: $a\ge b$ if $a\oplus b=b$.
A {\it PL-quasivaluation} on an algebra $A$ is a map $\val\colon A\to \mathcal O_{\Sigma}$ that satisfies: $(i)$ $\val(fg)\ge \val(f)+\val(g)$ and $(ii)$ $\val(f+g)\ge \min\{\val(f),\val(g)\}$, where the minimum is taken pointwise in $\mathcal O_{\Sigma}$.
If $(i)$ is an equality, $\val$ is called a {\it PL-valuation}.
If $A$ is a graded algebra $A=\bigoplus_{n\in \mathbb Z} A_n$ then $\val$ is called {\it homogeneous} if it is compatible with the grading, i.e.,~for $f=\sum_{n\in \mathbb Z} c_n g_n$ with $g_n\in A_n$ and coefficients $c_n$ we have $\val(f)=\val(g_k)$ for the smallest $k$ with $c_k\not =0$.

Given a sheaf of algebras $\mathcal A$ on $X_\Sigma$ its {\it relative spectrum} denoted by $\mathbf{Spec}(\mathcal A)$ is the scheme obtained from gluing affine schemes $\operatorname{Spec}(\mathcal A(U_i))$, where $\bigcup_i U_i$ is an open cover of $X_\Sigma$ and $\mathcal A(U_i)$ is the corresponding section of $\mathcal A$.

\begin{Theorem}[{\cite[Theorem~1.1]{KM-toricbundles}}]
Let $A=\bigoplus_{n\ge 0}A_n$ be a positively graded algebra and a~domain over an algebraically closed field of characteristic zero and $\Sigma$ a fan in $N\otimes_{\mathbb Z}\mathbb R$.
Then the following are equivalent information:
\begin{itemize}\itemsep=0pt
 \item[$(i)$] a $T_N$-equivariant flat sheaf $\mathcal A$ of positively graded algebras of finite type over $X_\Sigma$ such that $\mathbf{Spec}(\mathcal A)$ has reduced and irreducible fibers and its generic fibers are isomorphic to $\operatorname{Spec}(A);$
\item[$(ii)$] a homogeneous PL-valuation $\val\colon A\to \mathcal O_{\Sigma}$ with finite Khovanskii basis. $($Given $\val\colon A\to \mathcal O_{\Sigma}$ for every ray $\rho\in \Sigma$ one constructs a quasivaluation $\val_\rho\colon A\to \mathbb Z\cup \{\infty\}$ with associated graded algebra ${\rm gr}_{\rho}(A)$, see~{\rm \cite[Section~4.1]{KM-toricbundles}}. A {\it Khovanskii basis} for $\val$ is a subset $\mathcal B\subset A$ such that for every ray $\rho\in \Sigma$ the image of $\mathcal B$ in ${\rm gr}_{\rho}(A)$ is a set of algebra generators.$)$
\end{itemize}
\end{Theorem}

Consider a presentation of $A$, i.e.,~a surjection ${\rm pr}\colon \mathbb K[x_1,\dots,x_n]\to A$ whose kernel is a~weig\-hted homogeneous prime ideal $J$ and so it induces an isomorphism of graded algebras $A\cong \mathbb K[x_1,\dots,x_n]/J$.
Take a cone $\sigma$ in $\GF(J)$.
Then $\sigma$ defines PL-quasivaluation $\val_{\sigma}\colon A\to \mathcal O_\sigma$ as follows:
let $C$ be a maximal cone in $\GF(J)$ with face $\sigma$ and denote by $\mathbb B_<$ the standard monomial basis associated to $C$ (here $<$ is the compatible monomial term order).
If $\init_\sigma(J)$ is~prime, then~$\val_{\sigma}$ is a PL-valuation.
Then for $\bar{\bf x}^\alpha\in \mathbb B_<$ we have $\val_{\sigma}(\bar{\bf x}^\alpha):= -\cdot \alpha\colon \sigma\to \mathbb Z$, where~$-\cdot -$ is the pairing between the lattices $N$ and $M$, as above.
For an arbitrary element $f\in A$ write $f=\sum_{\bar{\bf x}^\alpha\in \mathbb B_<} c_\alpha \bar{\bf x}^\alpha$.
Then $\val_\sigma(f):=\min_{c_\alpha\not =0}\{\val_\sigma(\bar{\bf x}^\alpha)\}$.

\begin{Definition}[\cite{KM-toricbundles}]\thlabel{def:Rees}
The PL-quasivaluation $\val_\sigma\colon A\to\mathcal O_\sigma$ defines a {\it filtration} on $A$ with filtered pieces
$F_m(\val_\sigma):=\{f\in A\colon \val_\sigma(f)\ge -\cdot m \}$ for $m\in M$.
The {\it Rees algebra} of the PL-va\-lu\-ation $\val_\sigma$ is $\mathcal R_\sigma(A):=\bigoplus_{m\in M}F_m(\val_\sigma)t^m$, where we think of $t^m$ as a character of the torus $T_N=N\otimes_{\mathbb Z}\mathbb C^*$.
\end{Definition}

More generally, a PL-quasivaluation can be obtained from a collection of equidimensional cones $\sigma_1,\dots,\sigma_k$ in $\trop(J)$ which are faces of the same maximal cone $C$ in $\GF(J)$.
Let $\Sigma\subset \trop(J)$ be the subfan whose maximal cones are $\sigma_1,\dots,\sigma_k$ and $\val_{\Sigma}\colon A \to \mathcal O_N$ be the corresponding PL-valuation.
By Kaveh--Manon's classification of toric families this yields a flat family
\begin{equation}\label{eq:KM family}
\psi\colon\ \mathbf{Spec}(\mathcal R_\Sigma(A))\to X_\Sigma.
\end{equation}
Here $\mathcal R_\Sigma(A)$ is the $T_N$-equivariant flat sheaf of Rees algebras on $X_\Sigma$.
The scheme ${\bf Spec}(\mathcal{R}_{\Sigma}(A))$ is glued from $\operatorname{Spec}(\mathcal{R}_{\sigma})$ for $\sigma \in \Sigma$.
For $p\in T_N\subset X_{\Sigma}$ we have $\psi^{-1}(p)\cong \operatorname{Spec}(A)$.
Moreover, $\psi_\Sigma$ has a special fiber over every torus fixed point of $X$. More precisely, let $p_{\sigma_i}$ be the torus fixed point corresponding to the (maximal) cone $\sigma_i\in\Sigma$. Then
\begin{equation}\label{eq:fibers psi}
\psi^{-1}(p_{\sigma_i})\cong \operatorname{Spec}\big(\mathbb K[x_1,\dots,x_n]/\init_{\sigma_i}(J)\big).
\end{equation}

Note that we can apply the construction of~\thref{def:lift} to the ideal $J$ and the maximal cone $C\in \GF(J)$.
In what follows we explore the relation between the flat families from~\eqref{eq:KM family} and~\thref{thm:family}.
Before stating our results (\thref{thm:KM max cone} and~\thref{thm:KM vs grobner}), we recall necessary background from~\cite{KM-toricbundles}.
We fix a maximal cone $C$ in $\GF(J)$, denote by $\mathcal G$ the associated reduced Gr\"obner basis and let $\mathbb B_<$ be the standard monomial basis for the compatible monomial term order $<$.

\begin{Definition}[\cite{KM-toricbundles}]
The PL-quasivaluation $\mathfrak{w}_C\colon \mathbb K[x_1,\dots,x_n]\to \mathcal O_C$ associated to $C$ is~defi\-ned by $\mathfrak{w}_C({\bf x}^\alpha)= -\cdot \alpha\colon C\to \mathbb Z$.
For a polynomial $f=\sum c_\alpha {\bf x}^\alpha$ we have $\mathfrak{w}_C(f):=\min_{c_\alpha\not =0}\{ -\cdot \alpha\}$.
The associated Rees algebra is $\mathcal R_C(\mathbb K[x_1,\dots,x_n])=\bigoplus_{m\in M} F_m(\mathfrak{w}_C)t^m$, where
$F_m(\mathfrak{w}_C)$ is defined analogous to $F_m(\val_\sigma)$ in~\thref{def:Rees}.
\end{Definition}

The PL-quasivaluation $\val_C\colon A\to \mathcal O_C$ is obtained as a pushforward of the PL-quasivaluation $\mathfrak{w}_C\colon \mathbb K[x_1,\dots,x_n]\to \mathcal O_C$ along the morphism ${\rm pr}\colon \mathbb K[x_1,\dots,x_n]\to A$.
An {\it adapted basis} for~$\val_C$ is a vector space basis $\mathbb B$ for $A$ such that $\mathbb B\cap F_m(\val_\sigma)$ is a vector space basis of $F_m(\val_\sigma)$ for every $m\in M$.
By~\cite{KM-toricbundles} the standard monomial basis $\mathbb B_<$ is an adapted basis for $\val_C$.
Similarly, the monomial basis of $\mathbb K[x_1,\dots,x_n]$ is adapted to $\mathfrak{w}_C$.
Hence, $\mathcal R_C(A)$ is a free $\mathbb K[S_C]$-algebra with basis $\mathbb B_<$ and $\mathcal R_C(\mathbb K[x_1,\dots,x_n])$ is a free $\mathbb K[S_C]$-algebra whose basis is the monomial basis, where $S_C:=-C^\vee \cap \mathbb Z^n$ due to our $\min$-convention.
In particular: 
\begin{equation}\label{eq:Rees poly}
\mathcal R_C\big(\mathbb K[x_1,\dots,x_n]\big)\cong \mathbb K[S_C][x_1,\dots,x_n].
\end{equation}

Manon explained the proof of the following proposition to us. It follows from results in~\cite{KM-toricbundles}.

\begin{Proposition}\thlabel{prop:Rees quotient}
For any maximal cone $C\in\GF(J)$, the Rees algebra $\mathcal R_C(A)$ has an explicit presentation of form $\mathbb K[S_C][x_1,\dots,x_n]/\hat J$.
More precisely, the ideal $\hat J$ is generated by homogenizations of elements in the reduced Gr\"obner basis $\mathcal G$:
let $g={\bf x}^\gamma+\sum c_\alpha{\bf x}^\alpha\in \mathcal G$ and ${\bf x}^\gamma=\init_C(g)$, then the {\it homogenization} of $g$ is
\begin{equation}\label{eq:homog}
\hat g= {\bf x}^\gamma+\sum c_\alpha {\bf x}^\alpha t^\beta\qquad \text{with}\quad \gamma-\alpha=\beta\in S_C.
\end{equation}
\end{Proposition}

\begin{proof}
Consider $f\in \mathbb K[x_1,\dots,x_n]$.
It can be written as $f=\sum_{{\bf x}^\alpha\in \mathbb B_<} c_\alpha {\bf x}^\alpha+\sum_{{\bf x}^\beta\not\in \mathbb B_<} c_\beta {\bf x}^\beta$.
In~par\-ti\-cular, if $\mathfrak w_C(f)\ge -\cdot m$ in $\mathcal O_C$, then $\val_C(\text{pr}(f))\ge -\cdot m$.
So the morphism $\text{pr}\colon \mathbb K[x_1,\dots,x_n]$ $\to A$ induces a map $\text{pr}_m\colon F_m(\mathfrak w_C)\to F_m(\val_C)$ for all $m$.
Further, every element in $\mathbb B_<\cap F_m(\val_C)$ is the image of a monomial in $F_m(\mathfrak w_C)$.
Hence, $\text{pr}_m$ is surjective and we obtain:
\begin{displaymath}
\widehat{\text{pr}}:=\bigoplus_{m\in (\mathbb Z^n)^*} \text{pr}_m\colon\
\mathcal R_C(\mathbb K[x_1,\dots,x_n])\to \mathcal R_C(A).
\end{displaymath}
By~\eqref{eq:Rees poly} this implies the first claim.
Now let $\hat J$ be the kernel of $\widehat{\text{pr}}$.
Then $\hat J=\bigoplus_{m\in M} \hat J_m $, where $\hat J_m=\ker(\text{pr}_m)\subset F_m(\mathfrak w_C)$.
Now consider an element $g\in \mathcal G$ of form ${\bf x}^\gamma+\sum c_\alpha{\bf x}^\alpha$ with ${\bf x}^\gamma=\init_C(g)$.
We want to show that the element $\hat g$ defined in~\eqref{eq:homog} lies in $\hat J$ for all $g\in \mathcal G$.
Recall that by the proof of~\thref{prop:special fibers}, $\gamma$ is such that $w\cdot \gamma \le w\cdot \alpha $ for all $w\in C$ and all $\alpha$ with $c_\alpha\not =0$ in $g$.
Hence, $ w\cdot (\gamma-\alpha) \le 0$ and so $\beta=\gamma-\alpha\in -C^\vee$.
In particular, this implies that $\hat g\in \hat J_\gamma\subset F_\gamma(\mathfrak w_\sigma)$.

Lastly, we need to show that $\hat J$ is generated by $\{\hat g\colon g\in \mathcal G\}$.
Consider $h\in \hat J$.
As $\hat J=\bigoplus_{m\in M} \hat J_m$, we may assume that $h\in \hat J_m$.
Further, as $\hat J_m\subset F_m(\mathfrak w_C)\subset \mathbb K[x_1,\dots,x_n]$ we can think of $h$ as an element of $\mathbb K[x_1,\dots,x_n]$ as well as an element of $\mathcal R_C(\mathbb K[x_1,\dots,x_n])$.
Then $\init_<(h)={\bf x}^\mu \init_<(g)$ for some $g\in \mathcal G$ of form $g={\bf x}^\gamma+\sum c_\alpha{\bf x}^\alpha$ with ${\bf x}^\gamma=\init_<(g)$.
In particular, similar to the above the exponent of $\init_<({\bf x}^\mu g)={\bf x}^{\mu+\gamma}$ has the property that $w\cdot (\mu+\gamma) \le w\cdot (\mu+\alpha)$ for all $w\in C$ and all $\alpha$ with $c_\alpha\not =0$ in $g$.
So, $h-{\bf x}^{\mu} g\in F_m(\mathfrak w_C)$. 
In particular, the division algorithm with respect to $\mathcal G$ takes place inside $F_m(\mathfrak w_C)$ and yields an expression of $h$ in terms of $\{\hat g\colon g\in \mathcal G \}$.
\end{proof}

We now recall the Cox construction~\cite{Cox95} for toric varieties and refer to~\cite[Section~5]{CLS} for more details.
Let $\Sigma= C\cap \trop(J)$ and
$l$ denote the dimension of $\mathcal L(J)$.
Fix an $(m\times n)$-mat\-rix~${\bf r}$ whose rows $r_1,\dots,r_{m}$ are representatives of the primitive ray generators of $\overline{C}\subset \mathbb R^n/\mathcal L(J)$.
The rays of $\Sigma$, denoted by $\Sigma(1)$, form a subset of $\{{r}_1,\dots,{r}_{m}\}$.
Recall the definition of the lifted algebra $\tilde A_{\bf r}$ from~\eqref{eq:lifted algebra}.
We can apply the {\it Cox construction} in two ways, to $X_C$ and to $X_\Sigma$:
First, let $X_{C}$ be the affine toric variety associated to $C$.
Recall that a {\it quasitorus} is a product of a torus with a finite abelian group.
Then $X_{C}$ is isomorphic to the almost geometric quotient $\mathbb A^m/\!/G$ (see, e.g.,~\cite[Theorem~5.1.11]{CLS}), where the group $G$ by~\cite[Lemma~5.1.1(b)]{CLS} is the quasitorus
\begin{equation}\label{eq:quasitorus}
G=\{(t_1,\dots,t_m)\in(\mathbb C^*)^m\colon t_1^{r_1\cdot e_i}\cdots t_m^{r_m\cdot e_i}=1\, \forall \, i\in [n]\}.
\end{equation}
One can easily check that if $p,q\in \mathbb A^m$ lie in the same $G$-orbit, then $\pi^{-1}(p)\cong \pi^{-1}(q)$ by~\thref{prop:generic fibers}.
 Hence, the flat family $\pi\colon \operatorname{Spec}(\tilde A_{\bf r})\to \mathbb A^m$ induces the commutative diagram
 \begin{equation}\label{eq:Cox C}
 \begin{tikzcd}[ampersand replacement=\&,cramped]
 \& \operatorname{Spec}(\tilde A_{\bf r})\ar[d,"\pi"] \ar[dl,"p_{C}\circ \pi"'] \\
 X_{C} \& \mathbb A^m. \ar[l,"p_{C}"]
 \end{tikzcd}
 \end{equation}
 If $p,q\in X_{C}$ lie in the same torus orbit, then $(p_{C}\circ \pi)^{-1}(p)\cong (p_{C}\circ \pi)^{-1}(q)$ by~\thref{prop:generic fibers}.
 Note that $p_C\colon \mathbb A^m\to X_C$ is indeed a morphism. Namely, it is the {\it universal torsor} for $X_C$.

Similarly, we can apply the Cox construction to the toric variety $X_\Sigma$. For simplicity we assume that $\Sigma(1)=\{{r}_1,\dots,{r}_{m}\}$.
 In this case the construction has two steps: first we remove the locus of $\mathbb A^m$ that does not correspond to torus orbits in $X_{\Sigma}$, denoted by $Z(\Sigma)$.
 Recall that $\mathcal C\subset \{{r}_1,\dots,{r}_m\}$ is a {\it primitive collection} for $\Sigma$, if $(i)$ $\mathcal C\not\subset \sigma(1)$ for all $\sigma\in \Sigma$, and $(ii)$ for every $\mathcal C'\subsetneq \mathcal C$ there exists $\sigma\in \Sigma$ with $\mathcal C'\subset \sigma(1)$.
 Then by~\cite[Proposition 5.1.6]{CLS}
 \begin{equation}\label{eq:Z(Sigma)}
 Z(\Sigma)=\bigcup_{\mathcal C \text{ primitive collection }} V(t_i\colon r_i\in \mathcal C) \subset \mathbb A^m.
 \end{equation}
 Now $X_{\Sigma}$ is isomorphic to the almost geometric quotient $(\mathbb A^m-Z(\Sigma))/\!/G$, where $G$ is as in~\eqref{eq:quasitorus} (see, e.g.,~\cite[Theorem~5.1.11]{CLS}).
 In contrast to the morphism $p_C\colon \mathbb A^m\to X_C$ from~\eqref{eq:Cox C}, we~obtain a rational map $p_\Sigma\colon \mathbb A^m\dashrightarrow X_\Sigma$ as $p_\Sigma$ is only defined on $\mathcal T_\Sigma:=\mathbb A^m-Z(\Sigma)$.
 Moreover, $p_\Sigma\colon \mathcal T_\Sigma\to X_\Sigma$ is the universal torsor of $X_\Sigma$.
 Similarly to the first case, this gives the diagram:
 \begin{equation}\label{eq:Cox Sigma}
 \begin{tikzcd}[ampersand replacement=\&,cramped]
 \& \pi^{-1}(\mathcal T_\Sigma)\ar[d,"\pi\vert_{\pi^{-1}(\mathcal T_\Sigma)}"] \ar[dl,"p_\Sigma\circ \pi"']\ar[r,hook] \& \operatorname{Spec}(\tilde A_{\bf r})\ar[d,"\pi"] \\
 X_{\Sigma} \& \mathcal T_\Sigma \ar[l,"p_\Sigma"]\ar[r,hook]\& \mathbb A^m
 \end{tikzcd}
 \end{equation}
 which is commutative. The fibers of the family defined by $p_\Sigma\circ \pi\vert_{\pi^{-1}(\mathcal T_\Sigma)}$ coincide with the fibers of $\psi$ defined in~\eqref{eq:KM family}: $(p_\Sigma\circ \pi)^{-1}(x)\cong \psi^{-1}(x)$ for every $x\in X_\Sigma$, see~\eqref{eq:fibers psi}.

In the following, we explain how the flat family of~\thref{thm:family} is related to Kaveh--Manon's flat family associated to a maximal cone in the Gr\"obner fan.

\begin{Theorem}\thlabel{thm:KM max cone}
Let $J\subseteq\mathbb K_{\bf d}[{x_1,\dots,x_n}]$ be a {\bf d}-homogeneous ideal and $C$ a maximal cone in~$\GF(J)$.
Let ${\bf r}$ be an $(m\times n)$-matrix whose rows are representatives of primitive ray generators for $\overline C$. 
Then the morphism $\pi\colon \operatorname{Spec}\big(\tilde A_{\bf r}\big)\to \mathbb A^m$ fits into a pull-back diagram
 \begin{equation*}
 \begin{tikzcd}[ampersand replacement=\&,cramped]
 \operatorname{Spec}(\mathcal R_C)\ar[d,"\psi"'] \&
 \operatorname{Spec}\big(\tilde A_{\bf r}\big)\ar[d,"\pi"] \ar[l]\\
 X_{C}\& \mathbb A^m. \ar[l,"p_C"]
 \end{tikzcd}
 \end{equation*}
Here $p_C\colon \mathbb A^m\to X_C$ is the universal torsor of $X_C$ obtained from the Cox construction as in~\eqref{eq:Cox C}.
\end{Theorem}
\begin{proof}
We prove equivalently that the corresponding diagram between the algebras is a push-out.
Hence, it suffices to show that
\begin{equation*}
\mathcal R_C\otimes_{\mathbb K[S_C]} \mathbb K[t_1,\dots,t_m]\cong \tilde A.
\end{equation*}
Note that the map $p_C\colon \mathbb A^m \to X_C$ corresponds to $p_C^\#\colon \mathbb K[S_C] {\hookrightarrow} \mathbb K[t_1,\dots,t_m]$.
In particular, for a character $t^{-r_i^*}\in \mathbb K[S_C]$ with $r_i^*\in C^\vee\subset M$ the dual of $r_i\in C\subset N$ we have $p_C^\#(t^{-r_i^*})$ $=t_i\in \mathbb K[t_1,\dots,t_m]$.
Then under the extension of scalars (i.e.,~applying the functor $-\otimes_{\mathbb K[S_C]}\mathbb K[t_1,\dots,t_m]$) the homogenization $\hat g$ of an element $g\in \mathcal G$ is sent to the lift $\tilde g$.
This can be seen as follows:
by~\thref{lem:lifts} we have
$\tilde g={\bf x}^\gamma + \sum c_\alpha{\bf x}^\alpha{\bf t}^{{\bf r}\cdot(\alpha-\gamma)}$.
By~\thref{lem:init ell} for every $\ell\in \mathcal L(J)$ we have $\init_{\ell}(g)=g$.
In particular, $\ell\cdot \alpha=\ell\cdot \gamma$ and so $\alpha-\gamma\in (\mathcal L(J)\cap N)^\perp$.
Let $\{\ell_1,\dots,\ell_{k}\}$ be a basis for $\mathcal L(J)\cap N$.
Then $\{r_1,\dots,r_m,\ell_1,\dots,\ell_{k}\}$ spans $N$ and $\{r_1^*,\dots,r_m^*\}$ spans $(\mathcal L(J)\cap N)^\perp\subset M$.
So every $\alpha-\gamma$ has an expression of form $\alpha-\gamma=\sum c_ir_i^*$.
We compute
\begin{equation}\label{eq:t's}
t^{-(\alpha-\gamma)}=t^{-c_1r_1^*}\cdots t^{-c_mr_m^*} \overset{p_C^\#}{\longmapsto} t_1^{c_1}\cdots t_m^{c_m}={\bf t}^{{\bf r}\cdot(\alpha-\gamma)}.
\end{equation}
Note that $\gamma$ has the property that $r_i\cdot\gamma
=\min_{c_\alpha \not=0}\{r_i\cdot\alpha\}$.
So, $r_i\cdot(\alpha-\gamma)\ge 0$ and therefore $\alpha-\gamma\in C^\vee$.
Hence, we have
\begin{gather*}
\mathcal R_C\otimes_{\mathbb K[S_C]} \mathbb K[t_1,\dots,t_m]
\overset{\text{{Proposition}~\ref{prop:Rees quotient}}}{\cong}
\big(\mathbb K[S_C][x_1,\dots,x_n]/\hat J \big)\otimes_{\mathbb K[S_C]} \mathbb K[t_1,\dots,t_m]
\\ \hphantom{\mathcal R_C\otimes_{\mathbb K[S_C]} \mathbb K[t_1,\dots,t_m]Pro}
{}\overset{\eqref{eq:t's}}{\cong}
\mathbb K[t_1,\dots,t_m][x_1,\dots,x_n]/\tilde J=\tilde A.
\tag*{\qed}
\end{gather*}
\renewcommand{\qed}{}
\end{proof}

\begin{Corollary}\thlabel{thm:KM vs grobner}
Let $J\subseteq\mathbb K[{x_1,\dots,x_n}]$ be a weighted homogeneous ideal as in Definition~$\ref{def:d homog}$ and $C$ a maximal cone in $\GF(J)$.
Let ${\bf r}$ be the matrix whose rows are representatives of primitive ray generators for $\overline C$ 
and denote $\tilde A_{\bf r}$ by $\tilde A$.
Let $\Sigma$ be the intersection of $C$ with $\trop(J)$.
Then the restriction of morphism $\pi\colon \operatorname{Spec}\big(\tilde A\big)\to \mathbb A^m$ to $\pi^{-1}(\mathcal T_\Sigma)=\operatorname{Spec}\big(\tilde A\big)-\pi^{-1}(Z(\Sigma))$ fits into a~pull-back diagram
 \begin{equation*}
 \begin{tikzcd}[ampersand replacement=\&,cramped]
 \mathbf{Spec}(\mathcal R_\Sigma)\ar[d,"\psi"'] \&
 \pi^{-1}(\mathcal T_\Sigma)\ar[d,"\pi\vert_{\pi^{-1}(\mathcal T_\Sigma)}"] \ar[l]\\
 X_{\Sigma}\& \mathcal T_\Sigma. \ar[l,"p_\Sigma"]
 \end{tikzcd}
 \end{equation*}
Here $p_\Sigma\colon \mathcal T_\Sigma\to X_\Sigma$ is the universal torsor of $X_\Sigma$ obtained from the Cox construction as in~\eqref{eq:Cox Sigma}.
\end{Corollary}
\begin{proof}
The scheme $X_\Sigma$ and the sheaf $\mathcal R_\Sigma$ are defined locally for affine pieces $U_\sigma=\operatorname{Spec}(\mathbb K[S_\sigma])$ $\subset X_\Sigma$ for $\sigma\in \Sigma$.
Algebraically, for every $\sigma\in \Sigma$ the pull-back of $\psi$ and $p_\Sigma$ corresponds to the following push-out diagram:
 \begin{equation*}
 \begin{tikzcd}[ampersand replacement=\&,cramped]
 \mathcal R_\sigma\ar[r] \&
 \mathcal R_\sigma\otimes_{\mathbb K[S_\sigma]}\mathbb K[t_1,\dots,t_m]_{\langle t_i\colon r_i\not\in \sigma\rangle } \\
 \mathbb K[S_\sigma] \ar[r]\ar[u]\& \mathbb K[t_1,\dots,t_m]_{\langle t_i\colon r_i\not\in \sigma\rangle . } \ar[u]
 \end{tikzcd}
 \end{equation*}
Observe that by~\eqref{eq:Z(Sigma)} the localization $\mathbb K[t_1,\dots,t_m]_{\langle t_i\colon r_i\not\in \sigma\rangle }$ corresponds to an affine piece in~$\mathbb A^m$ that does not intersect $Z(\Sigma)$.
So it is in fact an affine piece of $\mathcal T_\Sigma$.
We have to show that
\begin{displaymath}
\mathcal R_\sigma\otimes_{\mathbb K[S_\sigma]}\mathbb K[t_1,\dots,t_m]_{\langle t_i\colon r_i\not\in \sigma\rangle } \cong \tilde A_{\langle t_i\colon r_i\not\in \sigma\rangle}.
\end{displaymath}
As $\sigma$ is a face of the maximal cone $C$ by~\cite[Proposition 3.13]{KM-toricbundles} we have $\mathcal R_\sigma\cong \mathcal R_C\otimes_{\mathbb K[S_C]}\mathbb K[S_\sigma]$.
Using~\thref{thm:KM max cone} we compute
\begin{gather*}
\mathcal R_\sigma\otimes_{\mathbb K[S_\sigma]}\mathbb K[t_1,\dots,t_m]_{\langle t_i\colon r_i\not\in \sigma\rangle}
\cong
\big(\mathcal R_C\otimes_{\mathbb K[S_C]}\mathbb K[S_\sigma]\big)\otimes_{\mathbb K[S_\sigma]}\mathbb K[t_1,\dots,t_m]_{\langle t_i\colon r_i\not\in \sigma\rangle}
\\ \hphantom{\mathcal R_\sigma\otimes_{\mathbb K[S_\sigma]}\mathbb K[t_1,\dots,t_m]_{\langle t_i\colon r_i\not\in \sigma\rangle}}
{}\cong\mathcal R_C\otimes_{\mathbb K[S_C]}\mathbb K[t_1,\dots,t_m]_{\langle t_i\colon r_i\not\in \sigma\rangle}\cong \tilde A_{\langle t_i\colon r_i\not\in \sigma\rangle}.\tag*{\qed}
\end{gather*}
\renewcommand{\qed}{}
\end{proof}

\section{Grassmannians and cluster algebras}\label{sec:Gr}

We now apply the machinery developed in Section~\ref{sec:flat families} to the Grassmannians $\Gr(2,\mathbb C^n)$ and $\Gr\big(3,\mathbb C^6\big)$.
Our aim is to make this paper as self-contained as possible and accessible to readers of broad mathematical background.
Therefore, we recall background on tropical Grassmannians below, on cluster algebras and the cluster structure of $\Gr(2,\mathbb C^n)$ in Section~\ref{sec:cluster} and then in Section~\ref{sec: toric deg} we translate between toric degenerations obtained from the tropicalizations and from the cluster structure.
In Section~\ref{sec:monomial_degeneration}, we turn to the application of the flat families introduced in Section~\ref{sec:flat families} and show how to recover the cluster algebra with universal coefficients in this case.
Similarly, in Section~\ref{sec:Gr(3,6)} we treat the case of $\Gr\big(3,\mathbb C^6\big)$.

We first fix our notation for $\Gr(2,\mathbb C^n)$.
Denote by $S$ the polynomial ring in Pl\"ucker variables $\mathbb C[p_{ij}\colon 1\leq i<j\le n]$.
Here, Pl\"ucker variables are associated with arcs in an $n$-gon whose vertices are labeled by $[n]$ in clockwise order.
Hence, we think of $[n]:=\{1,\ldots,n\}$ as a cyclically ordered set. In particular, $i<j<k<l$ indicates that in clockwise order starting at the vertex $i$ the other vertices appear in order $j$, $k$, $l$.
Note that in general we might not have $1\le i<j<k<l\le n$.
We set the convention $p_{ij}:=-p_{ji}$.
In $S$ we define for every $i<j<k<l$ the {\it Pl\"ucker relation}:
\begin{displaymath}
R_{ijkl}:=p_{ij}p_{kl}-p_{ik}p_{jl}+p_{il}p_{jk}\in S.
\end{displaymath}
The ideal $I_{2,n}\subset S$ generated by all Pl\"ucker relations $R_{ijkl}$ is called the {\it Pl\"ucker ideal}. The~quo\-tient $A_{2,n}:=S/I_{2,n}$ is the {\it Pl\"ucker algebra} which is the homogeneous coordinate ring of the Grassmannian $\Gr(2,\mathbb C^n)$ with respect to its Pl\"ucker embedding into $\mathbb P^{\binom{n}{2}-1}$.
The cosets of~Pl\"u\-cker variables in the quotient are denoted by $\bar p_{ij}\in A_{2,n}$ and are called {\it Pl\"ucker coordinates}.
Denote by $d:=2(n-2)$ the dimension of $\Gr(2,\mathbb C^n)$, so $d+1$ is the Krull-dimension of $A_{2,n}$.

In $S$, we fix the notation ${\bf p}^m$ with $m=(m_{ij})_{ij}\in \mathbb Z_{\ge 0}^{\binom{n}{2}}$ denoting the monomial $\prod_{1\le i<j\le n} p_{ij}^{m_{ij}}$.
The maximal cones in the Gr\"obner fan $\GF(I_{2,n})$ are those of full dimension in the ambient space~$\mathbb R^{\binom{n}{2}}$ with associated monomial initial ideals.
The lineality space $\mathcal L_{2,n}:=\mathcal L(I_{2,n})\subset \GF(I_{2,n})$ is $n$-dimensional and spanned by the elements $\ell_s$ for $s\in [n]$ which are defined by
\begin{displaymath}
(\ell_s)_{ij}:=\begin{cases}
1& \text{if}\quad s=i\quad \text{or}\quad s=j, \\ 0& \text{otherwise}.
\end{cases}
\end{displaymath}
Note that ${\bf 1}:=(1,\dots,1)\in \mathbb R^{\binom{n}{2}}$ lies in $\mathcal L_{2,n}$.
The Gr\"obner fan of the Pl\"ucker ideal has a subfan that is particularly interesting when studying toric degenerations of $\Gr(2,\mathbb C^n)$.

\begin{Definition}\thlabel{def:tropGr}
The {\it tropical Grassmannian} is a simplicial $(d+1)$-dimensional subfan of~$\GF(I_{2,n})$ defined as $\trop(I_{2,n})$.
\end{Definition}

Maximal cones in $\trop(I_{2,n})$ are called {\it tropical maximal cones}. They are in correspondence with trivalent trees with $n$ leaves as shown by Speyer and Sturmfels in~\cite{SS04}.
To a trivalent tree~$\Upsilon$ one associates a weight vector in the relative interior of the corresponding cone of~$\trop(I_{2,n})$ by 
\begin{gather}\label{eq:tree wt}
 \underline{w}^{\Upsilon}(p_{ij}):=-\#\{\text{edges on the unique path } i\to j \text{ in }\Upsilon\}.
\end{gather}
To simplify notation we denote $\init_{\Upsilon}(I):=\init_{\underline{w}^{\Upsilon}}(I)$.

\begin{Theorem}[\cite{SS04}]\thlabel{thm:SS}
For every trivalent tree $\Upsilon$ with $n$ leaves the ideal $\init_{{\Upsilon}}(I)$ is {\it toric}, i.e.,~bino\-mial and prime.
Moreover, it is generated by $\init_{{\Upsilon}}(R_{ijkl})$ for all $1\le i<j<k<l\le n$.
\end{Theorem}

\begin{Definition}\thlabel{def:trop+Gr}
The {\it totally positive tropical} Grassmannian is $\trop^+(I_{2,n})$.
\end{Definition}

In~\cite[Section~5]{SW05}, it was shown that
maximal cones in $\trop^+(I_{2,n})$ are in bijection with {\it planar trivalent trees}, i.e.,~trivalent trees whose vertices are labeled $1,\dots,n$ in a clockwise manner.

\subsection{Preliminaries on cluster algebras}\label{sec:cluster}

The Pl\"ucker algebra $A_{2,n}$ is in fact a cluster algebra (see Definition~\ref{def:cluster algebra}) whose cluster structure is closely related to the combinatorics governing the tropical Grassmannian.
To exhibit this con\-nection, we recall some facts about skew-symmetric cluster algebras following Fomin and Zelevinsky~\cite{FZ02,FZ_clustersIV}.
Afterwards, we define finite type cluster algebras with frozen variables and~universal coefficients.

Let $m,f\in \mathbb N$ be positive integers and $\mathcal{F} =\mathbb{C}(u_1, \dots ,u_{m+f})$ be the field of rational functions on $m+f$ variables.
Here $m$ stands for mutable and $f$ for frozen.
A {\it seed} in $\mathcal{F}$ is a pair $\big(\widetilde{{\bf x}},\widetilde{B}\big)$, where $\widetilde{{\bf x}}=\{x_1,\dots, x_{m+f}\}$ is a free generating set of $\mathcal F$ and $\widetilde{B}=(b_{ij})$ is an $(m+f)\times m$ rectangular matrix with the following property: the top square submatrix of $\widetilde{B}$ (that is the $m\times m$ submatrix of $\widetilde{B}$ formed by its first $m$ rows) is skew-symmetric.
We call $\widetilde{B} $ an {\it extended exchange matrix} and refer to its top square submatrix $B$ as an {\it exchange matrix}.
We call ${\bf x}=\{x_1, \dots , x_m \}$ a {\it cluster}, $\{x_{m+1}, \dots , x_{m+f} \}$ the set of {\it frozen variables} and $\widetilde{{\bf x}}$ an {\it extended cluster}.

Given $k\in \{1, \dots , m \}$, the {\it mutation in direction $k$} of a seed $\big(\widetilde{{\bf x}}, \widetilde{B}\big)$ is the new seed $\mu_k\big(\widetilde{{\bf x}}, \widetilde{B}\big)=\big(\widetilde{{\bf x}}', \widetilde{B}'\big)$, defined as follows:
\begin{itemize}\itemsep=0pt
 \item The extended cluster $\widetilde{{\bf x}}'=\widetilde{{\bf x}}\setminus\{ x_k \} \cup \{ x'_k \} $, where
 \begin{equation}
 \label{eq:exchange_rel}
 x_k x'_k= \prod_{i\colon b_{ik}>0}x^{b_{ik}}_i + \prod_{i\colon b_{ik}<0}x^{-b_{ik}}_i.
 \end{equation}
 We call an expression of the form~\eqref{eq:exchange_rel} an {\it exchange relation} and the monomial $x_kx_k'$ an~{\it exchange monomial}.
 \item The extended exchange matrix $\widetilde{B}'=(b'_{ij})$ is defined by the following formula
 \begin{displaymath}
 b'_{ij}:= \begin{cases} -b_{ij} \qquad\text{if}\quad i=k\quad \text{or}\quad j=k,\\
 b_{ij} +\mathop{\rm sgn}(b_{ik})\mathop{\rm max}(b_{ik}b_{kj},0) \qquad\text{else},
 \end{cases}
 \end{displaymath}
 where
\[
\mathop{\rm sgn}(b_{ik}):= \begin{cases}
-1 &\text{if}\quad b_{ik}<0,\\
0 &\text{if}\quad b_{ik}=0,\\
1 &\text{if}\quad b_{ik}>0.\end{cases}
\]
\end{itemize}

\begin{Definition}\label{def:cluster algebra}
The {\it cluster algebra} $A\big(\widetilde{{\bf x}},\widetilde{B}\big)$ associated to a seed $\big(\widetilde{{\bf x}},\widetilde{B}\big) $ is the $\mathbb C$-subalgebra of~$\mathcal F$ generated by the elements of \emph{all} the extended clusters that can be obtained from the {\it initial seed} $\big(\widetilde{{\bf x}},\widetilde{B}\big)$ by a finite sequence of mutations.
The elements of the clusters constructed in this way are called the {\it cluster variables}. If $f\neq 0$ we say that the cluster algebra $A\big(\widetilde{{\bf x}},\widetilde{B}\big)$ has frozen variables, namely, $x_{m+1}, \dots , x_{m+f}$.
\end{Definition}

\begin{Remark}
It can be easily verified that, up to a canonical isomorphism, $A\big(\widetilde{{\bf x}},\widetilde{B}\big)$ is independent of $\widetilde{{\bf x}}$. Therefore, we write $A\big(\widetilde{B}\big)$ instead of $A\big(\widetilde{{\bf x}},\widetilde{B}\big)$.
\end{Remark}

\begin{Remark}
Fomin and Zelevinsky usually define cluster algebras over $\mathbb Q$. We recover the construction described above applying tensor product with $\mathbb C$ to Fomin and Zelevinsky's construction.
\end{Remark}

\begin{Theorem}[{\cite[Laurent phenomenon]{FZ_clustersIV}}]
\thlabel{thm:SLaurent}
The cluster algebra $A\big(\widetilde{B}\big)$ is a $\mathbb C[x_{m+1},\dots , x_{m+f}]$-subalgebra of $\mathbb C\big[x_1^{\pm 1}, \dots , x_m^{\pm 1},x_{m+1},\dots , x_{m+f}\big]$.
\end{Theorem}

\begin{Definition}\thlabel{def:prin}
The cluster algebra with {\it principal coefficients} associated to an $ m \times m$ skew-symmetric matrix $B$ is the cluster algebra $A(B^{\rm{prin}})$, where $B^{\rm{prin}}$ is the $2m \times m$ matrix whose top $m\times m$ square submatrix is $B$ and whose lower $m\times m$ square submatrix is the identity matrix~$\mathbbm{1}_m$.
\end{Definition}

The notion of cluster algebras with principal coefficients is central in cluster theory. Its~rele\-vance in the framework of toric degenerations of cluster varieties was highlighted in~\cite{GHKK}. In~particular, in loc. cit. the authors work implicitly with cluster algebras with frozen variables and principal coefficients. This is precisely the notion we need here.

\begin{Definition}
Let $ \widetilde{B}$ be an extended exchange matrix. Let $\widehat{B}$ be the $(m+f)\times (m+f)$ skew-symmetric square matrix whose left $(m+f)\times m$ submatrix is $\widetilde{B}$ and whose bottom right~$f \times f $ square submatrix is the zero matrix (this information fully determines $ \widehat{B}$ since it is skew-symmetric).
Denote by $t_1,\dots , t_{m+f}$ the variables of $A\big(\widehat{B}^{\rm prin}\big) $ associated to the last $m+f$ directions, which in this case are the frozen directions.
Then the {\it cluster algebra with principal coefficients and frozen variables} $A\big(\widetilde{B}^{\rm prin}\big) $ associated to $\widetilde{B}$ is the $\mathbb{C}[t_1,\dots , t_{m+f}]$-subalgebra of~$A\big(\widehat{B}^{\rm prin}\big)$ spanned by the cluster variables that can be obtained from the initial extended cluster by iterated mutation in the first $m$ mutable directions.
\end{Definition}

\begin{Theorem}[{\cite[Theorem~1.7]{DWZ}}]
\thlabel{thm:g-vecs}
Let $X$ be a cluster variable associated to $A\big(\widetilde{B}^{\rm{prin}}\big)$.
By Theo\-rem~{\rm \ref{thm:SLaurent}} we know that $X\in \mathbb C [t_1,\dots ,t_{m+f}]\big[x_1^{\pm 1}, \dots , x_m^{\pm 1}, x_{m+1}, \dots , x_{m+f}\big]$.
Let $X|_{{\bf t}=0}$ be the Laurent polynomial in $\mathbb C \big[x_1^{\pm 1}, \dots , x_m^{\pm 1},x_{m+1}, \dots , x_{m+f}\big]$ obtained after the evaluation of $X$ at $t_1=0 ,\dots , t_{m+f}=0$. Then $X|_{{\bf t} = 0}$ is a non-constant Laurent monomial with coefficient~$1$. In~other words, for a non-zero vector ${\bf g}(X)=\big(g_1^X,\dots, g_{m+f}^X\big)\in \mathbb Z^{m+f}$ we have that
\begin{displaymath}
X|_{{\bf t}=0}=\prod_{i=1}^{m+f} x_i^{g^X_i}.
\end{displaymath}
Moreover, if $X$ and $X'$ are different cluster variables then ${\bf g }(X)\neq {\bf g}(X')$.
\end{Theorem}

\begin{Definition}\thlabel{def:g-vectors}
Let $X$ be a cluster variable associated to $A\big(\widetilde{B}^{\rm{prin}}\big)$. The vector ${\bf g}(X)\in \mathbb{Z}^{m+f}$ introduced in~\thref{thm:g-vecs} is the {\bf g}-{\it vector} associated to~$X$.
We refer to the vector $\big(g_1^X,\dots , g_m^X\big)\in \mathbb{Z}^m$ as the {\it truncated} {\bf g}-{\it vector} of $X$. The set of (truncated) {\bf g}-vectors associated to $\widetilde{B}$ is the set of (truncated) {\bf g}-vectors of all the cluster variables in $A\big(\widetilde{B}^{\rm{prin}}\big)$.
\end{Definition}

\begin{Remark}
Let $X|_{{\bf t}=1}$ be the Laurent polynomial in $\mathbb C \big[x_1^{\pm 1}, \dots , x_m^{\pm 1},x_{m+1}, \dots , x_{m+f}\big]$ obtai\-ned after the evaluation of $X$ at $t_1=1 ,\dots , t_{m+f}=1$. Then $X|_{{\bf t}=1}$ is a cluster variable of~$A\big(\widetilde{B}\big)$. By a slight abuse of terminology we say ${\bf g}(X)$ is the ${\bf g}$-vector of the cluster variable~$X|_{{\bf t}=1}$.
\end{Remark}

\noindent
{\bf Universal coefficients.} We now turn to the definition of finite type cluster algebras \emph{with} frozen variables \emph{and} universal coefficients.
This notion is a natural (and slight) generalization of the usual notion of universal coefficients for cluster algebras \emph{without} frozen variables introduced by Fomin and Zelevinsky in~\cite[Section~12]{FZ_clustersIV}.
The difference between these notions arises from the distinction we make between coefficients and frozen variables, notions that are usually identified.
This distinction was first suggested in~\cite{BFMN} and is of particular importance in the study of toric degenerations of cluster varieties.
It was shown in~\cite{FZ_clustersIV} that finite type cluster algebras can be endowed with universal coefficients. This construction was categorified in~\cite{NC_univ} and the categorification was used to perform various computations in the following sections.
To further clarify the preceding discussion let us first recall the notion of finite type cluster algebras with universal coefficients.

\begin{Definition}\thlabel{ice quiver/matrix}
An {\it ice quiver} is a pair $(Q,F)$, where $Q$ is a finite quiver with the vertex set~$V(Q)$ and $F\subset V(Q)$ is a set of {\it frozen vertices}.
The vertices in $V(Q)\backslash F$ are called {\it mutable}.
If $(Q,F)$ is an ice quiver then the corresponding $|V(Q)|\times |V(Q)\setminus F|$ matrix is $B_{(Q,F)}:=(b_{ij})$, where
\begin{equation}
\label{eq:matrix_quiver}
b_{ij}:= \#\{ \text{arrows } i \to j \text{ in }Q\} - \#\{ \text{arrows } j \to i \text{ in }Q\}
\end{equation}
for $i\in V(Q)$ and $j\in V(Q)\setminus F$. Note that $(\ref{eq:matrix_quiver})$ can be similarly used to recover $(Q,F)$ from~$B_{(Q,F)}$. The {\it quiver mutation} is the operation induced by matrix mutation at the level of quivers.
If $F=\emptyset$ we set $B_{Q}:=B_{(Q,\emptyset)}$.
\end{Definition}
\begin{Theorem-Definition}[{\cite[Corollary 8.15]{Reading}}]
\thlabel{def:univ_coeff_R}
Let $Q$ be a quiver of finite cluster type, that is $Q$ is mutation equivalent to an orientation of a Dynkin diagram.
The {\it cluster algebra with universal coefficients} associated to $Q$ is $A\big(B_Q^{{\rm univ}}\big)$, where
\begin{displaymath}
B^{{\rm univ}}_Q=
\left(
\begin{array}{c}
 B_Q \\ \hdashline
 U_Q
\end{array}
\right)
\end{displaymath}
and $U_Q$ is the rectangular matrix whose rows are the {\bf g}-vectors of the opposite quiver $Q^{\rm op}$.
\end{Theorem-Definition}

Observe that we have described the rows of $U_Q$ but did not specify in which order they appear.
Any order we choose provides a realization of $A\big(B_Q^{{\rm univ}}\big)$ since a rearrangement of the rows of $U_Q$ amounts to reindexing the corresponding coefficients. So any choice gives rise to an isomorphic algebra.
It can be verified that the cluster algebra with universal coefficients associated to $Q$ is an invariant of the mutation class of $Q$. In other words, if $Q'$ is mutation equivalent to $Q$ then $A\big(B_Q^{{\rm univ}}\big)$ is canonically isomorphic to $A\big(B_{Q'}^{{\rm univ}}\big)$. The canonical isomorphism is completely determined by sending the elements of the initial cluster of $B_{Q}^{{\rm univ}}$ to those of $B_{Q'}^{{\rm univ}}$.

\begin{Remark}
\thlabel{rem:univ_property}
Cluster algebras with universal coefficients as treated by Fomin and Zelevinsky are defined by a universal property~\cite[Definition~12.1]{FZ_clustersIV}. We have chosen Reading's equivalent definition since we will not use the universal property here.
\end{Remark}

If $Q$ is a bipartite orientation of a Dynkin diagram $\Delta$ then the matrix $B^{\rm univ}_Q$ has a particularly nice and explicit description.
More precisely, let $\Phi^{\vee}$ be the root system dual to the root system associated to $\Delta$. Fix $\Phi^{\vee}_+\subset \Phi^{\vee}$ a subset of positive coroots with $\{ \alpha_1, \dots , \alpha_n \} \subset \Phi^{\vee}_+$ the set of simple coroots.
The almost positive coroots are the coroots in the set $\Phi^{\vee}_{\geq -1}:= \Phi^{\vee}_+ \cup \{ -\alpha_1, \dots , -\alpha_n \}$.
The cardinality of $\Phi^{\vee}_{\geq -1}$ coincides with the number of cluster variables in $A\big(B^{\rm univ}_{Q}\big)$ which is the number of coefficients for $A\big(B^{\rm univ}_{Q}\big)$.
Hence, we may consider $\big\{ y_{\alpha} \colon \alpha \in \Phi^{\vee}_{\geq -1} \big\} $ as the set of coefficients for $A\big(B^{\rm univ}_{Q}\big)$.
Next, for a coroot $\beta = \sum_{i=1}^n c_i \alpha_i$ we define $[\beta: \alpha_i]:=c_i$. With this notation we have the following description of $ B^{\rm univ}_Q $ which is a reformulation of~\cite[Theorem~12.4]{FZ_clustersIV}.

\begin{Proposition}
\thlabel{prop:bipartite_univ_Q}
The entry in $B_Q^{\rm univ}$ in the column $i $ and the row corresponding to $y_{\beta} $ is $[\beta : \alpha_i]$ if $i$ is a source and $-[\beta : \alpha_i]$ if $i$ is a sink.
\end{Proposition}

\begin{Definition}\thlabel{def:univ frozen}
Let $(Q,F)$ be an ice quiver whose mutable part is a quiver $Q^{\rm mut}$ of finite cluster type. The {\it associated cluster algebra with universal coefficients} is $A\big(B_{(Q,F)}^{{\rm univ}}\big)$, where
\begin{displaymath}
B^{{\rm univ}}_{(Q,F)}=
\left(
\begin{array}{c}
 B_{(Q,F)} \\ \hdashline
 U_{Q^{\rm mut}}
\end{array}
\right)
\end{displaymath}
and $U_{Q^{\rm mut}}$ is the matrix whose rows are the truncated {\bf g}-vectors of the opposite quiver $(Q^{\rm mut})^{\rm op}$.
\end{Definition}

\begin{Remark}
The algebra $A\big(B_{(Q,F)}^{{\rm univ}}\big)$ satisfies a universal property which is a straight forward generalization of the universal property of a cluster algebra without frozen directions and with universal coefficients, see~\thref{rem:univ_property}.
\end{Remark}

\noindent
{\bf The Grassmannian cluster algebra $\boldsymbol{A_{2,n}}$}. We now turn to the cluster algebra that is of most interest to us, and use the combinatorics governing the cluster structure of $A_{2,n}$: triangulations of an $n$-gon. The vertices of the $n$-gon are labeled by $[n]$ in the clockwise order.
An arc connecting two vertices $i$ and $j$ is denoted by $\overline{ij}$ and we associate to it the Pl\"ucker coordinate $\bar p_{ij}\in A_{2,n}$.

\begin{Definition}\thlabel{def:triang}
A {\it triangulation} of the $n$-gon is a maximal collection of non-crossing arcs dividing the $n$-gon into $n-2$ triangles.
To a triangulation $T$ we associate a collection of Pl\"ucker coordinates ${\bf x}_T:=\big\{\bar p_{ij} \colon \overline{ij} \text{ is an arc in }T\big\}$ which is the associated {\it cluster}.
The elements of ${\bf x}_T$ are the {\it cluster variables}.
The monomials in Pl\"ucker coordinates which are all part of one cluster are the {\it cluster monomials}.
\end{Definition}

Note that every triangulation $T$ contains the boundary edges $\overline{12},\overline{23},\dots,\overline{n1}$.
The correspon\-ding cluster variables $\bar p_{12},\bar p_{23},\dots,\bar p_{1n}$ are {\it frozen} cluster variables.
All other variables $\bar p_{ij}$ with $i\not=j\pm 1$ are {\it mutable} cluster variables.
To every triangulation $T$ we associated an ice quiver.

\begin{Definition}\thlabel{def:quiver}
For a triangulation $T$, we define its {\it associated ice quiver $\big(Q_T,F_T\big)$}: the vertices $V(Q_T)$ are in correspondence with arcs and boundary edges $\overline{ij}\in T$; Two vertices $v_{\overline{ij}}$ and $v_{\overline{kl}}$ are connected by an arrow, if they correspond to arcs in one triangle. Inside every triangle, we orient the arrows clockwise.
The set $F_T$ consists of vertices corresponding to the boundary edges of the $n$-gon, hence $F$ does not depend on $T$.
We neglect arrows between vertices in $F$, see Figure~\ref{fig:quiver}.
\end{Definition}

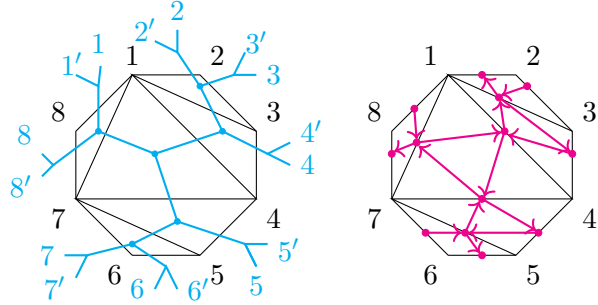
\begin{figure}[h!]
 \centering
 \begin{tikzpicture}[scale=.3]
 \draw (0,0) -- (3,0) -- (5.5,2.5) -- (5.5,5.5) -- (3,8) -- (0,8) -- (-2.5,5.5) -- (-2.5,2.5) -- (0,0);
 \draw (5.5,5.5) -- (0,8) -- (5.5,2.5) -- (-2.5,2.5) -- (3,0);
 \draw (0,8) -- (-2.5,2.5);
 \node[above] at (0,8) {1};
 \node[above right] at (3,8) {2};
 \node[above right] at (5.5,5.5) {3};
 \node[below right] at (5.5,2.5) {4};
 \node[below right] at (3,0) {5};
 \node[below left] at (0,0) {6};
 \node[below left] at (-2.5,2.5) {7};
 \node[above left] at (-2.5,5.5) {8};

 \draw[thick, cyan] (-2,0) -- (0,.5) -- (1.5,-.5);
 \draw[thick, cyan] (0,.5) -- (2,1.5) -- (5,.5) -- (5.5,-.5);
 \draw[thick, cyan] (2,1.5) -- (1,4.5) -- (-1.5,5.5) -- (-3.5,4);
 \draw[thick, cyan] (-1.5,5.5) -- (-1.5,7.5);
 \draw[thick, cyan] (1,4.5) -- (4,5.5) -- (6,4.5);
 \draw[thick, cyan] (4,5.5) -- (3,7.5) -- (2,9);
 \draw[thick, cyan] (3,7.5) -- (4.5,8);

 \draw[fill,cyan] (-1.5,5.5) circle (.125cm);
 \draw[fill,cyan] (1,4.5) circle (.125cm);
 \draw[fill,cyan] (2,1.5) circle (.125cm);
 \draw[fill,cyan] (4,5.5) circle (.125cm);
 \draw[fill,cyan] (0,.5) circle (.125cm);
 \draw[fill,cyan] (3,7.5) circle (.125cm);

 \draw[thick, cyan] (-3,0) -- (-2,0) -- (-2.5,-1);
 \draw[thick, cyan] (6,.5) -- (5,.5) -- (5.5,-.5);
 \draw[thick, cyan] (1,-1.5) -- (1.5,-.5) -- (2,-1.5);
 \draw[thick, cyan] (-4,3.5) -- (-3.5,4) -- (-4,4.5);
 \draw[thick, cyan] (-2.5,8) -- (-1.5,7.5) -- (-1.4,8.4);
 \draw[thick, cyan] (1,9.5) -- (2,9) -- (2,10);
 \draw[thick, cyan] (5.5,8) -- (4.5,8) -- (5,9);
 \draw[thick, cyan] (7,5) -- (6,4.5) -- (7,4);

 \node[above, cyan] at (-1.5,8.5) {1};
 \node[left, cyan] at (-1.75,8.5) {$1'$};
 \node[above, cyan] at (2,10) {2};
 \node[left, cyan] at (1.5,10) {$2'$};
 \node[right, cyan] at (5.5,8) {3};
 \node[above, cyan] at (5.5,8.5) {$3'$};
 \node[right, cyan] at (7,4) {4};
 \node[right, cyan] at (7,5.5) {$4'$};
 \node[below, cyan] at (5.5,-.5) {5};
 \node[right, cyan] at (6,0.25) {$5'$};
 \node[left, cyan] at (1,-1.5) {6};
 \node[right, cyan] at (2,-1.5) {$6'$};
 \node[left, cyan] at (-3,0) {7};
 \node[below left, cyan] at (-2.5,-.5) {$7'$};
 \node[above left, cyan] at (-4,4.5) {8};
 \node[below left, cyan] at (-4,4) {$8'$};

 \begin{scope}[xshift=14cm]
 \draw (0,0) -- (3,0) -- (5.5,2.5) -- (5.5,5.5) -- (3,8) -- (0,8) -- (-2.5,5.5) -- (-2.5,2.5) -- (0,0);
 \draw (5.5,5.5) -- (0,8) -- (5.5,2.5) -- (-2.5,2.5) -- (3,0);
 \draw (0,8) -- (-2.5,2.5);
 \node[above left] at (0,8) {1};
 \node[above right] at (3,8) {2};
 \node[above right] at (5.5,5.5) {3};
 \node[below right] at (5.5,2.5) {4};
 \node[below right] at (3,0) {5};
 \node[below left] at (0,0) {6};
 \node[below left] at (-2.5,2.5) {7};
 \node[above left] at (-2.5,5.5) {8};

\draw[fill,magenta] (1.5,8) circle (.15cm);
\draw[fill,magenta] (3.5,7.5) circle (.15cm);
\draw[fill,magenta] (2.25,7) circle (.15cm);
\draw[fill,magenta] (-1.5,6.5) circle (.15cm);
\draw[fill,magenta] (2.5,5.5) circle (.15cm);
\draw[fill,magenta] (-1.4,5) circle (.15cm);
\draw[fill,magenta] (-2.5,4.5) circle (.15cm);
\draw[fill,magenta] (5.5,4.5) circle (.15cm);
\draw[fill,magenta] (1.5,2.5) circle (.15cm);
\draw[fill,magenta] (-1,1) circle (.15cm);
\draw[fill,magenta] (.75,1) circle (.15cm);
\draw[fill,magenta] (1.5,0) circle (.15cm);
\draw[fill,magenta] (4,1) circle (.15cm);

\draw[thick, ->, magenta] (3.5,7.5) -- (2.35,7.1);
\draw[thick, ->, magenta] (2.25,7) -- (1.6,7.9);
\draw[thick, ->, magenta] (2.25,7) -- (5.4,4.6);
\draw[thick, ->, magenta] (5.5,4.5) -- (2.65,5.35);
\draw[thick, ->, magenta] (2.5,5.5) -- (2.3,6.85);
\draw[thick, ->, magenta] (2.5,5.5) -- (1.6,2.6);
\draw[thick, ->, magenta] (1.5,2.5) -- (-1.3,4.85);
\draw[thick, ->, magenta] (-1.4,5) -- (2.35,5.45);
\draw[thick, ->, magenta] (1.5,2.5) -- (3.9,1.1);
\draw[thick, ->, magenta] (4,1) -- (.9,1);
\draw[thick, ->, magenta] (.75,1) -- (1.4,2.4);
\draw[thick, ->, magenta] (.75,1) -- (1.4,.1);
\draw[thick, ->, magenta] (-1,1) -- (.6,1);
\draw[thick, ->, magenta] (-1.5,6.5) -- (-1.4,5.15);
\draw[thick, ->, magenta] (-1.4,5) -- (-2.4,4.6);
 \end{scope}
 \end{tikzpicture}
 \caption{On the right a triangulation $T$ of the $8$-gon and its corresponding quiver $Q_T$.
 On the left its corresponding extended tree $\widehat{D}_T$ dual to $T$ which is a trivalent tree with $16$ leaves.
 }
 \label{fig:quiver}
\end{figure}

We first recall some classic results due to Fomin--Zelevinsky and Fomin--Shapiro--Thurston.
\begin{Theorem}[\cite{FZ02}]
For all $n \geq 3$ the ring $A_{2,n}$ is a cluster algebra with frozen variables. More precisely, if $\big(Q_T,F_T\big)$ is the ice quiver associated to a triangulation $T$ of the $n$-gon then $A_{2,n}\cong A\big(B_{(Q_T,F_T)}\big)$ as $\mathbb C$-algebras. Moreover, the cluster variables associated to $B_{(Q_T,F_T)}$ are precisely the Pl\"ucker coordinates and the exchange relations are precisely the Pl\"ucker relations.
\end{Theorem}
\begin{Theorem}[\cite{FST08}]\thlabel{thm:cluster mono basis}
The set of all cluster monomials is a $\mathbb{C}$-vector space basis for the algebra~$A_{2,n}$ called the {\it basis of cluster monomials}.
\end{Theorem}

We endow $A_{2,n}$ with universal coefficients associated to its mutable part. This is an example of a cluster algebra with frozen directions and coefficients in the sense of~\cite{BFMN}.

\begin{Notation}
Let $\big(Q_T, F_T\big)$ be the quiver associated to a triangulation of the $n$-gon.
We~denote by $Q_T^{\rm{mut}}$ the full subquiver of $Q_T$ supported in the mutable vertices $V(Q_T)\setminus F_T$. In~particular, $Q_T^{\rm{mut}}$ is mutation equivalent to an orientation of a type $\mathtt A$ Dynkin diagram.
\end{Notation}

\begin{Definition}\thlabel{def:Plucker univ}
Let $T$ be a triangulation of the $n$-gon. The {\it Pl\"ucker algebra with universal coefficients} $A_{2,n}^{\rm{univ}}$ is the cluster algebra defined by the extended exchange matrix
\begin{displaymath}
B_{(Q_T,F_T)}^{\rm{univ}}=
\left(
\begin{array}{c}
 B_{(Q_T,F_T)} \\ \hdashline
 U_{Q^{\rm{mut}}_T}
\end{array}
\right)\!.
\end{displaymath}
Up to canonical isomorphism, $A_{2,n}^{\rm{univ}}=A\big(B_{(Q_T,F_T)}^{\rm{univ}}\big)$ is independent of $T$.
\end{Definition}

The quivers $Q_T^{\rm mut}$ and $\big(Q_T^{\rm mut}\big)^{\rm op}$ define canonically isomorphic cluster algebras whose cluster variables (denoted by $\bar p_{ij}$) are both in bijection with arcs $\overline{ij}$ of the $n$-gon for which $2\le i+1<j\le n$.
Hence, the rows of $U_{Q_T^{\rm mut}}$ are also in bijection with these arcs.
We write $y_{ij}$ for the coefficient variable of $A_{2,n}^{\rm univ}$ corresponding to the arc $\overline{ij}$.

\begin{Example}
Let $\big(Q_T,F_T\big)$ be the ice quiver associated to the triangulation of the pentagon depicted in Figure~\ref{fig:pentagon}. The corresponding rectangular matrix for universal coefficients can be found on the right side of Figure~\ref{fig:pentagon}.
Incidentally, in this case the rows corresponding to the universal coefficients coincide with the rows corresponding to the frozen part of $\big(Q_T,F_T\big)$. However, for polygons with more than 5 sides this will not be the case.
\end{Example}

\begin{figure}
 \centering
 \begin{tikzpicture}[scale=.35]
 \draw (0,5) -- (3,3) -- (2,0) -- (-2,0) -- (-3,3) -- (0,5);
 \draw (-2,0) -- (0,5) -- (2,0);
 \draw[fill,capri] (-2.5,1.5) circle (.175cm);
 \draw[fill,capri] (0,0) circle (.175cm);
 \draw[fill,capri] (2.5,1.5) circle (.175cm);
 \draw[fill,capri] (1.5,4) circle (.175cm);
 \draw[fill,capri] (-1.5,4) circle (.175cm);
 \draw[fill,capri] (1,2.5) circle (.175cm);
 \draw[fill,capri] (-1,2.5) circle (.175cm);
 \draw[capri,thick,->] (-1,2.5) -- (.8,2.5);
 \draw[capri,thick,->] (1,2.5) -- (1.4,3.8);
 \draw[capri,thick,->] (2.5,1.5) -- (1.2,2.4);
 \draw[capri,thick,->] (1,2.5) -- (.15,.15);
 \draw[capri,thick,->] (0,0) -- (-.9,2.4);
 \draw[capri,thick,->] (-1,2.5) -- (-2.35,1.65);
 \draw[capri,thick,->] (-1.5,4) -- (-1.15,2.65);

 \node[above] at (0,5) {1};
 \node[right] at (3,3) {2};
 \node[below right] at (2,0) {3};
 \node[below left] at (-2,0) {4};
 \node[left] at (-3,3) {5};

 \begin{scope}[xshift=15cm]

 \node at (0,2.5) {$
 B^{\rm univ}_{(Q_T,F_T)}=\left(\scriptsize{
\begin{array}{rr}
 0 & 1 \\
 -1 & 0 \\
 1 & 0 \\
 0 & -1 \\
 0 & 1 \\
 1 & -1 \\
 -1 & 0 \\ \hdashline
 1 & 0 \\
 0 & -1 \\
 0 & 1 \\
 1 & -1 \\
 -1 & 0
\end{array}}
\right)
=
\left(
\begin{array}{c}
 B_{(Q_T,F_T)} \\ \hdashline
 U_{Q_T^{\rm mut}}
\end{array}
\right)
 $};
 \end{scope}
 \end{tikzpicture}
 \caption{A triangulation of the $5$-gon and the associated matrix $B_{(Q_T,F_T)}^{\rm univ}$.}
 \label{fig:pentagon}
\end{figure}
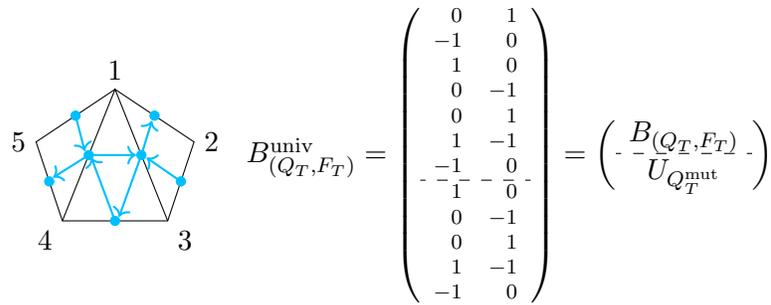

To get a better understanding of the cluster algebra $A_{2,n}^{\rm univ}$ we now turn to the combinatorial gadgets that govern its definition: {\bf g}-vectors of Pl\"ucker coordinates.
Their description involves two more combinatorial objects associated to a triangulation $T$, namely its dual graph $D_T$ and its extended dual graph $\widehat D_T$.
Recall that non-interior vertices of a tree graph are called {\it leaves}.
Moreover, two leaves that are connected to the same interior vertex form a so-called {\it cherry}.

\begin{Definition}\thlabel{def:trees}
Fix a triangulation $T$ of the $n$-gon. The trivalent tree $D_T$ is the {\it dual graph} or {\it ribbon graph} of $T$. It has $n$ leaves and its interior vertices correspond to triangles in $T$.
 Two interior vertices are connected by an edge if their corresponding triangles share an arc. Vertices corresponding to triangles involving a boundary edge $\overline{i-1,i}$ are connected to the leaf $i$.
 The trivalent tree $\widehat D_T$
 is obtained from $D_T$ by replacing every leaf $i$ by a cherry with leaves $i$ and $i'$, where (in the clockwise order) $i'$ labels the first leaf.
 We call $\widehat D_T$ the {\it extended tree dual to $T$}.
\end{Definition}
Note that the tree $D_T$ is by definition planar. See Figure~\ref{fig:quiver} for an example.
Moreover, the extended tree $\widehat D_T$ can alternatively be defined as the dual graph of a triangulation $\widehat T$ of the $2n$-gon with vertices $1',1,2',2,\dots,n',n$ in the clockwise order.
Here $\widehat T$ is obtained from $T$ by~replacing every boundary edge $\overline{i-1,i}$ by a triangle with a new vertex labeled by $i'$.
Then $\widehat{D}_T=D_{\widehat{T}}$.

We now focus on giving a combinatorial definition of the ${\bf g}$-vector associated to a Pl\"ucker coordinate with respect to a triangulation $T$.
The reader can find an equivalent definition above in~\thref{def:g-vectors}. The ${\bf g}$-vectors can be read from the extended tree $\widehat D_T$.
First, observe that the interior edges of $\widehat D_T$ correspond to the arcs and the boundary edges in $T$ and, therefore to the cluster variables of the associated seed.

\medskip

\noindent{\bf Sign conventions for ${\bf g}$-vectors.}
Consider a path in the tree $\widehat D_T$ with the end points in~$\{1,\dots,n\}$ (not in~$\{1',\dots,n'\}$).
Say the path goes from $i$ to $j$.
For simplicity, we fix an~ori\-en\-tation of the path and denote it by $i\rightsquigarrow j$ (what follows is independent of the choice of orientation).
As $\widehat D_T$ is trivalent, at every interior vertex the path can either turn \emph{right} or \emph{left}.
Denote by $e_{\overline{ab}}$ the edge in $\widehat{D}_T$ corresponding to an arc $\overline{ab}\in T$ and let $v_{abc}$ and $v_{abd}$ be the vertices in $\widehat{D}_T$ adjacent to $e_{\overline{ab}}$.
Given the path $i\rightsquigarrow j$ we associate signs $\sigma_{i\rightsquigarrow j}^{ab}\in \{-1,0,+1\}$ to every $\overline{ab}\in T$ as follows:
\begin{enumerate}\itemsep=0pt
 \item
 If the path $i\rightsquigarrow j$ either does not pass through $e_{\overline{ab}}$, or it passes through $e_{\overline{ab}}$ by turning \emph{right} at $v_{abc}$ and turning \emph{right} at $v_{abd}$, or it passes through $e_{\overline{ab}}$ by turning \emph{left} at $v_{abc}$ and turning \emph{left} at $v_{abd}$: then
 $\sigma_{i\rightsquigarrow j}^{{ab}}=0$.
 \item If the path $i\rightsquigarrow j$ passes through the $e_{\overline{ab}}$ by turning \emph{left} at $v_{abc}$ and turning \emph{right} at $v_{abd}$: then $\sigma_{i\rightsquigarrow j}^{{ab}}=+1$.
 \item If the path $i\rightsquigarrow j$ passes through the $e_{\overline{ab}}$ by turning \emph{right} at $v_{abc}$ and turning \emph{left} at $v_{abd}$: then $\sigma_{i\rightsquigarrow j}^{{ab}}=-1$.
\end{enumerate}

Cases 2 and 3 are depicted in Figure~\ref{fig:sign g-v}. We leave it to the reader to verify that $\sigma_{i\rightsquigarrow j}^{{ab}}=\sigma_{j\rightsquigarrow i}^{{ab}}$.
Later we replace the end points $i$ and $j$ by the interior vertices of $\widehat{D}_T$ and use the same symbol.

\begin{figure}[ht]
 \centering
\begin{tikzpicture}[scale=1.1]
\draw[thick] (-.5,-.5) -- (0,0) -- (1,0) -- (1.5,.5);
\draw[dashed] (-.5,.5) -- (0,0);
\draw[dashed] (1,0) -- (1.5,-.5);
\node[above] at (0.5,0) {$-$};
\begin{scope}[xshift=4cm]
\draw[thick] (-.5,.5) -- (0,0) -- (1,0) -- (1.5,-.5);
\draw[dashed] (-.5,-.5) -- (0,0);
\draw[dashed] (1,0) -- (1.5,.5);
\node[above] at (0.5,0) {$+$};
\end{scope}
\end{tikzpicture}
 \caption{Sign conventions to compute ${\bf g}$-vectors.}
 \label{fig:sign g-v}
\end{figure}
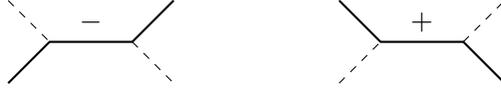

\begin{Definition}\thlabel{def:g}
Let $\{f_{{ab}}\}_{\overline{ab}\in T}$ be the standard basis of $\mathbb{Z}^{2n-3}$.
For the Pl\"ucker coordinate $\bar p_{ij}$ consider the path from $i$ to $j$ in $\widehat D_T$. The ${\bf g}$-{\it vector} of $\bar p_{ij}$ with respect to $T$ is defined as
\begin{displaymath}
{\bf g}^{\widehat T}_{ij}:=\sum_{\overline{ab}\in T} \sigma_{i\rightsquigarrow j}^{ab}f_{ab}.
\end{displaymath}
Similarly, we define the {\bf g}-vector of a cluster monomial as the sum of the {\bf g}-vectors of its factors.
\end{Definition}
Note that for $\overline{ij}\in T$ we have ${\bf g}_{ij}=f_{ij}$.
When there is no ambiguity, we write ${\bf g}_{ij}$ instead of~${\bf g}^{\widehat T}_{ij}$.
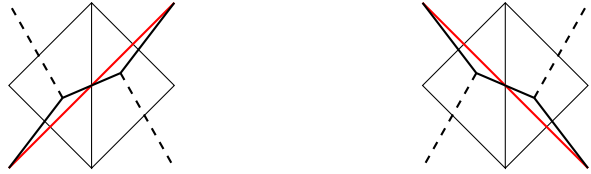
\begin{figure}[ht]
 \centering
\begin{tikzpicture}[scale=1.1]
\draw (0,0) -- (1,1) -- (0,2) -- (-1,1) -- (0,0) -- (0,2);
\draw[thick, red] (-1,0) -- (1,2);
\draw[thick] (1,2) -- (.35,1.15) -- (-.35,.85) -- (-1,0);
\draw[dashed,thick] (.35,1.15) -- (1,0);
\draw[dashed,thick] (-.35,.85) -- (-1,2);

\begin{scope}[xshift=5cm]
\draw (0,0) -- (1,1) -- (0,2) -- (-1,1) -- (0,0) -- (0,2);
\draw[thick, red] (-1,2) -- (1,0);
\draw[thick] (1,0) -- (.35,.85) -- (-.35,1.15) -- (-1,2);
\draw[dashed,thick] (.35,.85) -- (1,2);
\draw[dashed,thick] (-.35,1.15) -- (-1,0);
\end{scope}
\end{tikzpicture}
 \caption{Combinatorial translation between laminations on triangulations (as in~\cite{fomin2018cluster}) and paths on~trees (as in~\thref{def:g}) to compute truncated ${\bf g}$-vectors. The red lines corresponds to the simple lamination associated to $\overline{ij} $.}
 \label{fig:gv dictionary}
\end{figure}
\begin{Lemma}\thlabel{lem:g are gv}
Let $T$ be a triangulation of the $n$-gon. The ${\bf g}$-vector ${\bf g}^{\widehat T}_{ij}$ of a Pl\"ucker coordinate~$\overline{p}_{ij}$ with respect to $T$ introduced in~\thref{def:g} coincides with the ${\bf g}$-vector ${\bf g}(\bar p_{ij})$ $($from~Defi\-ni\-tion~$\ref{def:g-vectors})$ of the cluster variable $\bar p_{ij} $ of $ A\big(B_{(Q_T,F_T)}\big)$.
\end{Lemma}

\begin{proof}
We use some of the terminology of cluster algebras associated to surfaces developed in~\cite{fomin2018cluster}.
The algebra $A_{2,n}$ is a cluster algebra arising from an orientable surface with marked points, in this case a disc with $n$ marked points in its boundary.
As such, the truncated ${\bf g}$-vectors of its cluster variables can be computed using laminations in the surface as explained in~\cite[Proposition 17.3]{fomin2018cluster}. To be more precise, the truncated {\bf g}-vector of $\bar p_{ij}$ is the vector containing the {\it shear coordinates} of the internal arcs of $T$ with respect to the simple lamination associated to the arc $\overline{ij} $.
The key observation is that the combinatorial rule defining shear coordinates translates to our combinatorial rule to compute ${\bf g}^{\widehat T}_{ij}$.
In Figure~\ref{fig:gv dictionary} we visualize the translation from our combinatorial rule to the one defining shear coordinates from~\cite[Definition~12.2, Figure~34]{fomin2018cluster}.
Shear coordinates are only associated to internal arcs and, therefore, can only compute truncated ${\bf g}$-vectors.
However, the extended exchange matrix $\widehat{B}_{(Q_T,F_T)}$ is the submatrix of $B_{Q_{\widehat T},F_{\widehat T}}$ associated to the internal arcs of the triangulation $\widehat T$ of the $2n$-gon.
As a conse\-qu\-ence~${\bf g}(\bar p_{ij})={\bf g}^{\widehat T}_{ij}$.
\end{proof}

\begin{Example}
Consider the triangulation $T$ depicted in Figure~\ref{fig:quiver} and its corresponding tree~$\widehat{D}_T$.
Then ${\bf g}_{ij}=f_{ij}$ for $\overline{ij}\in T$. We compute the ${\bf g}$-vectors of the remaining Pl\"ucker coordinates:
\begin{displaymath}
 \begin{split}
 {\bf g}_{15} & = f_{17}-f_{47}+f_{45},\\
 {\bf g}_{24} & = f_{12}-f_{13}+f_{34},\\
 {\bf g}_{25} & = f_{12}-f_{14}+f_{45},\\
 {\bf g}_{46} & =f_{47}-f_{57}+f_{56},\\
 {\bf g}_{58} & = f_{78} -f_{47} +f_{45},\\
 \end{split}
\quad\quad
 \begin{split}
 {\bf g}_{16} & = f_{17}-f_{57}+f_{56},\\
 {\bf g}_{27} & = f_{12}-f_{14}+f_{47},\\
 {\bf g}_{28} & = f_{12}-f_{17}+f_{78},\\
 {\bf g}_{48} & = f_{14}-f_{17}+f_{78},\\
 {\bf g}_{68} & = f_{56}-f_{57}+f_{78},\\
 \end{split}
 \quad\quad
 \begin{split}
{\bf g}_{37} & = f_{47}-f_{14},\\
{\bf g}_{35} & = f_{45}-f_{14},\\
{\bf g}_{38} & = f_{78}-f_{17},\\
{\bf g}_{26} & = f_{12}-f_{14}+f_{47}-f_{57}+f_{56},\\
{\bf g}_{36} & = -f_{14}+f_{47}-f_{57}+f_{56}.\\
\end{split}
\end{displaymath}
\end{Example}

\subsection{Toric degenerations}\label{sec: toric deg}
In this section, we use the ${\bf g}$-vectors from~\thref{def:g} to obtain a weight vector for every triangulation of the $n$-gon.
The initial ideals of $I_{2,n}$ arising this way are toric.
Throughout this section we fix a triangulation $T$.

\begin{Definition}\thlabel{def:monomial ideal}
Consider a quadrilateral with vertices $i$, $j$, $k$, $l$ inside the $n$-gon.
The arcs~$\overline{ik}$ and $\overline{jl}$ are called {\it crossing arcs}, if when drawn inside the $n$-gon the corresponding arcs cross each other.
Let $M_{2,n}$ be the set of {\it crossing monomials} $p_{ik}p_{jl}$, where in the quadrilateral with vertices~$i$, $j$, $k$, $l$ the arcs $\overline{ik}$ and $\overline{jl}$ are crossing.
Let $\mathcal M_{2,n}\subset S$ be the monomial ideal generated by~$M_{2,n}$.
\end{Definition}

\begin{Lemma}\thlabel{lem:g-v on relations}
Given a Pl\"ucker relation $p_{ij}p_{kl}-p_{ik}p_{jl}+p_{il}p_{jk}$ with $i<j<k<l$ the multidegree
induced by the ${\bf g}$-vectors for the triangulation $T$ on the monomials agrees for exactly two terms.
Moreover, one of those terms is the crossing monomial $p_{ik}p_{jl}$.
\end{Lemma}

\begin{figure}[t!]
 \centering
\begin{tikzpicture}[scale=.25]
\draw[thick] (3,0) -- (0,0) -- (-2,2);
\draw[thick,dashed] (-2,2) -- (-3,3);
\draw[thick] (-3,3) -- (-3.5,3.5);
\draw[thick] (-3.5,4) -- (-3.5,3.5) -- (-4,3.5);
\node[above] at (-3.5,4) {$i$};
\draw[line width = 5pt,capri, opacity=.5] (-2,2) -- (-3.5,3.5);
\node[above right,capri] at (-3,3) {\footnotesize $v_i$};

\draw[thick] (0,0) -- (-2,-2);
\draw[thick, dashed] (-2,-2) -- (-3,-3);
\draw[thick] (-3,-3) -- (-3.5,-3.5) -- (-3.5,-4);
\draw[thick] (-3.5,-3.5) -- (-4,-3.5);
\node[left] at (-4,-3.5) {$l$};
\draw[line width = 5pt,capri, opacity=.5] (-2,-2) -- (-3.5,-3.5);
\node[above left,capri] at (-3,-3) {\footnotesize $v_l$};

\draw[thick,magenta] (2,4) -- (2,-4) -- (-2,0) -- (2,4);
\node[above,magenta] at (2,4) {$a_1$};
\node[below,magenta] at (2,-4) {$b_1$};
\node[left,magenta] at (-2,0) {$c_1$};

\draw[thick,dashed] (3,0) -- (6,0);
\draw[thick] (6,0) -- (11,0); 
\draw [decorate,decoration={brace,amplitude=7pt}]
(3,0) -- (6,0) node [black,midway,yshift=.25cm,above]
{\footnotesize $v'$}; 

\draw[line width = 5pt,capri, opacity=.5] (3,0) -- (8,0);
\node[below,capri] at (5.5,0) {\footnotesize{$v$}}; 

\begin{scope} 
\draw[thick] (11,0) -- (13,2);
\draw[thick, dashed] (13,2) -- (14,3);
\draw[thick] (14,3) -- (14.5,3.5) -- (14.5,4);
\draw[thick] (14.5,3.5) -- (15,3.5);
\node[right] at (15,3.5) {$j$};
\draw[line width = 5pt,capri, opacity=.5] (13,2) -- (14.5,3.5);
\node[below right,capri] at (14,3) {\footnotesize $v_j$};

\draw[thick] (11,0) -- (13,-2);
\draw[thick,dashed] (13,-2) -- (14,-3);
\draw[thick] (14,-3) -- (14.5,-3.5) -- (14.5,-4);
\draw[thick] (14.5,-3.5) -- (15,-3.5);
\node[below] at (14.5,-4) {$k$};
\draw[line width = 5pt,capri, opacity=.5] (13,-2) -- (14.5,-3.5);
\node[below left,capri] at (14,-3) {\footnotesize $v_k$};

\draw[thick,magenta] (9,4) -- (9,-4) -- (13,0) -- (9,4);
\node[above,magenta] at (9,4) {$a_2$};
\node[right,magenta] at (13,0) {$b_2$};
\node[below,magenta] at (9,-4) {$c_2$};
\end{scope}


\draw[thick] (8,0) -- (7,5);
\draw[thick,dashed] (7,5) -- (7,6);
\draw[thick] (7,6) -- (7,6.5) -- (7.5,7);
\draw[thick] (7,6.5) -- (6.5,7);
\node[right] at (7.5,7) {$i'$};
\draw[line width=5pt,capri,opacity=.5] (7,5) -- (7,6.5);
\node[left,capri] at (7,5.75) {\footnotesize{$v_{i'}$}};

\draw[thick,caribbeangreen] (9,4) -- (5,4) -- (9,-4);
\node[above,caribbeangreen] at (5,4) {$d$};
\end{tikzpicture}
 \caption{The full subtree of $\widehat D_T$ with leaves $i$, $i'$, $j$, $k$, $l$ and labeling as in the proof of~\thref{lem:g-v on relations} (ignoring $i'$, $d$, $v'$) and~\thref{prop:wt map}. Blue marks refer to certain paths in the tree and the labels $v_i$, $v_j$, $v_k$, $v_l$, $v$, $v'$ to the corresponding elements of $\mathbb Z^{2n-3}$. We follow the convention in Figure~\ref{fig:sign g-v}.}
 \label{fig:lem compute g-vect}
\end{figure}
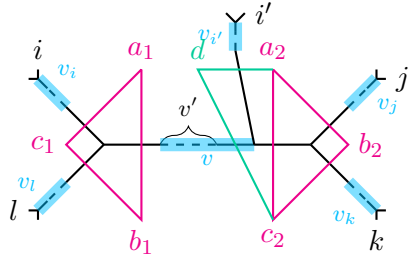

\begin{proof}
Consider the extended tree $\widehat D_T$ and restrict it to the full subtree with leaves $i$, $j$, $k$, $l$.
By definition, the subtree has four leaves and edges of valency three or less.
Without loss of generality we may assume the leaves are arranged as in Figure~\ref{fig:lem compute g-vect}.
Two vertices of the subtree are trivalent, namely the one where the paths $i\rightsquigarrow j$ and $l\rightsquigarrow j$ meet, respectively where the paths $j\rightsquigarrow i$ and $k\rightsquigarrow i$ meet.
Assume these vertices come from triangles in $T$ labeled by~$a_1$,~$b_1$,~$c_1$, respectively $a_2$, $b_2$, $c_2$.
Denote the vertex of $\widehat{D}_T$ corresponding to the triangle $a_1$,~$b_1$,~$c_1$ (respectively $a_2$, $b_2$, $c_2$) by~$\Delta_1$ (respectively $\Delta_2$).
We compute the ${\bf g}$-vectors for all Pl\"ucker coordinates whose indices are a subset of $\{i,j,k,l\}$.
To simplify the notation in the computation we use the following abbreviations
\begin{gather*}
v_i:= \sum_{\overline{pq}\in T} \sigma_{i\rightsquigarrow \Delta_1}^{pq}f_{pq}, \quad
v_j:= \sum_{\overline{pq}\in T} \sigma_{j\rightsquigarrow \Delta_2}^{pq}f_{pq}, \quad
v_l:= \sum_{\overline{pq}\in T} \sigma_{l\rightsquigarrow \Delta_1}^{pq}f_{pq}, \quad
v_k:= \sum_{\overline{pq}\in T} \sigma_{k\rightsquigarrow \Delta_2}^{pq}f_{pq}.
\end{gather*}
So $v_i$ is the weighted sum (with respect to the sign convention as in Figure~\ref{fig:sign g-v}) of the edges on the path from $i$ to the vertex $\Delta_1$ excluding $f_{a_1c_1}$, etc. Further, let $v:= \sum_{\overline{pq}\in T} \sigma_{\Delta_1\rightsquigarrow \Delta_2}^{pq}f_{pq}$. Then
\begin{gather*}
{\bf g}_{ij} = v_i +\sigma_{i\rightsquigarrow j}^{a_1c_1} f_{a_1c_1}+\sigma_{i\rightsquigarrow j}^{a_1b_1}f_{a_1b_1}+v + \sigma_{i\rightsquigarrow j}^{a_2c_2}f_{a_2c_2} + \sigma_{i\rightsquigarrow j}^{a_2b_2} f_{ab}+v_j,
\\
{\bf g}_{ik} = v_i + \sigma_{i\rightsquigarrow j}^{a_1c_1}f_{a_1c_1} +\sigma_{i\rightsquigarrow j}^{a_1b_1}f_{a_1b_1} + v + \big(1+ \sigma_{i\rightsquigarrow j}^{a_2c_2}\big)f_{a_2c_2} + \sigma_{i\rightsquigarrow k}^{b_2c_2}f_{b_2c_2} + v_k,
\\
{\bf g}_{jl} = v_l + \sigma_{j\rightsquigarrow l}^{b_1c_1}f_{b_1c_1} - \big(1- \sigma_{i\rightsquigarrow j}^{a_1b_1}\big)f_{a_1b_1} + v + \sigma_{i\rightsquigarrow j}^{a_2c_2}f_{a_2c_2} + \sigma_{i\rightsquigarrow j}^{a_2b_2}f_{a_2b_2} + v_j,
\\
{\bf g}_{kl} = v_l + \sigma_{j\rightsquigarrow l}^{b_1c_1}f_{b_1c_1} - \big(1 - \sigma_{i\rightsquigarrow j}^{a_1b_1}\big)f_{a_1b_1} + v + \big(1+\sigma_{i\rightsquigarrow j}^{a_2c_2}\big)f_{a_2c_2} + \sigma_{i\rightsquigarrow k}^{b_2c_2}f_{b_2c_2} + v_k,
\\
{\bf g}_{il} = v_i + \big(1+\sigma_{i\rightsquigarrow j}^{a_1c_1}\big)f_{a_1c_1} - \big(1-\sigma_{j\rightsquigarrow l}^{b_1c_1}\big)f_{b_1c_1} + v_l,
\\
{\bf g}_{jk} = v_j - \big(1+\sigma_{i\rightsquigarrow j}^{a_2b_2}
\big)f_{a_2b_2} + \big(1 +\sigma_{i\rightsquigarrow k}^{b_2c_2}\big)f_{b_2c_2} + v_k.
\end{gather*}
In particular, we have ${\bf g}_{ij}+{\bf g}_{kl}={\bf g}_{ik}+{\bf g}_{jl}\not = {\bf g}_{il}+{\bf g}_{jk}$.
\end{proof}

Recall from~\eqref{eq:tree wt} how to obtain a weight vector from a trivalent tree with $n$ leaves in $\text{Trop}(I_{2,n})$.
The following corollary is a direct consequence of the proof of~\thref{lem:g-v on relations}.
\begin{Corollary}\thlabel{cor:trop and g}
Without loss of generality by~\thref{lem:g-v on relations} assume that for $i<j<k<l$, ${\bf g}_{ij}+{\bf g}_{kl}={\bf g}_{ik}+{\bf g}_{jl}\not = {\bf g}_{il}+{\bf g}_{jk}$.
Then 
$\init_{D_{T}}(p_{ij}p_{kl}-p_{ik}p_{jl}+p_{il}p_{jk})=p_{ij}p_{kl}-p_{ik}p_{jl}$.
\end{Corollary}

\begin{Definition}\thlabel{def:wt map}
A {\it quadrilateral weight map} for a triangulation $T$ is a linear map $w_T: \mathbb R^{2n-3}$ $\to \mathbb R$, such that for every quadrilateral $\overline{abcd}$ with diagonal $\overline{ac}$ in $T$ we have:
\begin{equation}\label{eq:upper bound wt map}
w_T(f_{ad}+f_{bc})>w_T(f_{ab}+f_{cd}).
\end{equation}
\end{Definition}
Note that~\eqref{eq:upper bound wt map} is equivalent to saying that for every mutable vertex $\overline{ac}$ of the quiver $Q_T$ we require
$w_T\big(\sum_{\overline{pq}\to \overline{ac}} {\bf g}_{pq}\big)>w_T\big(\sum_{\overline{ac}\to \overline{pq}}{\bf g}_{pq}\big)$.
Before we show that quadrilateral weight maps do in fact exist (in~\thref{lem:ex quad}, using Algorithm~\ref{alg:weight})
we proceed by showing how they can be used to construct weight vectors in $\trop(I_{2,n})$ from ${\bf g}$-vectors.

\begin{Proposition}\thlabel{prop:wt map}
Consider any $i<j<k<l$ with ${\bf g}_{ij}+{\bf g}_{kl}={\bf g}_{ik}+{\bf g}_{jl}\not = {\bf g}_{il}+{\bf g}_{jk}$ and any quadrilateral linear map $w_T\colon \mathbb R^{2n-3}\to \mathbb R$. Then $w_T({\bf g}_{ij}+{\bf g}_{kl})=w_T({\bf g}_{ik}+{\bf g}_{jl}) < w_T({\bf g}_{il}+{\bf g}_{jk})$.
\end{Proposition}

\begin{proof}
Without loss of generality assume that the full subtree of $\widehat D_T$ with leaves $i$, $j$, $k$, $l$ is of~form as in Figure~\ref{fig:lem compute g-vect}.
For $s=1,2$ let $\Delta_s$ be the vertex of $\widehat D_T$ corresponding to the tri\-angle~$a_s$,~$b_s$,~$c_s$ in~$T$.
We proceed by induction on the number of edges along the unique path from $\Delta_2$ to $\Delta_1$, denote it by $p$.

\medskip \noindent
$p=1$: In this case we have $a_1=d=a_2$ and $b_1={c_2}$ in Figure~\ref{fig:lem compute g-vect}.
 Note that every quadrilateral linear map $w_T$ satisfies $w_T({f}_{a_2c_1}+{f}_{b_2c_2})>w_T({f}_{a_2b_2}+{f}_{c_2c_1})$.
 Using~\thref{lem:g-v on relations} we compute
 \begin{displaymath}
 w_T({\bf g}_{il}+{\bf g}_{jk})=w_T({\bf g}_{ij}+{\bf g}_{kl} + f_{a_2c_1}+f_{b_2c_2} - f_{a_2b_2}-f_{c_2c_1}) > w_T({\bf g}_{ij}+{\bf g}_{kl}).
 \end{displaymath}

\medskip \noindent
$p\ge 1$: Assume that the result holds for $p$. When the number of edges along the path from~$v_{a_2b_2c_2}$ to $v_{a_1b_1c_1}$ is $p+1$ we know there exists at least one leaf of form either $i'$ with~$i<i'<j$ or $k'$ with~$k<k'<l$.
 We treat the first case, where $i'$ exists depicted in Figure~\ref{fig:lem compute g-vect}.
 The proof of the second case is similar.
 All expressions of ${\bf g}$-vectors for Pl\"ucker variables in the relation~$R_{ijkl}$ appear in the proof of~\thref{lem:g-v on relations}.
 The ${\bf g}$-vectors involving $i'$ can be computed similarly.
 By~induction we know that for the relation $R_{ii'jl}$ we have $w_T({\bf g}_{ij}+{\bf g}_{i'l})<w_T({\bf g}_{il}+{\bf g}_{i'j})$.
 Further, by assumption on $w_T$ we have $w_T({f}_{a_2d}+{f}_{b_2c_2})>w_T({f}_{a_2b_2}+{f}_{c_2d})$.
 One verifies by direct computation that
 \begin{align*}
 w_T({\bf g}_{il}+{\bf g}_{jk}) &= w_T({\bf g}_{il}+{\bf g}_{i'j}) + w_T({\bf g}_{jk}-{\bf g}_{i'j}) \\
 &> w_T({\bf g}_{ij}+{\bf g}_{kl}) + w_T({\bf g}_{jk}-{\bf g}_{i'j}+{\bf g}_{i'l} - {\bf g}_{kl})\\
 &= w_T({\bf g}_{ij}+{\bf g}_{kl}) + w_T({f}_{a_2d}+{f}_{b_2c_2}-{f}_{a_2b_2}-{f}_{c_2d}) > w_T({\bf g}_{ij}+{\bf g}_{kl}).\!\!\!\!\!\! \tag*{\qed}
\end{align*}
\renewcommand{\qed}{}
\end{proof}

\begin{Definition}\thlabel{def:wt of T}
Fix
a quadrilateral linear map $w_T\colon \mathbb R^{2n-3}\to \mathbb R$.
We define ${\bf w}_T\in\mathbb R^{\binom{n}{2}}$ the {\it weight vector associated to $T$ and $w_T$} as $({\bf w}_T)_{ij}:=w_T({\bf g}_{ij})$.
\end{Definition}

{\sloppy
\begin{Proposition}\thlabel{thm:toric wt ideals}
For each quadrilateral weight map $w_T\colon \mathbb R^{2n-3}\to \mathbb R$, the initial ideal $\init_{{\bf w}_T}(I_{2,n})$ is toric.
Moreover, every binomial in the minimal generating set of $\init_{{\bf w}_T}(I_{2,n})$ corresponds to a~quadrilateral and it contains the monomial associated to the crossing in that quadrilateral.
\end{Proposition}

}

\begin{proof}
By~\thref{prop:wt map} and Corollary~\ref{cor:trop and g} we have $\init_{{D_T}}(R_{ijkl}) = \init_{{\bf w}_T}(R_{ijkl})$ for every Pl\"ucker relation $R_{ijkl}$.
By~\thref{lem:g-v on relations} these initial forms are binomials and one monomial corresponds to the crossing in the quadrilateral $\overline{ijkl}$.
By~\cite[Proof of Theorem~3.4]{SS04} we know that $\init_{{D_T}}(I_{2,n})$ is generated by $\init_{{D_T}}(R_{ijkl})$ for all $1\le i<j<k<l\le n$.
Moreover, by~\cite[Corollary~4.4]{SS04} the ideal $\init_{{D_T}}(I_{2,n})$ is binomial and prime.
\end{proof}

\begin{Remark}\thlabel{rmk:toric degen}
As shown in~\thref{thm:toric wt ideals}, each triangulation of the $n$-gon gives rise to a~Gr\"obner toric degeneration of $\Gr(2,\mathbb C^n)$.
For general Grassmannians, there are other examples of combinatorial objects leading to toric degenerations such as
plabic graphs~\cite{BFFHL,RW17} and matching fields~\cite{clarke2021combinatorial,clarke2019toric, MoSh}.
All such degenerations can be realized as Gr\"obner degenerations, nevertheless, this is not true in general; see, e.g.,
\cite{kateri2015family} for a family of toric degenerations arising from graphs that cannot be identified as Gr\"obner degenerations.
\end{Remark}

\noindent
{\bf Existence of quadrilateral weight maps:} For every triangulation $T$ we give a weight map obtained from a partition of the cluster ${\bf x}_T$, which is the output of Algorithm~\ref{alg:weight}.
We use it to define a linear map $v_T\colon \mathbb R^{2n-3}\to \mathbb R$ and show in~\thref{lem:ex quad} that it is quadrilateral for $T$.
\begin{algorithm}
\caption{Partitioning the cluster ${\bf x}_T$
based on its triangulation $T$ (union of triangles).
}
\label{alg:weight}

\KwIn{An initial cluster ${\bf x}_T$ with its corresponding triangulation $T$ and its quiver $Q_T$.
\\
}
\BlankLine
\KwOut{A partitioning of the initial seed indices.
}
\BlankLine
{\bf Initialization:}
$i\leftarrow 0$\;
$T\leftarrow$ the corresponding triangulation of ${\bf x}_T$\;
$Q\leftarrow Q_T$\;
\BlankLine
\While{$T$ is not empty}
{
\BlankLine
$V_i:=\{v \colon v\text{ is a frozen vertex in } Q \text{ and a sink}\}$\;
\ForEach{triangle $t$ in $T$}
{\If{$t$ has an edge whose corresponding vertex in $Q$ is in $V_i$}
{
$T\leftarrow T\setminus \{t\}$ (Remove the triangle $t$ from $T$)\;
$Q\leftarrow Q_T$\;
}
}
$i\leftarrow i+1$\;
}
$V_i\leftarrow {\bf x}_T\setminus (V_1\cup \cdots\cup V_{i-1})$\;
\Return{$(V_0,\ldots,V_i)$}\;
\end{algorithm}

\medskip

\noindent{\bf Description of Algorithm~\ref{alg:weight}.}
The input of the algorithm is a triangulation $T$ of the $n$-gon. The output is an ordered partition of ${\bf x}_T$.
We repeatedly shrink $T$ to obtain a quiver associated to smaller triangulations.
More precisely, in step $i$ we identify the frozen vertices of $Q_T$ that are sinks, i.e.,~they have no outgoing arrows.
Then we add the corresponding arcs to the set $V_i$ and remove the corresponding triangles from $T$.
Note that in this process, we might remove edges whose corresponding vertices are not in $V_i$, but they are part of a triangle with an edge in $V_i$.

Let $V_0,\dots,V_i$ be the partition of the initial seed indices obtained by applying Algorithm~\ref{alg:weight} to the triangulation $T$ associated to the initial seed ${\bf x}_T$.
Refining the output of the algorithm yields an ordering on the variables of our initial seed, or equivalently the arcs of our initial triangulation.

\begin{Example}
We continue with our running example and compute a partition of the initial seed indices as $V_0=\{56,12,78\}$, $V_1=\{13,45\}$, $V_{2}=\{17,14,47\}$ and $V_{3}=\{57,34,23,67,18\}$.
\end{Example}

\begin{Lemma}\thlabel{lem:ex quad}
Given a triangulation $T$ consider the partition of ${\bf x}_T=V_0\cup\dots\cup V_s$ obtained from Algorithm~$\ref{alg:weight}$.
Let $v_{T}\colon \mathbb R^{2n-3}\to \mathbb R$ be the linear map defined by sending basis elements ${f}_{{ij}}$ with $\overline{ij}\in T$ to $q$, where $\overline{ij}\in V_q$.
Then $v_T$ is a quadrilateral linear map for $T$.
\end{Lemma}
\begin{proof}
Consider a quadrilateral $i<j<k<l$ in $T$ with $\overline{ij}$, $\overline{jk}$, $\overline{kl}$, $\overline{li}$ and $\overline{ik}\in T$.
As the algorithm ends when there is only one triangle or one arc left, there are two possibilities: either one of the triangles $\overline{ijk}$, $\overline{kli}$ is cut off first, or both triangles are cut off at the same time.
In the first case we may assume the triangle $\overline{ijk}$ is removed first.
Then $\overline{ij}\in V_p$, $\overline{lk}\in V_q$ and $\overline{jk}, \overline{il}\in V_s$ for some $p\le q\le s$ with $p<s$.
In the first case $p$ and $q$ are different, while in the second they are equal.
So in either case,
we have $v_T({f}_{ij}{+f}_{kl})=p+q<s+s=v_T({f}_{jk}+{f}_{il})$.
\end{proof}

\subsection{A distinguished maximal Gr\"obner cone}
\label{sec:monomial_degeneration}
In order to apply the construction from Section~\ref{sec:flat families} in this section we identify a particular maximal cone $C$ in $\GF(I_{2,n})$.
We analyze the cone $C$ and apply~\thref{thm:family}.
In the following section we~will show how this construction is related to adding universal coefficients to the cluster algebra~$A_{2,n}$.

\begin{Definition}\thlabel{def:mono wt}
Let $u\in \mathbb Q^{\binom{n}{2}}$ be the weight vector such that the weight of $p_{ij}$ is:
 \begin{displaymath}
 u(p_{ij}):= -\left(j-i-\frac{n}{2}\right)^2\qquad \text{for}\quad 1\le i<j\le n.
 \end{displaymath}
\end{Definition}

\begin{Example}\thlabel{exp:mono wt}
For $n=8$, the Pl\"ucker variables associated to boundary arcs receive the lowest weight, which is $-9$. Mutable Pl\"ucker variables have the following weights:
\begin{displaymath}
u(p_{ij})=\begin{cases}
0, &\text{if}\quad \min\{\vert j-i\vert, \vert i-j\vert\}=4,\\
-1, &\text{if}\quad \min\{\vert j-i\vert, \vert i-j\vert\}=3, \\
-4, &\text{if}\quad \min\{\vert j-i\vert, \vert i-j\vert\}=2.
\end{cases}
\end{displaymath}

\end{Example}

\begin{Lemma}\thlabel{lem: mono ideal}
With $u$ as in~\thref{def:mono wt} we have $\init_u(I_{2,n})=\mathcal M_{2,n}$
from~\thref{def:monomial ideal}.
\end{Lemma}
\begin{proof}
Consider the quadrilateral with vertices $i<j<k<l$ and its corresponding Pl\"ucker relation $R_{ijkl}=p_{ik}p_{jl}-p_{ij}p_{kl}-p_{il}p_{jk}$.
The arcs $\overline{ik}$ and $\overline{jl}$ are crossing, hence we have to show that ($a$) $\init_u(R_{ijkl})=p_{ik}p_{jl}$ and ($b$) the crossing monomials generate $\init_u(I_{2,n})$.
To prove ($a$) we verify that
\begin{gather*}
u(p_{ik}p_{jl}) = u(p_{il}p_{jk}) +2(j-i)(l-k) > u(p_{il}p_{jk}),
\\
u(p_{ik}p_{jl}) = u(p_{ij}p_{kl})+2(k-j)(n+i-l)>u(p_{ij}p_{kl}).
\end{gather*}

To establish ($b$), we apply Buchberger's criterion and show that all the $S$-pairs $S(R_{ijkl},R_{abcd})$ reduce to $0$ modulo $\{R_{ijkl}\}_{ 1\le i<j<k<l\le n}$.
If $\init_u(R_{ijkl})$ and $\init_u(R_{abcd})$ have no common factor, then $S(R_{ijkl},R_{abcd})$ reduces to zero (see, for example,~\cite[Lemma~2.3.1]{herzog2011monomial}).
Assuming that $\init_u(R_{ijkl})$ and $\init_u(R_{abcd})$ have a common factor we distinguish two cases, in each we underline the crossing monomials.
Consider $R_{ijls}$ and $R_{iklr}$, where without loss of generality $i<j<k<l<r<s$:
\begin{gather*}\begin{split}
 & S(R_{ijls},R_{iklr})
 = \underline{p_{kr}p_{ls}}p_{ij} + \underline{p_{kr}p_{jl}}p_{is} - \underline{p_{js}p_{ik}}p_{lr} - \underline{p_{js}p_{ir}}p_{kl}
 \\
& \hphantom{S(R_{ijls},R_{iklr})}{} = -p_{ij} R_{klrs}- p_{is} R_{jklr}+ p_{lr} R_{ijks}+ p_{kl}R_{ijrs}.
\end{split}
 \end{gather*}
The second case for $R_{ijkl}$ and $R_{jklm}$ with $i<j<k<l<m$ is similar.
\end{proof}

Let $C\in \GF(I_{2,n})$ be the maximal cone corresponding to the monomial initial ideal $\mathcal M_{2,n}$ and let $<$ denote the compatible monomial term order on $S$. 
By definition the standard monomial basis $\mathbb B_<$ (see~\thref{def:SMT}) is the set of monomials that are not contained in the ideal generated by \emph{crossing monomials}.
So we call it the {\it basis of non-crossing monomials}.

\begin{Proposition}\thlabel{cluster vs SM basis}
The basis of non-crossing monomials of $A_{2,n}$ coincides with the basis of cluster monomials $($see~\thref{thm:cluster mono basis}$)$.
\end{Proposition}

\begin{proof}
For every standard monomial $\bar{\bf p}^m$ with $m\in \mathbb Z_{\ge 0}^{\binom{n}{2}}$ we draw all arcs $\overline{ij}$ in the $n$-gon, for which $m_{ij}\not =0$.
By definition there is no pair of crossing arcs.
Hence, the set of arcs can be completed to a triangulation and $\bar{\bf p}^m$ is a cluster monomial for the corresponding seed. On the other hand, every cluster monomial is contained in $\mathbb B_{<}$ which completes the proof.
\end{proof}

\begin{Proposition}\thlabel{prop:unique mono ideal}
The monomial ideal $\mathcal M_{2,n}$ is the unique common monomial degeneration of the toric ideals $\init_{T}(I_{2,n})$ associated to triangulations $T$ constructed in~\thref{thm:toric wt ideals}.
\end{Proposition}
\begin{proof}
For every weight vector ${\bf w}_T$ associated with a triangulation $T$ and a quadrilateral linear map $w_T$ we have $\init_u\big(\init_{{\bf w}_T}(I_{2,n})\big)=\init_{{\bf w}_T}\big(\init_u(I_{2,n})\big)$ for $u$ as in~\thref{def:mono wt}.
Given $i<j<k<l$, we choose a triangulation $T$ containing the arcs $\overline{ij}$, $\overline{jk}$, $\overline{kl}$, $\overline{li}$ and $\overline{ik}$.
Then $\init_{{\bf w}_T}(R_{ijkl})=p_{ij}p_{kl}-p_{ik}p_{jl}$ by~\thref{prop:wt map} and~\thref{cor:trop and g}.
Now consider $T'$ obtained from $T$ by~flipping the arc $\overline{ik}$, so $T'$ contains $\overline{jl}$.
Then $\init_{{\bf w}_{T'}}(R_{ijkl})=-p_{ik}p_{jl}+p_{il}p_{jk}$.
So $-p_{ik}p_{jl}$ is the only monomial that simultaneously is an initial of $\init_{{\bf w}_{T}}(R_{ijkl})$ and $\init_{{\bf w}_{T'}}(R_{ijkl})$ for all $i<j<k<l$.
\end{proof}

As mentioned in~\thref{rmk:toric degen} there are various ways to associate a toric degeneration to a~triangulation or more generally a seed.
The most general construction is the principal coefficient construction introduced by Gross--Hacking--Keel--Kontsevich in~\cite[Section~8]{GHKK}.
We can now relate their construction to the Gr\"obner toric degenerations for $\Gr(2,\mathbb C^n)$ in~\thref{thm:toric wt ideals}.

\begin{Corollary}\thlabel{cor:central fibre}
The quotient $S/\init_{T}(I_{2,n})$ is isomorphic to the algebra of the semigroup gene\-rated by $\{{\bf g}_{ij}\colon 1\le i<j\le n\}\subset \mathbb Z^{2n-3}$.
In particular, the central fiber of the toric Gr\"obner degeneration induced by $T$ is isomorphic to that of the degeneration induced by principal coefficients at $T$.
\end{Corollary}
\begin{proof}
The quotient $S/\init_{T}(I_{2,n})$ has a vector space basis consisting of standard monomials with respect to $C$. By~\thref{cluster vs SM basis} this basis is in bijection with cluster monomials which are further in bijection with their ${\bf g}$-vectors.
In particular, $S/\init_{T}(I_{2,n})$ is a direct sum of cluster monomials with multiplication induced by the addition of their ${\bf g}$-vectors by~\thref{prop:wt map}.
Hence, it is isomorphic to the algebra of the semigroup generated by $\{{\bf g}_{ij}\colon 1\le i<j\le n\}$ $\subset \mathbb Z^{2n-3}$.

Moreover, by~\cite[Theorem~8.30]{GHKK} the central fiber of the toric degeneration induced by adding principal coefficients to the seed corresponding to $T$ is defined by the algebra of the semigroup generated by $\{{\bf g}_{ij}\colon 1\le i<j\le n\}\subset \mathbb Z^{2n-3}$.
\end{proof}

\begin{Theorem}\thlabel{thm:trop+}
In $\mathbb R^{\binom{n}{2}}$ we have $C\cap \trop(I_{2,n})=\trop^+(I_{2,n})$.
In particular, every $(d+1)$-dimensional face of $C$ corresponds to a planar trivalent tree and $\trop^+(I_{2,n})$ is the closed subfan whose support is the set of $(d+1)$-dimensional faces of the cone $C$.
\end{Theorem}

\begin{proof}
Every planar trivalent tree with $n$ leaves arises as the dual graph of a triangulation $T$ of the $n$-gon.
By~\thref{thm:toric wt ideals}, every weight vector ${\bf w}_T$ gives a toric initial ideal. Hence, ${\bf w}_T$ lies in the relative interior of a maximal (i.e.,~$(d+1)$-dimensional) cone $\sigma_T$ in $\trop^+(I_{2,n})$.
Furthermore, $\sigma_T$ is a face of $C$ by~\thref{prop:unique mono ideal}.
The opposite direction follows from the following claim.

\medskip\noindent
{\it Claim:} Let $\Upsilon$ be a trivalent tree with $n$ leaves. If $\init_{{\Upsilon}}(R_{ijkl})$ contains the monomial $p_{ik}p_{jl}$ for
 all $i<j<k<l $ then $\Upsilon$ is planar.

\medskip

\noindent
Let $\Upsilon$ be a non-planar trivalent tree.
Without loss of generality we assume that $1\le i<j<k<l\le n$ are labels of $\Upsilon$ appearing either in clockwise order: $i$, $k$, $j$, $l$ or in anticlockwise order: $i$,~$k$,~$j$,~$l$.
We prove the claim for the clockwise case, the anticlockwise case is analogous.
We~consider the full subtree with leaves $i$, $j$, $k$, $l$ (similar to the picture in Figure~\ref{fig:lem compute g-vect}, but ignoring the doubled leaves and exchanging $k$ and $j$).
Then $\init_{{\Upsilon}}(R_{ijkl})=p_{ij}p_{kl}+p_{il}p_{jk}$, a~con\-tra\-diction.
\end{proof}

\begin{samepage}
\begin{Corollary}\thlabel{cor:C simplicial}
The cone $C$ is defined by the lineality space $\mathcal L_{2,n}$ and additional $\binom{n}{2}-n$ rays $r_{ij}\in \mathbb R^{\binom{n}{2}}$ with $2\le i+1<j\le n$ defined by
\begin{displaymath}
(r_{ij})_{kl}=\begin{cases}
-1 & \text{if}\quad k\in [i+1,j]\not\ni l\quad \text{or}\quad k\not\in [i+1,j]\ni l,
\\
0 & \text{otherwise}.
\end{cases}
\end{displaymath}
Moreover, $\overline{C}$ 
is a rational simplicial cone whose faces correspond to collections of arcs in the $n$-gon.
\end{Corollary}
\end{samepage}

\begin{proof}
As we have identified a fan consisting of all $(d+1)$-dimensional (closed) faces of $C$, all the rays of $C$ (i.e.,~$(n+1)$-dimensional faces, where $n=\dim\mathcal L_{2,n}$) are also rays of $\trop^+(I_{2,n})$.
By the combinatorial description of $\trop(I_{2,n})$ 
from~\cite{SS04}, we know that the rays correspond to trees with only one interior edge (corresponding to a partition of $[n]$ into two sets).
The rays of the cones corresponding to planar trivalent trees are therefore in correspondence with partitions of $[n]$ into two cyclic intervals.
To a (non-trivalent) planar tree we associate a weight vector as in~\eqref{eq:tree wt}.
\end{proof}

Let ${\bf r}$ denote the matrix whose rows are $\frac{1}{2}r_{ij}$ for all $2\le i+1<j\le n$ and recall the lifting of elements from~\thref{def:lift} (and~\thref{rmk:non primitive rays}).
In the following we denote $\tilde f_{\bf r}$ by $\tilde f$ for all~$f\in S$.

\begin{Proposition}\thlabel{prop:univ and groebner}
For all $i<j<k<l$ we have that $\tilde R_{ijkl}=R_{ijkl}({\bf t})$, where
\begin{displaymath}
 R_{ijkl}({\bf t}):= -p_{ik}p_{jl} + p_{il}p_{jk}\prod_{a\in[i,j-1],\,b\in[k,l-1]} t_{ab} + p_{ij}p_{kl}\prod_{a\in[j,k-1],\,b\in[l,i-1]}t_{ab}.
\end{displaymath}
\end{Proposition}

\begin{proof}
We show that every variable $t_{ab}$ occurs with the same exponent in $\tilde R_{ijkl}$ and $R_{ijkl}({\bf t})$.
For simplicity, we adopt the notation $R_{ijkl}=-{\bf p}^{e_{ik}+e_{jl}}+{\bf p}^{e_{il}+e_{jk}} + {\bf p}^{e_{ij}+e_{kl}}$.
To compute the exponents of a variable $t_{ab}$ in $\tilde R_{ijkl}$ we have to distinguish several cases. We give the proof for only two of them as all others are similar.
\medskip

\noindent{\it Case 1.} Assume that $i\in [a+1,b]\not\ni j,k,l$. Note that due to symmetries this is equivalent to assuming $i\not \in [a+1,b]\ni j,k,l$. Then
$\frac{1}{2}r_{ab}\cdot (e_{ik}+e_{jl})=\frac{1}{2}r_{ab}\cdot (e_{il}+e_{jk})=\frac{1}{2}r_{ab}\cdot (e_{ij}+e_{kl})=-\frac{1}{2}$.
Hence, in $\tilde R_{ijkl}$ the variable $t_{ab}$ does not appear. One can see that $t_{ab}$ also does not appear in $R_{ijkl}({\bf t})$, since neither $b\in[l,i-1]$ nor $b\in[k,l-1]$.
\medskip

\noindent{\it Case 2.} Assume that $i,j\in [a+1,b]\not\ni k,l$. Note that this is equivalent to assuming $i,j\not \in [a+1,b]\ni k,l$. Then $\frac{1}{2}r_{ab}\cdot (e_{ik}+e_{jl})=\frac{1}{2}r_{ab}\cdot (e_{il}+e_{jk})=-1 \ \ \text{ and }\ \ \frac{1}{2}r_{ab}\cdot (e_{ij}+e_{kl})=0$.
Hence, in~$\tilde R_{ijkl}$ the variable $t_{ab}$ does only appear with exponent 1 in the coefficients of the monomial $p_{ij}p_{kl}$.
As~we assumed $i,j\in [a+1,b]\not\ni k,l$ it follows that $a\in [l,i-1]$ and $b\in [j,k-1]$. So in~$R_{ijkl}({\bf t})$ the variable $t_{ab}$ appears also with exponent $1$ in the coefficient of the monomial~$p_{ij}p_{kl}$.
\end{proof}

\begin{Corollary}\thlabel{2n SMB}
The algebra $\tilde A_{2,n}=S[t_{ij}\colon 2\le i+1<j\le n]/\tilde I_{2,n}$ is a free $\mathbb C[t_{ij}\colon 2\le i+1<j\le n]$-module with basis given by the cluster monomials.
Moreover, fibers of the flat family $\pi\colon {\rm Proj}(\tilde A_{2,n})\to \mathbb A^{\frac{n(n-3)}{2}}$ are in correspondence with collections of arcs in the $n$-gon.
\end{Corollary}

\begin{proof}
The first part is a direct corollary of~\thref{prop:univ and groebner},~\thref{thm:family}$(i)$ and
\thref{cluster vs SM basis}.
The second part follows from~\thref{thm:family}$(ii)$ and~\thref{cor:C simplicial}.
\end{proof}

We are now prepared to state and prove our main result for $\Gr(2,\mathbb C^n)$.

\begin{Theorem}\thlabel{thm:A univ as quotient}
The Pl\"ucker algebra with universal coefficients $A_{2,n}^{\rm univ}$ is canonically isomorphic to the quotient $\tilde A_{2,n}=S[t_{ij}\colon 2\le i+1<j\le n]/\tilde I_{2,n}$.
\end{Theorem}

\begin{proof}
Recall that $A_{2,n}^{\rm univ}$ has the same set of (frozen and mutable) cluster variables as $A_{2,n}$, namely $\{\bar p_{ij}\colon 1\le i<j\le n\}$, and additionally coefficients $\{y_{ij}\colon 2\le i+1<j\le n\}$.
We define
\begin{displaymath}\begin{split}
\Psi\colon\quad S[t_{ij}\colon 2\le i+1<j\le n] &\to A_{2,n}^{\rm univ},\\
 p_{ij} &\mapsto \bar p_{ij} \qquad \text{for}\quad 1\le i<j\le n, \\
t_{ij} &\mapsto y_{ij}\qquad \text{for}\quad 2\le i+1<j\le n.
\end{split}
\end{displaymath}
This morphism of $\mathbb C$-algebras induces the desired isomorphism between $A_{2,n}^{\rm univ}$ and $\tilde A_{2,n}$.
By~Pro\-po\-si\-tion~\ref{prop:lifted generators} the lifts of Pl\"ucker relations $\tilde R_{ijkl}$ generate the lifted ideal $\tilde I_{2,n}$.
We proceed by showing that $\Psi\big(\tilde R_{ijkl}\big)$ is the corresponding exchange relations in $A^{\rm univ}_{2,n}$.
Since the mutable parts of $B^{\rm univ}_{(Q_T,F_T)}$ and $B_{(Q_T,F_T)}$ coincide for every triangulation $T$ there is a natural bijection between cluster monomials of $A_{2,n}$ and $A_{2,n}^{\rm univ}$.
It is the only bijection that sends the initial cluster variables of $B_{(Q_T,F_T)}$ to those of $B^{\rm univ}_{(Q_T,F_T)}$ and commutes with mutation.
Further, it induces a bijection between the exchange relations associated to $B_{(Q_T,F_T)}$ and $B^{\rm univ}_{(Q_T,F_T)}$, which yields bijections between the sets:
\begin{displaymath}
 \left\{\begin{matrix}
 \tilde R_{ijkl}\in \tilde I_{2,n} \text{ with} \\
 i,j,k,l\in [n],\\
 i<j<k<l
 \end{matrix}\right\}
\longleftrightarrow
 \left\{\begin{matrix}
 \text{quadrilaterals with vertices } \\
 i<j<k<l \text{ in the $n$-gon}
 \end{matrix}\right\}
\longleftrightarrow
 \left\{\begin{matrix}
 \text{exchange relations}\\
 \text{of } A_{2,n}^{\rm univ}
 \end{matrix}\right\}\!.
\end{displaymath}
Consider a quadrilateral with vertices $i<j<k<l$ in the $n$-gon and fix a triangulation $T$ in which this quadrilateral occurs.
So without loss of generality we have $\overline{ij},\overline{ik}, \overline{il}, \overline{jk}, \overline{kl}\in T$ (see, e.g.,~left side of Figure~\ref{fig:ex rel Auniv}).
The exchange relation of $A_{2,n}^{\rm univ}$ associated with the quadrilateral $i<j<k<l$ is of form:\vspace{-1ex}
\begin{displaymath}
 E^{\rm univ}_{ijkl}:= - \bar p_{ik}\bar p_{jl} + \bar p_{ij}\bar p_{kl} \prod_{({\bf g}^\vee_{ab})_{ik}=-1} y_{ab} + \bar p_{il}\bar p_{jk} \prod_{({\bf g}^\vee_{ab})_{ik}=+1} y_{ab},
\end{displaymath}

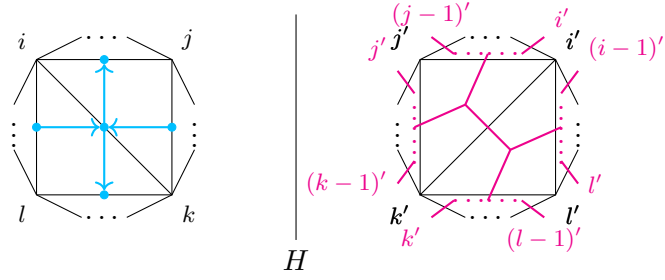
\begin{figure}
\centering
\begin{tikzpicture}[scale=.3]
\draw (0,0) -- (0,6) -- (6,6) -- (6,0) -- (0,0);
\draw (0,6) -- (6,0);
\draw (-1,4) -- (0,6) -- (2,7);
\draw (4,7) -- (6,6) -- (7,4);
\draw (4,-1) -- (6,0) -- (7,2);
\draw (-1,2) -- (0,0) -- (2,-1);

\draw[fill,capri] (3,3) circle (.175cm);
\draw[fill,capri] (3,6) circle (.175cm);
\draw[fill,capri] (6,3) circle (.175cm);
\draw[fill,capri] (3,0) circle (.175cm);
\draw[fill,capri] (0,3) circle (.175cm);
\draw[thick,capri,->] (3,3) -- (3,5.8);
\draw[thick,capri,->] (3,3) -- (3,.2);
\draw[thick,capri,->] (0,3) -- (2.8,3);
\draw[thick,capri,->] (6,3) -- (3.2,3);

\node at (-1,3) {$\vdots$};
\node at (3,7) {$\dots$};
\node at (7,3) {$\vdots$};
\node at (3,-1) {$\dots$};
\node[above left] at (0,6) {\footnotesize{$i$}};
\node[above right] at (6,6) {\footnotesize{$j$}};
\node[below right] at (6,0) {\footnotesize{$k$}};
\node[below left] at (0,0) {\footnotesize{$l$}};

\draw (11.5, 8) -- (11.5,-2);
\node[below] at (11.5,-2) {$H$};

\begin{scope}[xshift=17cm]
\draw (0,0) -- (0,6) -- (6,6) -- (6,0) -- (0,0) -- (6,6);
\draw (-1,4) -- (0,6) -- (2,7);
\draw (4,7) -- (6,6) -- (7,4);
\draw (4,-1) -- (6,0) -- (7,2);
\draw (-1,2) -- (0,0) -- (2,-1);
\node at (-1,3) {$\vdots$};
\node at (3,7) {$\dots$};
\node at (7,3) {$\vdots$};
\node at (3,-1) {$\dots$};
\node[above left] at (0,6) {\footnotesize{$j'$}};
\node[above right] at (6,6) {\footnotesize{$i'$}};
\node[below right] at (6,0) {\footnotesize{$l'$}};
\node[below left] at (0,0) {\footnotesize{$k'$}};

\draw[thick,magenta] (.5,7) -- (1.5,6.25);
\node[above,magenta] at (.5,7) {\footnotesize$(j-1)'$};
\draw[thick,magenta] (4.5,6.25) -- (5.5,7);
\node[above right,magenta] at (5.5,7) {\footnotesize$i'$};
\draw[thick,magenta] (6.25,4.5) -- (7,5.5);
\node[above right,magenta] at (7,5.5) {\footnotesize $(i-1)'$};
\draw[thick,magenta] (6.25,1.5) -- (7,.5);
\node[right,magenta] at (7,.5) {\footnotesize $l'$};
\draw[thick,magenta] (4.5,-.25) -- (5.5,-1);
\node[below,magenta] at (5.5,-1) {\footnotesize $(l-1)'$};
\draw[thick,magenta] (1.5,-.25) -- (.5,-1);
\node[below left,magenta] at (.5,-1) {\footnotesize $k'$};
\draw[thick,magenta] (-.25,1.5) -- (-1,.5);
\node[left,magenta] at (-1,.5) {\footnotesize $(k-1)'$};
\draw[thick,magenta] (-.25,4.5) -- (-1,5.5);
\node[above left,magenta] at (-1,5.5) {\footnotesize $j'$};

\draw[thick,magenta] (-.25,3) -- (2,4) -- (3,6.25);
\draw[thick,magenta] (3,-.25) -- (4,2) -- (6.25,3);
\draw[thick,magenta] (2,4) -- (4,2);
\node[magenta] at (3.75,6.25) {$\dots$};
\node[magenta] at (2.25,6.25) {$\dots$};
\node[magenta] at (6.25,3.75) {$\vdots$};
\node[magenta] at (6.25,2.25) {$\vdots$};
\node[magenta] at (2.25,-.25) {$\dots$};
\node[magenta] at (3.75,-.25) {$\dots$};
\node[magenta] at (-.25,2.25) {$\vdots$};
\node[magenta] at (-.25,3.75) {$\vdots$};

\node[above left] at (0,6) {\footnotesize{$j'$}};
\node[above right] at (6,6) {\footnotesize{$i'$}};
\node[below right] at (6,0) {\footnotesize{$l'$}};
\node[below left] at (0,0) {\footnotesize{$k'$}};
\end{scope}
\end{tikzpicture}
 \caption{A quadrilateral $i<j<k<l$ in a triangulation $T$ and the reflected triangulation $T^\vee$ from which truncated ${\bf g}$-vectors with respect to $\big(Q_T^{\rm mut}\big)^{\rm op}$ (i.e.,~columns of $U_{Q_T^{\rm mut}}$) can be read off.}
 \label{fig:ex rel Auniv}
\end{figure}\noindent
where ${\bf g}^\vee_{ab}$ is the $ab^{\text{th}}$ row of $U_{Q_T^{\rm mut}}$ and $\big({\bf g}^\vee_{ab}\big)_{ik}$ is its entry in the column of $B_{(Q_T,F_T)}^{\rm univ}$ corresponding to the mutable variable $\bar p_{ik}$.
Hence, we need to compute those rows of $U_{Q_T^{\rm mut}}$ with non-zero entries in the $ik^{\text{th}}$ column.
To do so, we embed $T$ into $\mathbb R^2$ and consider a hyperplane $H\subset \mathbb R^2$ which does not intersect $T$.
Let $T^\vee$ be the image of $T$ under the reflection $s_H$ with respect to $H$ and denote $m':=s_H(m)$ for all $m\in [n]$.
Naturally, we have $Q_{T^\vee}=(Q_T)^{\rm op}$, so $Q^{\rm mut}_{T^\vee}=\big(Q_T^{\rm mut}\big)^{\rm op}$.
Using the right side of Figure~\ref{fig:ex rel Auniv} we compute the $ik^{\text{th}}$ entry of truncated ${\bf g}$-vectors with respect to $Q^{\rm mut}_{T^{\vee}}$:\vspace{-1ex}
\begin{equation}\label{eq:signs gv}
({\bf g}^\vee_{ab})_{ik} =
\begin{cases}
 +1 & \text{if}\quad a\in[i,j-1]\quad \text{and}\quad b\in [k,l-1]\quad \text{(or vice versa)},\\
 -1 & \text{if}\quad a\in[j,k-1]\quad \text{and}\quad b\in [l,i-1]\quad \text{(or vice versa)},\\
 0 & \text{otherwise}.
\end{cases}
\end{equation}
We compute:\vspace{-1ex}
\begin{displaymath}
\Psi\big(\tilde R_{ijkl}\big)
= - \bar p_{ik}\bar p_{jl} + \bar p_{il}\bar p_{jk} \prod_{a\in[i,j-1],\,b\in[k,l-1]} y_{ab} + \bar p_{ij}\bar p_{kl} \prod_{a\in[j,k-1],\,b\in[l,i-1]} y_{ab} \overset{\eqref{eq:signs gv}}{=} E_{ijkl}^{\rm univ}.
\end{displaymath}
In particular, $\Psi$ induces a surjective map $\bar \Psi\colon \tilde A_{2,n} \to A_{2,n}^{\rm univ}$.
By~\thref{2n SMB}, $\tilde A_{2,n}$ is a free $\mathbb C[t_{ij}]$-module whose basis are the cluster monomials.
Similarly, after identifying $t_{ij}$, $2\le i+1<j\le n$, with $\Psi(t_{ij})=y_{ij}$, $A_{2,n}^{\rm univ}$ is a free $\mathbb C[t_{ij}]$-module with basis given by the cluster monomials by~\cite[Theorem~0.3 and p.~502]{GHKK}.
Hence, $\bar \Psi$ is also injective and the claim follows.
\end{proof}

As $A_{2,n}^{\rm univ}$ is by definition a domain the following is now a direct consequence.

\begin{Corollary}\thlabel{prop:I2n prime}
The ideal $\tilde I_{2,n}\subset S[{\bf t}]$ is prime.
\end{Corollary}

\begin{Example}
We list the lifted Pl\"ucker relations, respectively the exchange relations of $A_{2,5}^{\rm univ}$, associated to our running example. These relations also constitute a Gr\"obner basis for $\tilde I_{2,5}$, the crossing monomial of each relation is the first one. As~\thref{lem:lifts} predicts, this is the only term with coefficient in $\mathbb C$. Pl\"ucker variables of frozen cluster variables are marked in \textcolor{blue}{blue}:
\begin{displaymath}
\begin{matrix*}[l]
\tilde R_{1234} & = & p_{13}\,p_{24} & - & \textcolor{blue}{p_{12}\,p_{34}}\, t_{24}\,t_{25} & - & p_{14}\,\textcolor{blue}{p_{23}}\,t_{13}, \\
\tilde R_{1235} & = & p_{13}\,p_{25} & - & \textcolor{blue}{p_{15}\,p_{23}} \,t_{13}\, t_{14} & - & \textcolor{blue}{p_{12}}\,p_{35}\,t_{25}, \\
\tilde R_{1245} & = & p_{14}\,p_{25} & - & \textcolor{blue}{p_{12}\,p_{45}}\, t_{25}\, t_{35} & - & \textcolor{blue}{p_{15}}\,p_{24} \,t_{14}, \\
\tilde R_{1345} & = & p_{14}\,p_{35} & - & \textcolor{blue}{p_{15}\,p_{34}}\, t_{14}t_{24} & - & p_{13}\,\textcolor{blue}{p_{45}}\, t_{35}, \\
\tilde R_{2345} & = & p_{24}\,p_{35} & - & \textcolor{blue}{p_{23}\,p_{45}} \,t_{13}\,t_{35} & - & p_{25}\,\textcolor{blue}{p_{34}}\, t_{24}.
\end{matrix*}
\end{displaymath}
\end{Example}

\subsection[The Grassmannian of 3-planes in 6-space]
{The Grassmannian $\boldsymbol{\Gr\big(3,\mathbb C^6\big)}$}\label{sec:Gr(3,6)}

We now turn to the case of $\Gr\big(3,\mathbb C^6\big)$ and prove an analogue of~\thref{thm:A univ as quotient}. To highlight various important differences between this case and the case of $\Gr(2,\mathbb C^n)$, we focus more on~explicit computations.
We believe that the explicit computations help to understand the difficulties that may arise when studying other compactifications of finite type cluster varieties such as $\Gr(3,\mathbb C^7)$, $\Gr(3,\mathbb C^8)$ or (skew-)Schubert varieties inside Grassmannians as in~\cite{SSBW}.

Denote by $A_{3,6}$ the homogeneous coordinate ring of the Grassmannian $\Gr\big(3,\mathbb C^6\big)$ with res\-pect to its Pl\"ucker embedding.
We use the cluster structure on $A_{3,6}$ to consider $\Gr\big(3,\mathbb C^6\big) $ as a~weighted projective variety
as follows.
The cluster algebra $A_{3,6}$ has 22 cluster variables, out of which 20 are the Pl\"ucker coordinates and the 2 additional elements are $\bar X=\bar p_{145}\bar p_{236}-\textcolor{blue}{\bar p_{123}}\textcolor{blue}{\bar p_{456}}$ and $\bar Y=\bar p_{125}\bar p_{346}-\textcolor{blue}{\bar p_{126}}\textcolor{blue}{\bar p_{345}}$, see~\cite[Theorem~6]{Sco06}.
We systematically write the frozen Pl\"ucker coordinates in {\color{blue}blue}.
We list the 22 cluster variables of $A_{3,6}$ in the following order and fix this order for later use
\begin{gather}\nonumber
\textcolor{blue}{\bar p_{123}},\;\bar p_{124},\;\bar p_{125},\; \textcolor{blue}{\bar p_{126}},\;\bar p_{134},\;\bar p_{135},\;\bar p_{136},\;\bar p_{145},\;\bar p_{146},\;\textcolor{blue}{\bar p_{156}},\;
\textcolor{blue}{\bar p_{234}},\;\bar p_{235},\;\bar p_{236},\;
\bar p_{245},
\\ \qquad
\bar p_{246},\;\bar p_{256},\; \textcolor{blue}{\bar p_{345}},\;\bar p_{346},\;\bar p_{356},\;\textcolor{blue}{\bar p_{456}},\;\bar X,\;\bar Y.
\label{eq:order variables}
\end{gather}
Every seed $s$ of $A_{3,6}$ consists of the six frozen variables $\textcolor{blue}{\bar p_{123}}$, $\textcolor{blue}{\bar p_{234}}$, $\textcolor{blue}{\bar p_{345}}$, $\textcolor{blue}{\bar p_{456}}$, $\textcolor{blue}{\bar p_{156}}$, $\textcolor{blue}{\bar p_{126}}$
and four additional mutable cluster variables.
Consider the following bijection that sends the Pl\"ucker variable $p_{ijk}$ to the Pl\"ucker coordinate $\bar p_{ijk}\in A_{3,6}$ and $X$ to $\bar X\in A_{3,6}$ as well as $Y$ to $\bar Y\in A_{3,6}$:
\begin{equation}\label{eq:identification}
 \{p_{123}, \dots , p_{456}, X,Y \} \longleftrightarrow \{ \text{cluster variables of }A_{3,6}\}.
\end{equation}

Denote by ${\bf d}\in \mathbb{Z}^{22}$ the vector $(1, \dots , 1, 2,2)$. That is, the first $20$ entries of ${\bf d}$ are $1$ and the last two are $2$.
The bijection (\ref{eq:identification}) induces a surjective map
\begin{equation}\label{eq:psi36}
\psi\colon\ \mathbb{C}_{\bf d}[p_{123}, \dots , p_{456}, X,Y ] \twoheadrightarrow A_{3,6}.
\end{equation}
Let $I^{\rm ex}:=\ker(\psi)$, so we obtain a ring isomorphism $A:=\mathbb C_{\bf d}[p_{123}, \dots , p_{456}, X,Y]/I^{\rm ex}\cong A_{3,6} $.
Moreover, the weig\-h\-ted projective variety $V(I^{\rm ex})\subset \mathbb P({\bf d})$ is isomorphic to $\Gr\big(3,\mathbb C^6\big)$ as a~weig\-h\-ted projective variety. Indeed, ${\bf d}$ was chosen so that every exchange relation in the cluster structure of $A_{3,6}$ is identified with a {\bf d}-homogeneous element of $\mathbb{C}_{\bf d}[p_{123}, \dots , p_{456}, X,Y]$ and $I^{\rm ex }$ is prime (hence, radical) since $A_{3,6} $ is a domain.
Therefore, $I(V(I^{\rm ex}))= I^{\rm ex}$ and $\psi$ induces an isomorphism of $\mathbb Z$-graded rings $S(V(I^{\rm ex}))\to A_{3,6}$.
One can verify (e.g.,~using \texttt{Macaulay2}~\cite{M2}) that the ideal $I^{\rm ex}$ is generated by all three-term Pl\"ucker relations and the following seven additional relations:
\begin{gather}
p_{145}p_{236} - \textcolor{blue}{p_{123}}\textcolor{blue}{p_{456}} - X,\qquad
p_{136}p_{245} - \textcolor{blue}{p_{126}}\textcolor{blue}{p_{345}} - X,\qquad
p_{146}p_{235} - \textcolor{blue}{p_{156}}\textcolor{blue}{p_{234}} - X,\nonumber
\\
p_{124}p_{356} - \textcolor{blue}{p_{123}}\textcolor{blue}{p_{456}} - Y,\qquad
p_{125}p_{346} - \textcolor{blue}{p_{126}}\textcolor{blue}{p_{345}} - Y,\qquad
p_{134}p_{256} - \textcolor{blue}{p_{156}}\textcolor{blue}{p_{234}} - Y,\nonumber
\\
\label{eq:4terms}
 p_{135}p_{246} - p_{134}p_{256} - \textcolor{blue}{p_{126}}\textcolor{blue}{p_{345}} - p_{145}p_{236}.
\end{gather}
Note that with exception of the last relation all of them correspond to exchange relations in~$A_{3,6}$.
We study the Gr\"obner fan of $I^{\rm ex}$. It contains the lineality space $\mathcal L(I^{\rm ex})$ generated by
\begin{gather*}
\ell_1 =( 1 , 1 , 1 , 1 , 1 , 1 , 1 , 1 , 1 , 1 , 0 , 0 , 0 , 0 , 0 , 0 , 0 , 0 , 0 , 0 , 1 , 1 ),\\
\ell_2 =( 1 , 1 , 1 , 1 , 0 , 0 , 0 , 0 , 0 , 0 , 1 , 1 , 1 , 1 , 1 , 1 , 0 , 0 , 0 , 0 , 1 , 1 ),\\
\ell_3 =( 1 , 0 , 0 , 0 , 1 , 1 , 1 , 0 , 0 , 0 , 1 , 1 , 1 , 0 , 0 , 0 , 1 , 1 , 1 , 0 , 1 , 1 ),\\
\ell_4 =( 0 , 1 , 0 , 0 , 1 , 0 , 0 , 1 , 1 , 0 , 1 , 0 , 0 , 1 , 1 , 0 , 1 , 1 , 0 , 1 , 1 , 1 ),\\
\ell_5 =( 0 , 0 , 1 , 0 , 0 , 1 , 0 , 1 , 0 , 1 , 0 , 1 , 0 , 1 , 0 , 1 , 1 , 0 , 1 , 1 , 1 , 1 ),\\
\ell_6 =( 0 , 0 , 0 , 1 , 0 , 0 , 1 , 0 , 1 , 1 , 0 , 0 , 1 , 0 , 1 , 1 , 0 , 1 , 1 , 1 , 1 , 1 ).
\end{gather*}
The order of the entries corresponds to the order on cluster variables of $A_{3,6}$ in~\eqref{eq:order variables}.
Note that ${\bf d}=\frac{1}{3}(\ell_1+\dots+\ell_6)$.
In the Gr\"obner fan $\GF(I^{\rm ex})$ we identify a maximal cone $C$ and consider its image $\overline{C}\subset \GF(I^{\rm ex})/\mathcal L(I^{\rm ex})$.
The cone $\overline{C}$ is strongly convex (by~\thref{lem:strictly_convex}) and simplicial.
We choose the following representatives of the 16 minimal ray generators of $\overline{C}$:
\begin{gather}
 r_{1}=(0, 0, 0, 1, 0, 0, 0, 0, 0, 0, 0, 0, 0, 0, 0, 0, 0, 0, 0, 0, 0, 0),\nonumber\\
 r_{2}=(1, 1, 2, 1, 1, 2, 1, 2, 1, 2, 1, 1, 0, 1, 0, 1, 1, 0, 1, 1, 2, 2),\nonumber\\
 r_{3}=(0, 0, 1, 1, 0, 0, 0, 0, 0, 1, 0, 0, 0, 0, 0, 1, 0, 0, 0, 0, 0, 1),\nonumber\\
 r_{4}=(1, 0, 1, 2, 0, 1, 2, 0, 1, 2, 1, 2, 2, 1, 1, 2, 1, 1, 2, 1, 3, 2),\nonumber\\
 r_{5}=(1, 1, 1, 1, 1, 1, 1, 2, 1, 1, 1, 1, 0, 1, 0, 0, 1, 0, 0, 1, 2, 1),\nonumber\\
 r_{6}=(1, 1, 1, 0, 1, 1, 0, 2, 1, 1, 1, 1, 0, 1, 0, 0, 1, 0, 0, 1, 2, 1),\nonumber\\
 r_{7}=(0, 0, 0, 0, 0, 0, 0, 0, 0, 1, 0, 0, 0, 0, 0, 0, 0, 0, 0, 0, 0, 0),\nonumber\\
 r_{8}=(1, 2, 1, 1, 1, 0, 0, 1, 1, 0, 1, 0, 0, 1, 1, 0, 1, 1, 0, 1, 1, 2),\nonumber\\
 r_{9}=(2, 3, 2, 1, 3, 2, 1, 3, 2, 1, 2, 1, 0, 2, 1, 0, 2, 1, 0, 2, 3, 3),\nonumber\\
r_{10}=(2, 3, 2, 1, 3, 2, 1, 3, 2, 2, 2, 1, 0, 2, 1, 1, 2, 1, 1, 2, 3, 4),\nonumber\\
r_{11}=(2, 3, 3, 2, 3, 2, 1, 3, 2, 2, 2, 1, 0, 2, 1, 1, 2, 1, 1, 2, 3, 4),\nonumber\\
r_{12}=(1, 1, 1, 2, 0, 0, 1, 0, 1, 1, 1, 1, 2, 1, 2, 2, 1, 1, 1, 1, 2, 2),\nonumber\\
r_{13}=(1, 1, 0, 1, 1, 0, 1, 1, 1, 0, 1, 0, 1, 1, 1, 0, 1, 1, 0, 1, 2, 1),\nonumber\\
r_{14}=(1, 0, 0, 1, 0, 0, 1, 0, 1, 1, 1, 1, 2, 1, 1, 1, 1, 1, 1, 1, 2, 1),\nonumber\\
r_{15}=(1, 1, 0, 0, 1, 0, 0, 1, 1, 0, 1, 0, 0, 1, 1, 0, 1, 1, 0, 1, 1, 1),\nonumber\\
r_{16}=(1, 2, 1, 1, 1, 0, 0, 1, 1, 1, 1, 0, 0, 1, 1, 1, 1, 1, 0, 1, 1, 2).\label{eq:rays of C}
\end{gather}
From now on, we denote by ${\bf r}$ the $(16\times 22)$-matrix with rows $r_1,\dots,r_{16}$.
Recall that a monomial is called {\bf squarefree} if the exponent of each variable in it is either $0$ or $1$.

\begin{Theorem}\thlabel{thm:Gr36}
There exists a unique maximal cone $C\in\GF(I^{\rm ex})$ with the following properties:
\begin{itemize}\itemsep=0pt
 \item[$(i)$] the associated initial ideal $\init_C(I^{\rm ex})$ is generated by squarefree monomials of degree $2$ and it contains all exchange monomials;
\item[$(ii)$] the free $\mathbb C[t_1,\dots,t_{16}]$-algebra associated to $C$ and $\bf r$ defined in~\thref{def:lift} has the pro\-perty
 \[
 \tilde A_{\bf r}:=\mathbb C[t_1,\dots,t_{16}][p_{123}, \dots , p_{456}, X,Y]/\tilde I^{\rm ex}_{\bf r} \cong A_{3,6}^{\rm univ};
 \]
\item[$(iii)$] for every seed $s$ in the cluster structure of $A_{3,6}$ the cone $C$ has a 10-dimensional face~$\tau_s$ whose associated initial ideal $\init_{\tau_s}(I^{\rm ex})$ is a totally positive binomial prime ideal $($hence~$\tau_s\in \trop^+(I^{\rm ex}))$.
 Moreover, $C\cap \trop(I^{\rm ex})=\trop^+(I^{\rm ex})$.
\end{itemize}
\end{Theorem}

Before we prove~\thref{thm:Gr36} we explain the conventions we use to describe the exchange relations of $A_{3,6}^{\rm univ}$.
The algebra $A_{3,6}^{\rm univ}$ is of cluster type $\mathtt D_4$.
In particular, as explained in~Section~\ref{sec:cluster}, we label the coefficients of $A_{3,6}^{\rm univ}$ with the set of almost positive coroots $\Phi^{\vee}_{\geq -1} $ in the root system dual to a root system of type $\mathtt{D}_4$.
For this we fix an initial seed for $A_{3,6}^{\rm univ}$ such that the mutable part of its quiver is a bipartite orientation of $\mathtt{D}_4$.
We choose the seed that contains the mutable variables $p_{246}$, $p_{346}$, $p_{124}$ and $p_{256}$ together with the frozen variables. The part of~the quiver that involves only the vertices corresponding to these variables is the following:
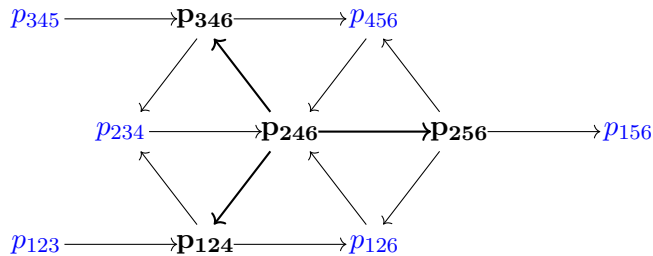
\begin{center}
 \begin{tikzpicture}[scale=.75]
 \node at (0,0) {$\mathbf{p_{246}}$};
 \node at (3,0) {$\mathbf{p_{256}}$};
 \node at (-1.5,2) {$\mathbf{p_{346}}$};
 \node at (-1.5,-2) {$\mathbf{p_{124}}$};
 \node[blue] at (1.5,-2) {$p_{126}$};
 \node[blue] at (1.5,2) {$p_{456}$};
 \node[blue] at (-3,0) {$p_{234}$};
 \node[blue] at (6,0) {$p_{156}$};
 \node[blue] at (-4.5,2) {$p_{345}$};
 \node[blue] at (-4.5,-2) {$p_{123}$};

 \draw[->] (-4,2) -- (-2,2);
 \draw[->] (-1,2) -- (1,2);
 \draw[->] (-2.5,0) -- (-.5,0);
 \draw[->,thick] (.5,0) -- (2.5,0);
 \draw[->] (3.5,0) -- (5.5,0);
 \draw[->] (-4,-2) -- (-2,-2);
 \draw[->] (-1,-2) -- (1,-2);
 \draw[->] (2.65,-.35) -- (1.65,-1.65);
 \draw[->] (1.35,-1.65) -- (.35,-.35);
 \draw[->,thick] (-.35,-.35) -- (-1.35,-1.65);
 \draw[->] (-1.65,-1.65) -- (-2.65,-.35);
 \draw[->] (2.65,.35) -- (1.65,1.65);
 \draw[->] (1.35,1.65) -- (.35,.35);
 \draw[->,thick] (-.35,.35) -- (-1.35,1.65);
 \draw[->] (-1.65,1.65) -- (-2.65,.35);
 \end{tikzpicture}
\end{center}
The coefficients are labeled by $\Phi^{\vee}_{\geq -1}$ and can be realized as frozen vertices of the quiver.
In order to compute the arrows between the coefficient vertices and the vertices corresponding to cluster variables we identify the mutable vertices of this quiver with $\{ 1, 2, 3, 4 \}$:
let $p_{246}$ correspond to~1, $p_{346}$ to 2, $p_{124}$ to 3 and $p_{256}$ to 4.
Now~\thref{prop:bipartite_univ_Q} contains all the information necessary to compute the arrows.
As the resulting quiver is rather complicated, we refrain from visualizing it here.
It is available for download on the homepage~\cite{homepage} in a format that can directly be opened in quiver mutation app~\cite{Qmut}.
Finally, we use the quiver mutation app to compute all exchange relations. They can be found explicitly in the Appendix~\ref{sec:data Gr36}.

\medskip

\looseness=1
One more ingredient we need for the proof of~\thref{thm:Gr36} is the basis of cluster monomials for $ A^{\rm univ}_{3,6}$.
By~\cite[Theorem~0.3 and p.~502]{GHKK} cluster monomials form a $\mathbb C[y_\alpha\colon \alpha\in \Phi^{\vee}_{\ge -1}]$-basis for $A_{3,6}^{\rm univ}$.
If $x$ and $x'$ are cluster variables that do not occur together in any seed, then any monomial divisible by $xx'$ cannot be a cluster monomial.
In fact, using~\cite[Theorem~7.12(b)]{Kel12} this gives us the following characterization of cluster monomials.
Write $\bar {\bf x}^{a}\in A_{3,6}^{\rm univ}$, $a\in \mathbb Z_{\ge 0}^{22}$, to denote a monomial in the (mutable and frozen) cluster variables $\bar p_{123},\dots, \bar p_{456},\bar X,\bar Y$. Then
\begin{gather}\label{eq:basis A36univ}
\bar {\bf x}^{a}\in A_{3,6}^{\univ} \text{ is a {\it cluster monomial} if and only if } m \not\vert \ \bar {\bf x}^a \text{ for all } m\in M_{3,6},
\end{gather}
where $M_{3,6}=\{\text{exchange monomials}\} \cup \big\{\bar X\bar Y,\bar p_{135}\bar p_{246}\big\}$. We write $\mathbb M_{3,6}$ to denote the $\mathbb C[y_\alpha\colon \alpha\in \Phi^{\vee}_{\ge -1}]$-basis of cluster monomials for $A_{3,6}^{\rm univ}$.

\begin{proof}[{\bf Proof of Theorem~\ref{thm:Gr36}}]
We prove the statements of the theorem in order.

 $(i)$ In \texttt{Macaulay2}~\cite{M2} we compute the initial ideal of $I^{\rm ex}$ with respect to the cone $C$.
 For the computation, we fix the weight vector $w=r_1+\dots+r_{16}$ in the relative interior of $C$:
 \begin{displaymath}
 w= (16, 19, 16, 16, 16, 11, 10, 19, 16, 16, 16, 10, 7, 16, 11, 10, 16, 10, 7, 16, 27, 27).
 \end{displaymath}
 We compute a minimal generating set of $\init_C(I^{\rm ex})$: it has 54 elements, 52 of those are exchange monomials, and the other two are $p_{135}p_{246}$ and $XY$.
 This implies the first claim of the Theorem.

$(ii)$ We prove this part in three steps: first, we compute the generators of the ideal $\tilde I_{\bf r}^{\rm ex}$, then we define a surjective map $\tilde A_{\bf r}\to A_{3,6}^{\rm univ}$, and lastly, we show that the map is also injective.

 \medskip
 \noindent{\it Step 1:} By~\thref{prop:lifted generators} the lifted ideal $\tilde I^{\rm ex}_{\bf r}$ is generated by the lifts of elements of a Gr\"obner basis for $I^{\rm ex}$ and $C$.
 As a Gr\"obner basis we choose the exchange relations (whose initial forms are the exchange monomials), the four-term relation in~\eqref{eq:4terms} (whose initial form is $p_{135}p_{246}$) and the following element (whose initial form is $XY$):
 \begin{gather*}
 S(p_{134}p_{256} - \textcolor{blue}{p_{156}}\textcolor{blue}{p_{234}} - Y,\ p_{134}X - p_{136}p_{145}\textcolor{blue}{p_{234}} - \textcolor{blue}{p_{123}}p_{146}\textcolor{blue}{p_{345}}),
 \end{gather*}
 which is computed explicitly in~\thref{exp:lem facet} below.
 By the proof of $(i)$ above, the set of exchange relations together with the four-term relation~\eqref{eq:4terms}
 and the above $S$-pair form a (minimal) Gr\"obner basis for $I^{\rm ex}$ with respect to $C$.
 The reduced Gr\"obner basis $\mathcal G_C(I^{\rm ex})$ consists of~the 52 exchange relations and the additional two elements
 \begin{gather}
 f:=p_{135}p_{246} - \textcolor{blue}{p_{156}}\textcolor{blue}{p_{234}} - Y - \textcolor{blue}{p_{123}}\textcolor{blue}{p_{456}} - X - \textcolor{blue}{p_{126}}\textcolor{blue}{p_{345}}, \nonumber
 \\
 g:=XY - \textcolor{blue}{p_{123}}(\textcolor{blue}{p_{156}}p_{246}\textcolor{blue}{p_{345}} + \textcolor{blue}{p_{156}}\textcolor{blue}{p_{234}}\textcolor{blue}{p_{456}} + \textcolor{blue}{p_{126}}\textcolor{blue}{p_{345}}\textcolor{blue}{p_{456}})\nonumber
 \\ \hphantom{ g:=}
 - \textcolor{blue}{p_{126}}(p_{135}\textcolor{blue}{p_{234}}\textcolor{blue}{p_{456}} + \textcolor{blue}{p_{156}}\textcolor{blue}{p_{234}}\textcolor{blue}{p_{345}}).
 \label{eq:f and g}
 \end{gather}
 The first monomial in $f$ (and $g$) is its leading monomial.
 We now compute the lifts of the elements in $\mathcal G_C(I^{\rm ex})$ with respect to the matrix ${\bf r}$ in~\thref{def:lift}, which are given explicitly in~Appendix~\ref{sec:data Gr36}.

 \medskip
 \noindent{\it Step 2:} $A_{3,6}^{\rm univ}$ has 22 cluster variables (in one-to-one correspondence with those of $A_{3,6}$) and 16 coefficients labeled by almost positive roots of type ${\mathtt D}_4$: $y_\alpha$ with $\alpha\in \Phi_{\geq -1}^{\vee}$.
 We extend the morphism $\psi$
 in~\eqref{eq:psi36} to $\Psi\colon \mathbb C[t_1,\dots,t_{16}][p_{123}, \dots , p_{456}, x,y]\to A_{3,6}^{\univ}$ by sending $t_i$'s to $y_\alpha$'s according to the following identification:
\begin{gather}\label{eq:y-t}
y_{-\alpha_1} \leftrightarrow t_{16},\qquad
y_{\alpha_1} \leftrightarrow t_{14},\qquad
y_{\alpha_1 + \alpha_2} \leftrightarrow t_{6},\qquad\quad\
y_{\alpha_1 +\alpha_2+\alpha_4} \leftrightarrow t_{9}, \nonumber
\\
y_{-\alpha_2} \leftrightarrow t_{12},\qquad
y_{\alpha_2} \leftrightarrow t_{10},\qquad
y_{\alpha_1 + \alpha_3} \leftrightarrow t_{4},\qquad\quad\
y_{\alpha_1 +\alpha_2+\alpha_3} \leftrightarrow t_{2}, \nonumber
\\
y_{-\alpha_3} \leftrightarrow t_{15},\qquad
y_{\alpha_3} \leftrightarrow t_{3},\qquad\
y_{\alpha_1 + \alpha_4} \leftrightarrow t_{13},\qquad\quad
y_{2\alpha_1 + \alpha_2+\alpha_3+\alpha_4} \leftrightarrow t_{5}, \nonumber
\\
y_{-\alpha_4} \leftrightarrow t_{7},\qquad\
y_{\alpha_4} \leftrightarrow t_{8},\qquad\
y_{\alpha_1 +\alpha_3+\alpha_4} \leftrightarrow t_{1},\qquad
y_{\alpha_1 + \alpha_2+\alpha_3+\alpha_4} \leftrightarrow t_{11}.
 \end{gather}
 We now verify that $\tilde I^{\rm ex}_{\bf r}\subseteq \ker(\Psi)$: the lifts of those elements in $I^{\rm ex}$ that correspond to exchange relations in $A_{3,6}$ are sent to exchange relations in $A_{3,6}^{\rm univ}$ by $\Psi$, so they lie in $\ker(\Psi)$.
 For the elements $f,g\in\mathcal{G}_C(I^{\rm ex})$ in~\eqref{eq:f and g}
 note that, for example, $p_{245}f$ has an expression in terms of exchange relations with coefficients that are Pl\"ucker variables (see~\eqref{eq:f} for the precise expression).
 Hence,
 \begin{displaymath}
 p_{245}{\bf t}^{\mu(f)}\tilde f= p_{245}t_2^2t_4^2t_5t_6t_8t_9^3t_{10}^3t_{11}^3t_{12}^2t_{13}t_{14}t_{15}t_{16}\tilde f
 \end{displaymath}
 has an expression in terms of the lifts of those exchange relations with monomial coefficients in $t$'s and Pl\"ucker variables.
 So, $ 0=\Psi\big(\tilde fp_{245}{\bf t}^{\mu(f)}\big)=\Psi\big(\tilde f\big)\bar p_{245}{\bf y}^{\mu(f)}$,
 where ${\bf y}^{\mu(f)}=\Psi\big({\bf t}^{\mu(f)}\big)$.
 As $A_{3,6}^{\rm univ}$ is a domain, this implies $\tilde f\in \ker(\Psi)$.
 A similar argument implies that $\tilde g\in \ker(\Psi)$ (see~\eqref{eq:g}) and so we obtain the induced morphism $\bar\Psi\colon \tilde A_{\bf r}\to A_{3,6}^{\rm univ}$.
 Note that the image of $\bar \Psi$ contains all cluster variables and all coefficients of $A_{3,6}^{\rm univ}$, so $\bar \Psi$ is in fact surjective.

 \medskip

\noindent{\it Step 3:} Lastly, we show that $\bar \Psi$ is injective.
 By~\thref{thm:family}$(i)$ the standard monomial basis $\mathbb B_{<}$ (for $<$ a monomial term order compatible with $C$) is a $\mathbb C[t_1,\dots,t_{16}]$-basis for $\tilde A_{\bf r}$.
 Similarly, $A_{3,6}^{\rm univ}$ has the $\mathbb C[y_\alpha\colon \alpha\in \Phi^{\vee}_{\ge -1}]$-basis of cluster monomials $\mathbb M_{3,6}$.
 The test for membership in $\mathbb M_{3,6}$ is given by $M_{3,6}$ in~\eqref{eq:basis A36univ}, which is in one-to-one correspondence with the set $\{\init_<(g)\colon g\in \mathcal G_<(I^{\rm ex})\}$.
 Hence, there is a bijection between the standard monomial basis $\mathbb B_{<}$ for $\tilde A_{\bf r}$ (see~\thref{thm:family}$(i)$) and the cluster monomial basis $\mathbb M_{3,6}$ of $A_{3,6}^{\rm univ}$ induced by $\bar \Psi$.
 In particular, $\bar \Psi$ is injective and $\tilde A_{\bf r}\cong A_{3,6}^{\rm univ}$.

$(iii)$ To prove this part, we identify the rays $r_1,\dots,r_{16}$ with mutable cluster variables.
 As we have already identified $y_\alpha$'s with $t_i$'s in~\eqref{eq:y-t} (and by definition $t_i$'s correspond to $r_i$'s) it is enough to identify the $y_\alpha$'s with the mutable cluster variables of $A_{3,6}$.
 This can be done by considering the exchange relations obtained by repeatedly mutating our bipartite initial seed at a sink.
 More precisely, we only consider the mutable part of the initial quiver and mutate at all the vertices with only incoming arrows from mutable vertices, which (by slight abuse of language) we refer to as sinks.
 The order of the individual mutations in this mutation sequence is irrelevant as they pairwise commute.
 Every exchange relation produced by mutation at a sink corresponding to a cluster variable $x$ has the property that one of the cluster monomials involves exactly one coefficient $y_{\alpha}$ (see~\cite[Lemma~12.7]{FZ_clustersIV}).
 When iterating the process of mutating at sinks, every mutable cluster variable appears as a sink at some point.
 Moreover,~\cite[Lemma~12.8]{FZ_clustersIV} implies that the assignment $x \mapsto y_\alpha$ defines a bijection between mutable cluster variables and coefficients.
 Combining with the identification of $y_\alpha$'s with $t_i$'s in~\eqref{eq:y-t} we obtain:
\begin{gather*}
r_{1} \leftrightarrow \bar p_{125}, \qquad\ r_{2} \leftrightarrow \bar p_{134},\ \qquad r_{3} \leftrightarrow \bar p_{124},\ \qquad r_{4} \leftrightarrow \bar p_{145}
\\
r_{5} \leftrightarrow \bar p_{135}, \qquad\ r_{6} \leftrightarrow \bar p_{136},\ \qquad r_{7} \leftrightarrow \bar p_{146},\ \qquad r_{8} \leftrightarrow \bar p_{256},
\\
r_{9} \leftrightarrow \bar p_{356}, \qquad\ r_{10} \leftrightarrow \bar p_{346}, \qquad r_{11} \leftrightarrow Y, \ \ \ \qquad r_{12} \leftrightarrow \bar p_{245},
\\
r_{13} \leftrightarrow \bar p_{235},\qquad r_{14} \leftrightarrow X, \ \ \qquad r_{15} \leftrightarrow \bar p_{236}, \qquad r_{16} \leftrightarrow \bar p_{246}.
\end{gather*}
Next, to every seed we associate a weight vector that is the sum of the rays corresponding to its mutable cluster variables.
For example, the weight vector associated to $s=\{\bar p_{124},\bar p_{125},\bar p_{245},\bar p_{256}\}$ is $w_s=r_1+r_3+r_8+r_{12}$.
Using \texttt{Macaulay2}~\cite{M2} we verify that $\init_{w_s}(I^{\rm ex})$ is a totally positive binomial prime ideal for every seed listed above.
The initial ideals can be found on~\cite{homepage}.

To see that $C\cap \trop(I^{\rm ex})=\trop^+(I^{\rm ex})$ observe that for every element $h\in \mathcal G_C(I^{\rm ex})$ its initial monomial $\init_C(h)$ is the unique monomial with positive coefficient (the complete list of $\mathcal G_C(I^{\rm ex})$ can be found in Section~\ref{sec:data Gr36}).
Hence, a weight vector $w$ lies in $C\cap \trop(I^{\rm ex})$ if and only if it lies in $\trop^+(I^{\rm ex})$.
\end{proof}

\begin{Remark}\thlabel{rmk:cluster techniques}
There are various methods in cluster theory to compute the exchange relations for $A_{3,6}^{\rm univ}$ and $M_{3,6} $, e.g.,~one can use the categorification of finite type cluster algebras with universal coefficients introduced in~\cite{NC_univ}.
To compute $M_{3,6} $ one can use the compatibility degree of cluster variables from~\cite{FZ_Y-systems}.
In fact, the elements of $M_{3,6}$ are exactly those pairs of cluster variables whose compatibility degrees are positive.
However, we have presented the case of~$\Gr\big(3,\mathbb C^6\big)$ with as few machinery from the cluster theory as possible to make it more digestible to a broader audience.
\end{Remark}
\begin{Example}\thlabel{exp:lem facet}
Here we demonstrate the need of~\thref{lem:facet} and that in fact the Lemma~is not true for arbitrary elements of the ideal $J$. To see this, let $J=I^{\rm ex}$ and take
\begin{displaymath}
v:= \frac{1}{4}\sum_{\begin{smallmatrix}
i=1\\
i\not =2,8
\end{smallmatrix}}^{16} r_i + \frac{5}{4}r_8\qquad \text{and}\qquad w := r_2 + v.
\end{displaymath}
Note that $w\in C^\circ$, hence the assumptions of~\thref{lem:facet} hold.
One can explicitly compute
\begin{gather*}
v = \frac{1}{4}(19, 26, 18, 19, 19, 9, 9, 21, 19, 14, 19, 9, 7, 19, 15, 9, 19, 14, 6, 19, 29, 33),
\\
w = \frac{1}{4}(23, 30, 26, 23, 23, 17, 13, 29, 23, 22, 23, 13, 7, 23, 15, 13, 23, 14, 10, 23, 37, 41).
\end{gather*}
Now take the $S$-pair of two exchange relations
\begin{displaymath}\begin{split}
h&:=S(p_{134}p_{256} - \textcolor{blue}{p_{156}}\textcolor{blue}{p_{234}} - {Y},\ p_{134}x - p_{136}p_{145}\textcolor{blue}{p_{234}} - \textcolor{blue}{p_{123}}p_{146}\textcolor{blue}{p_{345}}) \\
&= -{XY} - \textcolor{blue}{p_{156}}\textcolor{blue}{p_{234}}x + p_{136}p_{145}\textcolor{blue}{p_{234}}p_{256} - \textcolor{blue}{p_{123}}p_{146}p_{256}\textcolor{blue}{p_{345}} \in I^{\rm ex}.\end{split}
\end{displaymath}
The weights of the non-zero monomials in $h$ with respect to $v$ are (in order) $\frac{31}{2} , \frac{31}{2}, \frac{29}{2}, \frac{33}{2}$ and with respect to $w$ are $\frac{39}{2} , \frac{41}{2}, \frac{39}{2}, \frac{41}{2}$.
We compute
\begin{displaymath}
\init_{v}(h) = \init_{r_2}(\init_{v}(h)) = p_{136}p_{145}\textcolor{blue}{p_{234}}p_{256} \not = -{XY} + p_{136}p_{145}\textcolor{blue}{p_{234}}p_{256} = \init_w(h).
\end{displaymath}
In particular, the statement of~\thref{lem:facet} does not hold for $h$, hence it is false in general for arbitrary elements of $J$.
More importantly, the initial form of an arbitrary element $h\in J$ need not be the same with respect to different weight vectors in the relative interior of a maximal Gr\"obner cone of $J$.
This may occur when $h$ contains more than one monomial in the monomial initial ideal.
Here, the monomials $XY$, $p_{136}p_{145}\textcolor{blue}{p_{234}}p_{256}$ and $\textcolor{blue}{p_{123}}p_{146}p_{256}\textcolor{blue}{p_{345}}$ all lie in $\init_C(I^{\rm ex})$.
\end{Example}

\subsection{Stanley--Reisner ideals and the cluster complex}

\begin{Definition}
Let $(Q,F)$ be an ice-quiver and $\mathcal{V}$ the set of mutable cluster variables of~$A_{(Q,F)}$.
We call $x$ and $x'$ in $\mathcal{V}$ {\it compatible} if there exists a cluster containing both of them. Similarly, a subset of $\mathcal{V}$ is compatible if it consists of pairwise compatible cluster variables. The {\it cluster complex} $\Delta(A_{(Q,F)})$ is the simplicial complex on $\mathcal{V}$ whose simplices are the compatible subsets.
\end{Definition}

\begin{Remark}
Note that by definition the cluster complex can be realized by the ${\bf g}$-fan with simplices corresponding to simplicial cones.
It was shown in~\cite{GHKK} that the ${\bf g}$-fan is a simplicial fan in complete generality.
This occurred several years after the cluster complex had been defined.
\end{Remark}

\begin{Definition}\thlabel{def:SR-ideal}
Let $\Delta$ be a simplicial complex with vertex set $V=\{x_1,\dots,x_n\}$.
The {\it Stanley--Reisner ideal} of $\Delta$ is generated by monomials associated to the minimal non-faces of $\Delta$ as:
\begin{displaymath}
I_{\Delta}:=\langle x_{i_1} \cdots x_{i_s}\colon \{x_{i_1},\dots,x_{i_s}\}\not\in \Delta\rangle \subseteq \mathbb K[x_1,\dots,x_n].
\end{displaymath}
Reversely, to every squarefree monomial ideal one can associate its {\it Stanley--Reisner complex}, whose non-faces are defined by the monomials in the ideal.
\end{Definition}

\begin{Corollary}\thlabel{cor:Stanley-Reisner}
Let $A$ be $A_{2,n}$ or $A_{3,6}$ and $\Delta(A)$ the associated cluster complex.
Similarly, let $I$ be the ideal $I_{2,n}$ or $I^{\rm ex}$ and $C$ the maximal cone in $\GF(I)$ whose initial ideal contains all the exchange monomials.
Then the Stanley--Reisner ideal $I_{\Delta(A)}$ coincides with the initial ideal~$\init_C(I)$.
\end{Corollary}
\begin{proof}
The initial ideal $\init_C(I)$ is squarefree and generated by the monomials in the set $M$, which is respectively, $M_{2,n}$ or $M_{3,6}$.
Observe that $M$ defines the minimal non-faces of $\Delta(A)$.
Hence, $\Delta(A)$ is the Stanley--Reisner complex of $\init_C(I)$.
\end{proof}

\subsection{Newton--Okounkov bodies and mutations}\label{sec:NO}

In this section, we explain how our results relate to Escobar--Harada's work on wall-crossing phenomena for Newton--Okounkov bodies in~\cite{EH-NObodies}.
Given a homogeneous ideal $J\subseteq \mathbb K[x_1,\dots,x_n]$ assume there exists a maximal cone $\sigma$ in $\trop(J)$ whose associated ideal is toric.
Then~\cite[Section~4]{KM16} provides a recipe to construct a full rank valuation $\val_\sigma\colon A\setminus \{0\}\to \mathbb Z^{d}$, where $A=\mathbb K[x_1,\dots,x_n]/J$ and $d$ is the dimension of the affine variety $V(J)$.
Furthermore,~\cite[Corollary~4.7]{KM16} shows how to compute the corresponding Newton--Okounkov body $\Delta(\sigma)$. Without loss of generality we may assume that the first entry of $\val_\sigma(f)$ equals the degree of $f$ (with respect to the standard grading). Then the Newton--Okounkov body of $\val_\sigma$ and $\mathbb K[x_1,\dots,x_n]/J$~is
\begin{displaymath}
\Delta(\sigma):={\rm Cone}(\val_\sigma(f)\colon \ f\in A\setminus \{0\}) \cap \{1\}\times \mathbb R^{d-1}.
\end{displaymath}
Escobar and Harada study the relation between the Newton--Okounkov bodies of two maximal prime cones intersecting in a facet.
They give two piecewise linear maps, called {\it flip} and {\it shift} which send one Newton--Okounkov body to another.

Here, we focus on the case of $\Gr(2,\mathbb C^n)$.
We fix a triangulation $T$ of the $n$-gon.
The output of Algorithm~\ref{alg:weight} can be used to define a total order on $\mathbb Z^{2n-3}$ as follows.

\begin{Definition}\thlabel{rmk:compatible order}
Let $V_0,\dots,V_i$ be the output of Algorithm~\ref{alg:weight} for some triangulation $T$.
To each~$V_j$ we associate a sequence $\overline{V}_j$ with the same elements as $V_j$.
Let $\overline{V}$ be the sequence $\big(\overline{V}_0,\dots,\overline{V}_i\big)$ and $\prec$ the associated lexicographic order on $\mathbb Z^{2n-3}$.
Call $\prec$ a {\it total order compatible with~$T$}.
\end{Definition}

Recall from Section~\ref{sec:monomial_degeneration} the standard monomial basis $\mathbb B_{<}$ consisting of non-crossing monomials
(i.e.,~every monomial in $\mathbb B_{<}$ corresponds to a collection of non-crossing (boundary and internal) arcs, where arcs may appear with multiplicity greater than 1).
Then the map
\begin{equation}\label{eq:g-map}
{\bf g}^{\widehat{T}}\colon\quad \mathbb B_{<} \to \mathbb Z^{2n-3} \qquad \text{given by} \quad
\bar{\bf p}^a \mapsto \sum_{\overline{ij}\in T} a_{ij}{\bf g}_{ij}^{\widehat{T}}.
\end{equation}
extends to a full rank valuation on $A_{2,n}\setminus \{0\}$ as follows.
Fix a total order $\prec$ compatible with $T$.
Every $0\neq f\in A_{2,n}$ is a unique linear combination of elements in $\mathbb B_{<}$, that is $f=\sum_{\bar{\bf p}^a\in \mathbb B_{<}}c_a\bar {\bf p}^a$. We define the valuation of $f$ as
${\bf g}^{\widehat{T}}(f):= {\rm min}_{\prec} \big\{{\bf g}^{\widehat{T}}(\bar{\bf p}^a)\colon c_a\not =0\big\}$.
Denote the associated graded algebra by ${\rm gr}_{T}(A_{2,n})$ and the corresponding Newton--Okounkov body by $\Delta(T)$.
By~\cite[Corollary~2]{B-quasival} and~\thref{prop:wt map} we have ${\rm gr}_T(A_{2,n})\cong S/\init_T(I_{2,n})$ and $\Delta(T)={\rm conv}\big({\bf g}^{\widehat{T}}_{ij}\big)_{ij}$.

\begin{Remark}\thlabel{rem:gv val}
The ${\bf g}$-vectors and the corresponding valuation make sense in greater generality. For example, the case of arbitrary Grassmannians is treated in~\cite{BCMN}, where the authors (among other things) explain how the associated Newton--Okounkov bodies arise in the context of~\cite{GHKK}.
Even more generally, the algebra $A_{2,n}$ can be replaced by the {\it middle cluster algebra} in the sense of~\cite{GHKK}, the standard monomial basis by the corresponding {\it theta basis}, and the total order $\prec$ by a lexicographic refinement of the {\it dominance order}.
Similarly, Fujita and Oya in~\cite[Section~3]{FO20} define a valuation on the ambient field of a cluster algebra (whose exchange matrix has full rank). Their valuation recovers Fomin--Zelevinsky's ${\bf g}$-vectors for cluster monomials.
\end{Remark}

Interpreting ${\bf g}$-vectors as a valuation reveals the necessity for working with \emph{extended} ${\bf g}$-vectors as opposed to their \emph{truncated} counterparts that are popular in algebraic or representation theoretic applications of cluster algebras.
The following example shows that truncated ${\bf g}$-vectors do not have the desired properties for applications.

\begin{Example}\thlabel{exp:need ext gv}
Consider for $\Gr(2,6)$ the triangulation $T$ consisting of arcs $\overline{24}$, $\overline{25}$, $\overline{26}$.
Order the cluster variables corresponding to $T$ in a compatible way using the output of Algorithm~\ref{alg:weight}, e.g.,~$p_{16}$, $p_{23}$, $\underline{p_{24}}$, $p_{56}$, $p_{12}$, $\underline{p_{25}}$, $\underline{p_{26}}$, $p_{34}$, $p_{45}$.
Now compute the ${\bf g}$-vectors with respect to $T$ and consider the following elements in $A_{2,n}$ endowed with their ${\bf g}$-vectors:
\begin{displaymath}
\begin{matrix}
(1,0,\underline{0},0,0,\underline{1},\underline{0},0,0) &=& (1,0,\underline{0},0,0,\underline{1},\underline{0},0,0) &\prec&
(0,0,\underline{0},1,1,\underline{0},\underline{0},0,0),\\
 \bar p_{15}\bar p_{26} & = & \bar p_{16}\bar p_{25} & + & \bar p_{12}\bar p_{56},
&&&& \\[1ex]
(0,1,\underline{0},0,0,\underline{0},\underline{0},0,1) & = &
(0,1,\underline{0},0,0,\underline{0},\underline{0},0,1) &\prec& (0,0,\underline{0},0,0,\underline{1},\underline{0},1,0),\\
\bar p_{24}\bar p_{35} & = & \bar p_{23}\bar p_{45} & + & \bar p_{25}\bar p_{34}. 
\end{matrix}
\end{displaymath}
The total order $\prec$ is compatible with $T$.
The truncated ${\bf g}$-vectors are the underlined parts of the ${\bf g}$-vectors above. So if we decided to only consider those we would need to find an order that satisfies $(0,0,0)\succ (0,1,0)$ and $(0,1,0)\succ (0,0,0)$, a contradiction.
\end{Example}

Next we describe how the Newton--Okounkov body $\Delta(T)$ behaves under changes of the triangulation.
As all triangulations can be obtained from one another by a sequence of flips of arcs, we focus on performing a single such flip.
Let $T_1$ and $T_2$ be two triangulations such that there exist $a<b<c<d$ in $[n]$ with $\overline{ac}\in T_1, \overline{bd}\in T_2$ and $T_1\setminus \{\overline{ac}\}=T_2\setminus\{\overline{bd}\}$.
In other words, $T_2$ is obtained from $T_1$ by flipping the arc $\overline{ac}$.
We denote by $\mathbb R^{T_1}$ the vector space $\mathbb R^{2n-3}$ with standard basis $\{f_{ij}\colon \overline{ij}\in T_1\}$ and similarly, $\mathbb R^{T_2}$ for $\mathbb R^{2n-3}$ with basis $\{f'_{ij}\colon \overline{ij}\in T_2\}$.

The theory of cluster varieties gives us two distinguished maps from $\mathbb R^{2n-3} $ to itself. The first map is a piecewise linear shear that can be obtained by the {\it Fock--Goncharov tropicalization of a cluster mutation} (see~\cite[Definition~1.22]{GHKK}):\vspace{-.5ex}
\begin{align*}
 \zeta_{ac}\colon\ \mathbb R^{T_1} &\to \mathbb R^{T_1},
 \\
 m=\sum_{\overline{ij}\in T_1} m_{ij}f_{ij} &\mapsto
 \begin{cases}
 m & \text{if}\quad m_{ac}\le 0,\\
 m+m_{ac}v_{ac} & \text{if}\quad m_{ac}>0,
 \end{cases}
\end{align*}
where $v_{ac}=-f_{ab}-f_{cd}+f_{ad}+f_{bc}$.
The second map is a ${\rm GL}_{2n-3}(\mathbb Z)$-base change from $\{f_{ij}\colon \overline{ij}$ $\in T_1\}$ to $\{f'_{ij}\colon \overline{ij}\in T_2\}$ corresponding to a seed mutation (see~\cite[equation~(8)]{FG_ensem}) denoted by $\mu^{T_1}_{T_2}\colon \mathbb R^{T_1} \to \mathbb R^{T_2}$.
It is given by $f'_{ij}=f_{ij}$ for $\overline{ij}\in T_1\cap T_2$ and $f'_{bd}=-f_{ac}+f_{ab}+f_{cd}$.

The following Lemma~follows by combining the results in~\cite[Sections~1.3 and~5]{GHKK}.
For the convenience of the reader, we give a self-contained elementary proof below.
\begin{Lemma}\thlabel{lem:mutate NObodies}
Let $T_1$ and $T_2$ be two triangulations related by a flip as above.
Then $\Delta(T_2) = \mu^{T_1}_{T_2} \circ \zeta_{ac} (\Delta(T_1))$. More precisely, for every $\bar{\bf p}^a\in \mathbb B_{<}$ we have
${\bf g}^{\widehat{T_2}}(\bar {\bf p}^a) = \mu^{T_1}_{T_2} \circ \zeta_{ac} \circ {\bf g}^{\widehat{T_1}}(\bar {\bf p}^a)$.
\end{Lemma}

\begin{proof}
By definition of ${\bf g}^{\widehat{T_p}}$ for $p=1,2$ from~\eqref{eq:g-map} it is enough to show the second claim for the Pl\"ucker coordinates.
As $\Delta(T_p)$ is the convex hull of ${\bf g}^{\widehat{T_p}}_{ij}$ for $1\le i<j\le n$, this implies the first claim.

Given $1\le i<j\le n$, by~\thref{def:g}
the coefficient of $f_{ac}$ in ${\bf g}^{\widehat{T_1}}(\bar p_{ij})={\bf g}^{\widehat{T_1}}_{ij}$ can be $1$, $-1$ or~$0$.
We go through the details of the case when it is $1$, the other two cases are similar.
Recall the sign conventions for the ${\bf g}$-vectors from Figure~\ref{fig:sign g-v}.
Take
${\bf g}^{\widehat{T}_1}_{ij}= f_{ac}+\sum_{\overline{kl}\in T_1\setminus \{\overline{ac}\} } \sigma_{i \rightsquigarrow j}^{\overline{kl}} f_{kl}$.
Recall that $T_1\setminus \{\overline{ac}\}= T_2\setminus \big\{\overline{bd}\big\}$ and $f_{kl}=f'_{kl}$ for $\overline{kl}\in T_1 \setminus \{ \overline{ac} \}$. Now we compute $\mu^{T_1}_{T_2} \circ \zeta_{ac} \circ {\bf g}^{\widehat{T_1}}(p_{ij})$ step-by-step as follows:
\begin{align*}
{\bf g}^{\widehat{T}_1}(p_{ij}) = f_{ac} + \sum_{\overline{kl}\in T_1 \setminus \{ \overline{ac} \} } \sigma_{i \rightsquigarrow j}^{\overline{kl}} f_{kl} &\xmapsto{\zeta_{ac}} f_{ac} + v_{ac} + \sum_{\overline{kl}\in T_1 \setminus \{ \overline{ac} \} } \sigma_{i \rightsquigarrow j}^{\overline{kl}} f_{kl} \\
&\xmapsto{\mu^{T_1}_{T_2}} -f'_{bd}+f'_{ad}+f'_{bc}+\displaystyle{\sum_{\overline{kl}\in T_2\setminus \{ \overline{bd} \} }} \sigma_{i \rightsquigarrow j}^{\overline{kl}} f'_{kl} = {\bf g}^{\widehat{T}_2}(p_{ij}).
\end{align*}
Observe that the sign $\sigma_{i \rightsquigarrow j}^{\overline{kl}}$ depends on the triangulation, but the above equation takes into account the change of signs that might happen when we pass from $T_1$ to $T_2$.
\end{proof}

\begin{figure}[ht]
 \centering
\begin{center}
 \begin{tikzpicture}[scale=.35]
 \draw (-2,1) -- (-1,0) -- (-2,-1);
 \draw (-1,0) -- (0,0) -- (0,1);
 \draw (0,0) -- (1,0);
 \draw (2,1) -- (1,0) -- (2,-1);

 \node[left] at (-2,1) {1};
 \node[left] at (-2,-1) {3};
 \node[above] at (0,1) {5};
 \node[right] at (2,1) {2};
 \node[right] at (2,-1) {4};
 \node at (0,-1) {$\Upsilon_1$};

 \begin{scope}[xshift=10cm]
 \draw (-2,1) -- (-1,0) -- (-2,-1);
 \draw (-1,0) -- (0,0) -- (1.5,0);
 \draw (1,1) -- (0,0) -- (1,-1);

 \node[left] at (-2,1) {1};
 \node[left] at (-2,-1) {3};
 \node[right] at (1.5,0) {2};
 \node[right] at (1,1) {5};
 \node[right] at (1,-1) {4};
 \node at (-.25,-1) {$\Upsilon$};
 \end{scope}

 \begin{scope}[xshift=20cm]
 \draw (-2,1) -- (-1,0) -- (-2,-1);
 \draw (-1,0) -- (0,0) -- (0,1);
 \draw (0,0) -- (1,0);
 \draw (2,1) -- (1,0) -- (2,-1);

 \node[left] at (-2,1) {1};
 \node[left] at (-2,-1) {3};
 \node[above] at (0,1) {4};
 \node[right] at (2,1) {2};
 \node[right] at (2,-1) {5};
 \node at (0,-1) {$\Upsilon_2$};
 \end{scope}

 \begin{scope}[xshift=30cm]
 \node at (-5,0) {$\longleftrightarrow$};
 \draw (-2,1) -- (-1,0) -- (-2,-1);
 \draw (-1,0) -- (0,0) -- (0,-1);
 \draw (0,0) -- (1,0);
 \draw (2,1) -- (1,0) -- (2,-1);

 \node[left] at (-2,1) {1};
 \node[left] at (-2,-1) {3};
 \node[below] at (0,-1) {4};
 \node[right] at (2,1) {5};
 \node[right] at (2,-1) {2};
 \node at (0,1) {$\Upsilon_2'$};
 \end{scope}
 \end{tikzpicture}
\end{center}\vspace{-2ex}
 \caption{Trivalent trees for $\Gr(2,\mathbb C^5)$ corresponding to maximal cones in $\trop(I_{2,n})$ that intersect in a facet ($\Upsilon_1$ and $\Upsilon_2$) and the tree corresponding to the intersection is $\Upsilon$, see~\thref{exp:NOGr25}.}
 \label{fig:trees25}
\end{figure}
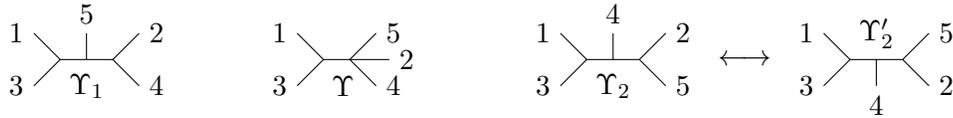

Note that although in Section~\ref{sec:monomial_degeneration} we are only considering a fixed maximal Gr\"obner cone $C$, all other maximal cones can be obtained from $C$ through a symmetric group action (see~\cite[Second Proof of $\supseteq$ Theorem~4.3.5]{M-S}).
Moreover,~\cite[Lemma~5.13]{EH-NObodies} shows that any two maximal cones in~$\trop(I_{2,n})$ that intersect in a facet are faces of one maximal Gr\"obner cone (up to symmetry).

\begin{Corollary}\thlabel{cor:mutate NObodies}
Let $\sigma_1,\sigma_2\in\trop(I_{2,n})$ be two maximal prime cones that intersect in a facet. Then $\Delta(\sigma_1)$ and $\Delta(\sigma_2)$ are 
related by the shear map $\zeta_{ac}$ for appropriate $a,c\in [n]$.
\end{Corollary}
\begin{proof}
The cones $\sigma_1$ and $\sigma_2$ are in correspondence with trivalent trees $\Upsilon_{1}$, $\Upsilon_{2}$ by~\thref{thm:SS}.
By~\cite[Lemma~5.13]{EH-NObodies} we can assume that $\sigma_1$ and $\sigma_2$ are faces of the same maximal Gr\"obner cone~$C$.
Hence, we may assume that the leaves of $\Upsilon_{1}$ and $\Upsilon_{2}$ are labeled in the same order.
The symmetric group action on $\trop(I_{2,n})$
translates to permuting the leaf labeling of trivalent trees.
So we can simultaneously move the cones $\sigma_1$ and $\sigma_2$ to faces of $C$.
Say these faces correspond to triangulation $T_1$ and $T_2$ of the $n$-gon, respectively.
One consequence of the symmetric group action is that $V(\init_{\sigma_i}(I))\cong V(\init_{T_i}(I))$ for $i=1,2$.
In particular, these isomorphisms induce unimodular equivalences
$\Delta(\sigma_i)\cong \Delta(T_i)$ for $i=1,2$.
The assumption that $\sigma_1$ and $\sigma_2$ intersect in a facet, directly implies that $T_1$ and $T_2$ differ by flipping one arc.
So there exist $a<b<c<d$ in $[n]$ with ${\overline{ac}}\in T_1$, $\overline{bd}\in T_2$ and $T_1\setminus \{\overline{ac}\}=T_2\setminus\{\overline{bd}\}$.
Now applying~\thref{lem:mutate NObodies} completes the proof.
\end{proof}

\begin{Example}\thlabel{exp:NOGr25}
We illustrate the proof of~\thref{cor:mutate NObodies} for $\Gr(2,\mathbb C^5)$.
Let $\Upsilon_1$ and $\Upsilon_2$ be the trees in Figure~\ref{fig:trees25}.
They correspond to cones $\sigma_1$ and $\sigma_2$ sharing a facet corresponding to the tree~$\Upsilon$ from Figure~\ref{fig:trees25}.
First, we apply~\cite[Lemma~5.13]{EH-NObodies} and replace $\Upsilon_2$ by a tree $\Upsilon_2'$ that gives the same initial ideal as $\Upsilon_2$, and its leaves are labeled in the same order as $\Upsilon_1$.
This implies that~$\sigma_1$ and $\sigma_2$ are indeed faces of the same maximal Gr\"obner cone (corresponding to the clockwise labeling of the leaves: $15243$).
Then we apply the symmetric group element $s=(235)$ to the leaves of $\Upsilon_1$ and $\Upsilon_2'$ to obtain the following trees which are both dual to triangulations of the $5$-gon.
Moreover, their triangulations are related by exchanging the diagonal $\overline{24}$ for $\overline{13}$ (this is also called a {\it flip}):
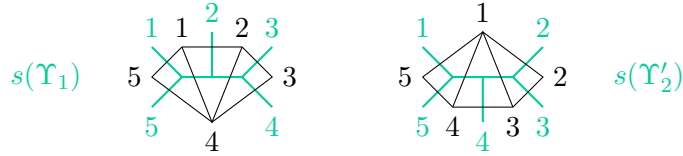
\begin{center}
 \begin{tikzpicture}[scale=.4]
 \draw[thick, caribbeangreen] (-2,1) -- (-1,0) -- (-2,-1);
 \draw[thick, caribbeangreen] (-1,0) -- (0,0) -- (0,1.5);
 \draw[thick, caribbeangreen] (0,0) -- (1,0);
 \draw[thick, caribbeangreen] (2,1) -- (1,0) -- (2,-1);

 \draw (-1,1) -- (1,1) -- (2,0) -- (0,-1.5) -- (-2,0) -- (-1,1);
 \draw (-1,1) -- (0,-1.5) -- (1,1);
 \node[above] at (-1,1) {1};
 \node[above] at (1,1) {2};
 \node[right] at (2,0) {3};
 \node[below] at (0,-1.5) {4};
 \node[left] at (-2,0) {5};

 \node[above, caribbeangreen] at (-2,1) {1};
 \node[below, caribbeangreen] at (-2,-1) {5};
 \node[above, caribbeangreen] at (0,1.5) {2};
 \node[above, caribbeangreen] at (2,1) {3};
 \node[below, caribbeangreen] at (2,-1) {4};
 \node[left, caribbeangreen] at (-4,0) {$s(\Upsilon_1)$};

 \begin{scope}[xshift=9cm]
 \draw[thick, caribbeangreen] (-2,1) -- (-1,0) -- (-2,-1);
 \draw[thick, caribbeangreen] (-1,0) -- (0,0) -- (0,-1.5);
 \draw[thick, caribbeangreen] (0,0) -- (1,0);
 \draw[thick, caribbeangreen] (2,1) -- (1,0) -- (2,-1);

 \draw (-2,0) -- (-1,-1) -- (1,-1) -- (2,0) -- (0,1.5) -- (-2,0);
 \draw (-1,-1) -- (0,1.5) -- (1,-1);

 \node[above, caribbeangreen] at (-2,1) {1};
 \node[below, caribbeangreen] at (-2,-1) {5};
 \node[below, caribbeangreen] at (0,-1.5) {4};
 \node[above, caribbeangreen] at (2,1) {2};
 \node[below, caribbeangreen] at (2,-1) {3};
 \node[right, caribbeangreen] at (4,0) {$s(\Upsilon_2')$};

 \node[below] at (-1,-1) {4};
 \node[below] at (1,-1) {3};
 \node[right] at (2,0) {2};
 \node[above] at (0,1.5) {1};
 \node[left] at (-2,0) {5};
 \end{scope}
 \end{tikzpicture}
 \end{center}
The symmetric group element $s$ induces an automorphism $s\colon S\!\to \!S$ given by $s(p_{ij})\!=p_{s^{-1}(i)s^{-1}(j)}$.
It is straight-forward to verify that $s(\init_{s(\Upsilon)}(I))=\init_{\Upsilon}(I)$ for $\Upsilon\in \{\Upsilon_1,\Upsilon_2'\}$.
\end{Example}

Consider $\sigma_1$ and $\sigma_2$ as in~\thref{cor:mutate NObodies} and assume they lie in the same maximal Gr\"obner cone~$C$.
Then the standard monomial basis of $C$ induces a bijective map between the value semigroups $\operatorname{im}(\val_{\sigma_1})$ and $\operatorname{im}(\val_{\sigma_2})$ (see~\cite[Section~4.2]{EH-NObodies}, where this is called an {\it algebraic wall-crossing}).
As seen in~\thref{lem:mutate NObodies} the map $\mu^{T_1}_{T_2}\circ \zeta_{ac}$ extends to a map between the value semigroups $\operatorname{im}\big({\bf g}^{\widehat{T_1}}\big)$ and $\operatorname{im}\big({\bf g}^{\widehat{T_2}}\big)$.
In~\cite[Theorem~5.15]{EH-NObodies} the authors show that their piecewise linear flip map induces the algebraic wall-crossing for $\Gr(2,\mathbb C^n)$.
Hence,~\thref{cor:mutate NObodies} implies that the flip map for $\Gr(2,\mathbb C^n)$ is of cluster nature in the sense that (up to unimodular equivalence) it is given by the Fock--Goncharov tropicalization of a cluster mutation.

\begin{Remark}
Cluster mutations are a very special class of automorphisms of a complex alge\-braic torus preserving its canonical volume form.
Automorphisms preserving this form have various names in the literature such as {\it Laurent polynomial mutations}~\cite{Ilten_mutation}, elements of the {\it special Cremona group}~\cite{Usn13} or {\it combinatorial mutations}~\cite{ACGK12}.
Ilten in~\cite[Appendix]{EH-NObodies} related the wall-crossing formulas to the theory of polyhedral divisors for complexity-one $T$-varieties~\cite{AH06} and outlined how this relates to combinatorial mutations in the sense of~\cite{ACGK12}.
\end{Remark}

\appendix
\section{Computational data}\label{sec:data Gr36}
Here, we present data on the ideal $I^{\rm ex}$.
Recall that the weighted homogeneous coordinate ring of $\Gr\big(3,\mathbb C^6\big)$ with respect to this embedding is $A_{3,6}\cong \mathbb C[p_{123},\dots,p_{456},X,Y]/I^{\rm ex}$.
We frequently identify the variables $p_{123},\dots,p_{456},X,Y$ with their cosets $\bar p_{123},\dots,\bar p_{456},\bar X,\bar Y$ in $A_{3,6}$, as well as elements of the ideal $I^{\rm ex}$ with the corresponding relations in $A_{3,6}$; that is, we identify, for example
\begin{displaymath}
p_{145}p_{236} - \textcolor{blue}{p_{123}}\textcolor{blue}{p_{456}} - X \in I^{\rm ex} \qquad \text{and} \qquad \bar p_{145}\bar p_{236} = \textcolor{blue}{\bar p_{123}}\textcolor{blue}{\bar p_{456}} + \bar X \ \text{ in } A_{3,6}.
\end{displaymath}

\medskip
\noindent{\bf A minimal generating set for $\boldsymbol{I^{\rm ex}}$.}
We now list a minimal generating set for $I^{\rm ex}$
consisting of elements of the reduced Gr\"obner basis $\mathcal G_C(I^{\rm ex})$. Note that except the last polynomial $f$, these are the exchange relations of $A_{3,6}$ with the first monomial being the exchange monomial.
\begin{gather*}
 p_{145}p_{236} - \textcolor{blue}{p_{123}}\textcolor{blue}{p_{456}} - X,\qquad\qquad\ \
 p_{124}p_{356} - \textcolor{blue}{p_{123}}\textcolor{blue}{p_{456}} - Y,
\\
 p_{136}p_{245} - \textcolor{blue}{p_{126}}\textcolor{blue}{p_{345}} - X,\qquad\qquad\ \
 p_{125}p_{346} - \textcolor{blue}{p_{126}}\textcolor{blue}{p_{345}} - Y,
\\
 p_{146}p_{235} - \textcolor{blue}{p_{156}}\textcolor{blue}{p_{234}} - X,\qquad\qquad\ \
 p_{134}p_{256} - \textcolor{blue}{p_{156}}\textcolor{blue}{p_{234}} - Y,
\\
 p_{246}p_{356} - p_{346}p_{256} - p_{236}\textcolor{blue}{p_{456}},\qquad
 p_{245}p_{356} - \textcolor{blue}{p_{345}}p_{256} - p_{235}\textcolor{blue}{p_{456}},
\\
 p_{146}p_{356} - p_{346}\textcolor{blue}{p_{156}} - p_{136}\textcolor{blue}{p_{456}},\qquad
 p_{145}p_{356} - \textcolor{blue}{p_{345}}\textcolor{blue}{p_{156}} - p_{135}\textcolor{blue}{p_{456}},
\\
 p_{245}p_{346} - \textcolor{blue}{p_{345}}p_{246} - \textcolor{blue}{p_{234}}\textcolor{blue}{p_{456}},\qquad
 p_{235}p_{346} - \textcolor{blue}{p_{345}}p_{236} - \textcolor{blue}{p_{234}}p_{356},
\\
 p_{145}p_{346} - \textcolor{blue}{p_{345}}p_{146} - p_{134}\textcolor{blue}{p_{456}},\qquad
 p_{135}p_{346} - \textcolor{blue}{p_{345}}p_{136} - p_{134}p_{356},
\\
 p_{146}p_{256} - p_{246}\textcolor{blue}{p_{156}} - \textcolor{blue}{p_{126}}\textcolor{blue}{p_{456}},\qquad
 p_{145}p_{256} - p_{245}\textcolor{blue}{p_{156}} - p_{125}\textcolor{blue}{p_{456}},
\\
 p_{136}p_{256} - p_{236}\textcolor{blue}{p_{156}} - \textcolor{blue}{p_{126}}p_{356},\qquad
 p_{135}p_{256} - p_{235}\textcolor{blue}{p_{156}} - p_{125}p_{356},
\\
 p_{235}p_{246} - p_{245}p_{236} - \textcolor{blue}{p_{234}}p_{256},\qquad
 p_{145}p_{246} - p_{245}p_{146} - p_{124}\textcolor{blue}{p_{456}},
\\
 p_{136}p_{246} - p_{236}p_{146} - \textcolor{blue}{p_{126}}p_{346},\qquad
 p_{134}p_{246} - \textcolor{blue}{p_{234}}p_{146} - p_{124}p_{346},
\\
 p_{125}p_{246} - p_{245}\textcolor{blue}{p_{126}} - p_{124}p_{256},\qquad
 p_{134}p_{245} - \textcolor{blue}{p_{234}}p_{145} - p_{124}\textcolor{blue}{p_{345}},
\\
 p_{135}p_{245} - p_{235}p_{145} - p_{125}\textcolor{blue}{p_{345}},\qquad
 p_{135}p_{236} - p_{235}p_{136} - \textcolor{blue}{p_{123}}p_{356},
\\
 p_{134}p_{236} - \textcolor{blue}{p_{234}}p_{136} - \textcolor{blue}{p_{123}}p_{346},\qquad
 p_{125}p_{236} - p_{235}\textcolor{blue}{p_{126}} - \textcolor{blue}{p_{123}}p_{256},
\\
 p_{124}p_{236} - \textcolor{blue}{p_{234}}\textcolor{blue}{p_{126}} - \textcolor{blue}{p_{123}}p_{246},\qquad
 p_{134}p_{235} - \textcolor{blue}{p_{234}}p_{135} - \textcolor{blue}{p_{123}}\textcolor{blue}{p_{345}},
\\
 p_{124}p_{235} - \textcolor{blue}{p_{234}}p_{125} - \textcolor{blue}{p_{123}}p_{245},\qquad
 p_{135}p_{146} - p_{145}p_{136} - p_{134}\textcolor{blue}{p_{156}},
\\
 p_{125}p_{146} - p_{145}\textcolor{blue}{p_{126}} - p_{124}\textcolor{blue}{p_{156}},\qquad
 p_{125}p_{136} - p_{135}\textcolor{blue}{p_{126}} - \textcolor{blue}{p_{123}}\textcolor{blue}{p_{156}},
\\
 p_{124}p_{136} - p_{134}\textcolor{blue}{p_{126}} - \textcolor{blue}{p_{123}}p_{146},\qquad
 p_{124}p_{135} - p_{134}p_{125} - \textcolor{blue}{p_{123}}p_{145},
\\
f=p_{135}p_{246} - \textcolor{blue}{p_{156}}\textcolor{blue}{p_{234}} - Y - \textcolor{blue}{p_{123}}\textcolor{blue}{p_{456}} - X - \textcolor{blue}{p_{126}}\textcolor{blue}{p_{345}}.
\end{gather*}

\looseness=-1 \noindent{\bf The reduced Gr\"obner basis $\boldsymbol{\mathcal G_C(I^{\rm ex})}$.} Let $C$ be the maximal cone in the Gr\"obner fan of~$I^{\rm ex}$ whose rays are given in~\eqref{eq:rays of C}.
Recall from~\thref{thm:Gr36}$(i)$ that the associated monomial initial ideal of $C$ is generated by the exchange monomials, together with the monomials $p_{135}p_{246}$ and $XY$.
To obtain the reduced Gr\"obner basis $\mathcal G_C(I^{\rm ex})$ we need to add the missing exchange relations:
\begin{align*}
 p_{235}Y - p_{125}\textcolor{blue}{p_{234}}p_{356} - \textcolor{blue}{p_{123}}p_{256}\textcolor{blue}{p_{345}}, \qquad
 p_{134}X - p_{136}p_{145}\textcolor{blue}{p_{234}} - \textcolor{blue}{p_{123}}p_{146}\textcolor{blue}{p_{345}},
 \\
 p_{146}Y - p_{124}\textcolor{blue}{p_{156}}p_{346} - \textcolor{blue}{p_{126}}p_{134}\textcolor{blue}{p_{456}}, \qquad
 p_{256}X - \textcolor{blue}{p_{156}}p_{236}p_{245} - \textcolor{blue}{p_{126}}p_{235}\textcolor{blue}{p_{456}},
 \\
 p_{136}Y - \textcolor{blue}{p_{123}}\textcolor{blue}{p_{156}}p_{346} - \textcolor{blue}{p_{126}}p_{134}p_{356}, \qquad
 p_{346}X - p_{136}\textcolor{blue}{p_{234}}\textcolor{blue}{p_{456}} - p_{146}p_{236}\textcolor{blue}{p_{345}},
 \\
 p_{245}Y - p_{125}\textcolor{blue}{p_{234}}\textcolor{blue}{p_{456}} - p_{124}p_{256}\textcolor{blue}{p_{345}}, \qquad
 p_{125}X - \textcolor{blue}{p_{123}}\textcolor{blue}{p_{156}}p_{245} - \textcolor{blue}{p_{126}}p_{145}p_{235},
 \\
 p_{145}Y - p_{125}p_{134}\textcolor{blue}{p_{456}} - p_{124}\textcolor{blue}{p_{156}}\textcolor{blue}{p_{345}}, \qquad
 p_{124}X - \textcolor{blue}{p_{126}}p_{145}\textcolor{blue}{p_{234}} - \textcolor{blue}{p_{123}}p_{146}p_{245},
 \\
 p_{236}Y - \textcolor{blue}{p_{126}}\textcolor{blue}{p_{234}}p_{356} - \textcolor{blue}{p_{123}}p_{256}p_{346}, \qquad
 p_{356}X - p_{136}p_{235}\textcolor{blue}{p_{456}} - \textcolor{blue}{p_{156}}p_{236}\textcolor{blue}{p_{345}},
 \\
p_{135}Y - p_{125}p_{134}p_{356} - \textcolor{blue}{p_{123}}\textcolor{blue}{p_{156}}\textcolor{blue}{p_{345}}, \qquad
p_{135}X - p_{136}p_{145}p_{235} - \textcolor{blue}{p_{123}}\textcolor{blue}{p_{156}}\textcolor{blue}{p_{345}},
\\
 p_{246}Y - p_{124}p_{256}p_{346} - \textcolor{blue}{p_{126}}\textcolor{blue}{p_{234}}\textcolor{blue}{p_{456}}, \qquad
 p_{246}X - p_{146}p_{236}p_{245} - \textcolor{blue}{p_{126}}\textcolor{blue}{p_{234}}\textcolor{blue}{p_{456}}.
\end{align*}
Further, we need to add the following additional element to the generating list above (the first monomial is its leading monomial in $\init_C(I^{\rm ex})$):
\begin{gather*}
g=XY - \textcolor{blue}{p_{123}}\textcolor{blue}{p_{156}}p_{246}\textcolor{blue}{p_{345}} - \textcolor{blue}{p_{126}}p_{135}\textcolor{blue}{p_{234}}\textcolor{blue}{p_{456}} - \textcolor{blue}{p_{126}}\textcolor{blue}{p_{156}}\textcolor{blue}{p_{234}}\textcolor{blue}{p_{345}} - \textcolor{blue}{p_{123}}\textcolor{blue}{p_{156}}\textcolor{blue}{p_{234}}\textcolor{blue}{p_{456}}
\\ \hphantom{g=}
{}- \textcolor{blue}{p_{123}}\textcolor{blue}{p_{126}}\textcolor{blue}{p_{345}}\textcolor{blue}{p_{456}}.
\end{gather*}
Here, we list the identities used in the proof of~\thref{thm:Gr36}:
\begin{gather}
p_{245}f = p_{246}(p_{135}p_{245} - p_{235}{p_{145}} - p_{125}\textcolor{blue}{p_{345}})
+ p_{145}(p_{235}p_{246} - p_{245}p_{236} - \textcolor{blue}{p_{234}}{p_{256}}) \nonumber
\\ \hphantom{p_{245}f =}
{} + p_{245}(p_{145}p_{236} - \textcolor{blue}{p_{123}p_{456}} - X)
+ \textcolor{blue}{p_{234}}(p_{145}p_{256} - p_{245}\textcolor{blue}{p_{156}} - \nonumber p_{125}\textcolor{blue}{p_{456}})
\\ \hphantom{p_{245}f =}
{} - p_{125}(p_{245}p_{346} - {p_{345}}{p_{246}} - \textcolor{blue}{p_{234}p_{456}})
 + p_{245}(p_{125}p_{346} - \textcolor{blue}{p_{126}}\textcolor{blue}{p_{345}} - Y).\label{eq:f}
\\[1ex]
g = p_{256}(p_{134}X - p_{136}p_{145}\textcolor{blue}{p_{234}} - \textcolor{blue}{p_{123}}p_{146}\textcolor{blue}{p_{345}})
 - X(p_{134}p_{256} - \textcolor{blue}{p_{156}}\textcolor{blue}{p_{234}} - Y) \nonumber
\\ \hphantom{g =}
{} - p_{145}p_{234}(p_{136}p_{256} - p_{236}\textcolor{blue}{p_{156}} - \textcolor{blue}{p_{126}}p_{356})
 - p_{156}p_{234}(p_{145}p_{236} - \textcolor{blue}{p_{123}}\textcolor{blue}{p_{456}} - X) \nonumber
 \\ \hphantom{g =}
{}- p_{126}p_{234}(p_{145}p_{356} - \textcolor{blue}{p_{345}}\textcolor{blue}{p_{156}} - p_{135}\textcolor{blue}{p_{456}})\nonumber
\\ \hphantom{g =}
{} + p_{123}p_{345}(p_{146}p_{256} - p_{246}\textcolor{blue}{p_{156}} - \textcolor{blue}{p_{126}}\textcolor{blue}{p_{456}}).\label{eq:g}
\end{gather}

\noindent{\bf The exchange relations of $\boldsymbol{A_{3,6}^{\rm univ}}$} (or equivalently the lifts of the elements of $\mathcal G_{C}(I^{\rm ex})$) are:
\begin{gather*}
p_{245}p_{356} - t_{3}t_{12}t_{16}\textcolor{blue}{p_{345}}p_{256}-t_{5}t_{6}t_{9}p_{235}\textcolor{blue}{p_{456}},
\\[.2ex]
p_{146}p_{356}-t_{3}t_{7}t_{16}p_{346}\textcolor{blue}{p_{156}}-t_{5}t_{9}t_{13}p_{136} \textcolor{blue}{p_{456}},
\\[.2ex]
p_{145}p_{256}-t_{4}t_{7}t_{14}p_{245}\textcolor{blue}{p_{156}}-t_{8}t_{9}t_{11}p_{125}\textcolor{blue}{p_{456}},\\[.2ex]
p_{136}p_{256}-t_{6}t_{7}t_{14}p_{236}\textcolor{blue}{p_{156}}-t_{1}t_{8}t_{11}\textcolor{blue}{p_{126}}p_{356},\\[.2ex]
p_{134}p_{245}-t_{2}t_{5}t_{6}\textcolor{blue}{p_{234}}p_{145}-t_{8}t_{12}t_{16}p_{124}\textcolor{blue}{p_{345}},\\[.2ex]
p_{134}p_{236}-t_{2}t_{4}t_{5}\textcolor{blue}{p_{234}}p_{136}-t_{8}t_{15}t_{16}\textcolor{blue}{p_{123}}p_{346},\\[.2ex]
p_{125}p_{146}-t_{1}t_{5}t_{13}p_{145}\textcolor{blue}{p_{126}}-t_{7}t_{10}t_{16}p_{124}\textcolor{blue}{p_{156}},\\[.2ex]
p_{124}p_{136}-t_{1}t_{3}t_{11}p_{134}\textcolor{blue}{p_{126}}-t_{6}t_{14}t_{15}\textcolor{blue}{p_{123}}p_{146},\\[.2ex]
p_{125}p_{236}-t_{1}t_{4}t_{5}p_{235}\textcolor{blue}{p_{126}}-t_{10}t_{15}t_{16}\textcolor{blue}{p_{123}}p_{256},\\[.2ex]
p_{124}p_{235}-t_{2}t_{3}t_{11}\textcolor{blue}{p_{234}}p_{125}-t_{13}t_{14}t_{15}\textcolor{blue}{p_{123}}p_{245},\\[.2ex]
p_{235}p_{346}-t_{12}t_{13}t_{14}\textcolor{blue}{p_{345}}p_{236}-t_{2}t_{10}t_{11}\textcolor{blue}{p_{234}}p_{356},\\[.2ex]
p_{145}p_{346}-t_{4}t_{12}t_{14}\textcolor{blue}{p_{345}}p_{146}-t_{9}t_{10}t_{11}p_{134}\textcolor{blue}{p_{456}},\\[.2ex]
p_{146}p_{256}-t_{7}p_{246}\textcolor{blue}{p_{156}}-t_{1}t_{5}t_{8}t_{9}t_{11}t_{13}\textcolor{blue}{p_{126}}\textcolor{blue}{p_{456}},\\[.2ex]
p_{124}p_{236}-t_{15}\textcolor{blue}{p_{123}}p_{246}-t_{1}t_{2}t_{3}t_{4}t_{5}t_{11}\textcolor{blue}{p_{234}}\textcolor{blue}{p_{126}},\\[.2ex]
p_{125}p_{136}-t_{1}p_{135}\textcolor{blue}{p_{126}}-t_{6}t_{7}t_{10}t_{14}t_{15}t_{16}\textcolor{blue}{p_{123}}\textcolor{blue}{p_{156}},
\\[.2ex]
p_{145}p_{236}-{}_{\phantom{1}}t_{4}X-t_{8}t_{9}t_{10}t_{11}t_{15}t_{16}\textcolor{blue}{p_{123}}\textcolor{blue}{p_{456}},\\[.2ex]
p_{136}p_{245}-{}_{\phantom{1}} t_{6}X - t_{1}t_{3}t_{8}t_{11}t_{12}t_{16}\textcolor{blue}{p_{126}}\textcolor{blue}{p_{345}},\\[.2ex]
p_{146}p_{235} - t_{13}X - t_{2}t_{3}t_{7}t_{10}t_{11}t_{16}\textcolor{blue}{p_{156}p_{234}},
\\[.2ex]
p_{135}p_{256}-t_{8}t_{11}p_{125}p_{356}-t_{4}t_{5}t_{6}t_{7}t_{14}p_{235}\textcolor{blue}{p_{156}},\\[.2ex]
p_{134}p_{246}-t_{8}t_{16}p_{124}p_{346}-t_{2}t_{4}t_{5}t_{6}t_{14}\textcolor{blue}{p_{234}}p_{146},\\[.2ex]
p_{246}p_{356}-t_{3}t_{16}p_{346}p_{256}-t_{5}t_{6}t_{9}t_{13}t_{14}p_{236}\textcolor{blue}{p_{456}},\\[.2ex]
p_{135}p_{245}-t_{5}t_{6}p_{235}p_{145}-t_{3}t_{8}t_{11}t_{12}t_{16}p_{125}\textcolor{blue}{p_{345}},\\[.2ex]
p_{136}p_{246}-t_{6}t_{14}p_{236}p_{146}-t_{1}t_{3}t_{8}t_{11}t_{16}\textcolor{blue}{p_{126}}p_{346},\\[.2ex]
p_{145}p_{246}-t_{4}t_{14}p_{245}p_{146}-t_{8}t_{9}t_{10}t_{11}t_{16}p_{124}\textcolor{blue}{p_{456}},\\[.2ex]
p_{125}p_{246}-t_{10}t_{16}p_{124}p_{256}-t_{1}t_{4}t_{5}t_{13}t_{14}p_{245}\textcolor{blue}{p_{126}},\\[.2ex]
p_{135}p_{236}-t_{4}t_{5}p_{235}p_{136}-t_{8}t_{10}t_{11}t_{15}t_{16}\textcolor{blue}{p_{123}}p_{356},\\[.2ex]
p_{135}p_{146}-t_{5}t_{13}p_{145}p_{136}-t_{3}t_{7}t_{10}t_{11}t_{16}p_{134}\textcolor{blue}{p_{156}},\\[.2ex]
p_{124}p_{135}-t_{3}t_{11}p_{134}p_{125}-t_{5}t_{6}t_{13}t_{14}t_{15}\textcolor{blue}{p_{123}}p_{145},\\[.2ex]
p_{135}p_{346}-t_{10}t_{11}p_{134}p_{356}-t_{4}t_{5}t_{12}t_{13}t_{14}\textcolor{blue}{p_{345}}p_{136},\\[.2ex]
p_{235}p_{246}-t_{13}t_{14}p_{245}p_{236}-t_{2}t_{3}t_{10}t_{11}t_{16}\textcolor{blue}{p_{234}}p_{256},\\[.2ex]
p_{245}p_{346}-t_{12}\textcolor{blue}{p_{345}}p_{246}-t_{2}t_{5}t_{6}t_{9}t_{10}t_{11}\textcolor{blue}{p_{234}}\textcolor{blue}{p_{456}},\\[.2ex]
p_{145}p_{356}-t_{9}p_{135}\textcolor{blue}{p_{456}}-t_{3}t_{4}t_{7}t_{12}t_{14}t_{16}\textcolor{blue}{p_{345}}\textcolor{blue}{p_{156}},\\[.2ex]
p_{134}p_{235}-t_{2}\textcolor{blue}{p_{234}}p_{135}-t_{8}t_{12}t_{13}t_{14}t_{15}t_{16}\textcolor{blue}{p_{123}}\textcolor{blue}{p_{345}},\\[.2ex]
p_{124}p_{356} -{}_{\phantom{1}} t_{3}Y - t_{5}t_{6}t_{9}t_{13}t_{14}t_{15}\textcolor{blue}{p_{123}}\textcolor{blue}{p_{456}},\\[.2ex]
p_{134}p_{256} - {}_{\phantom{1}}t_{8}Y - t_{2}t_{4}t_{5}t_{6}t_{7}t_{14}\textcolor{blue}{p_{156}p_{234}}, \\[.2ex]
p_{346}p_{125} - t_{10}Y - t_{1}t_{4}t_{5}t_{12}t_{13}t_{14}\textcolor{blue}{p_{126}}\textcolor{blue}{p_{345}},
\\[.2ex]
p_{125}X-t_{1}t_{5}\textcolor{blue}{p_{126}}p_{145}p_{235}-t_{7}t_{10}t_{14}t_{15}t_{16} \textcolor{blue}{p_{123}}\textcolor{blue}{p_{156}}p_{245},\\[.2ex]
p_{145}Y-t_{9}t_{11}p_{125}p_{134}\textcolor{blue}{p_{456}}-t_{4}t_{7}t_{12}t_{14}t_{16}p_{124}\textcolor{blue}{p_{156}}\textcolor{blue}{p_{345}},\\[.2ex]
p_{124}X-t_{14}t_{15}\textcolor{blue}{p_{123}}p_{146}p_{245}-t_{1}t_{2}t_{3}t_{5}t_{11}\textcolor{blue}{p_{126}}p_{145}\textcolor{blue}{p_{234}},\\[.2ex]
p_{236}Y-t_{15}t_{16}\textcolor{blue}{p_{123}}p_{256}p_{346}-t_{1}t_{2}t_{4}t_{5}t_{11}\textcolor{blue}{p_{126}}\textcolor{blue}{p_{234}}p_{356},\\[.2ex]
p_{134}X-t_{2}t_{5}p_{136}p_{145}\textcolor{blue}{p_{234}}-t_{8}t_{12}t_{14}t_{15}t_{16}\textcolor{blue}{p_{123}}p_{146}\textcolor{blue}{p_{345}},\\[.2ex]
p_{235}Y-t_{2}t_{11}p_{125}\textcolor{blue}{p_{234}}p_{356}-t_{12}t_{13}t_{14}t_{15}t_{16}\textcolor{blue}{p_{123}}p_{256}\textcolor{blue}{p_{345}},\\[.2ex]
p_{246}X-t_{14}p_{146}p_{236}p_{245}-t_{1}t_{2}t_{3}t_{5}t_{8}t_{9}t_{10}t_{11}^2t_{16}\textcolor{blue}{p_{126}}\textcolor{blue}{p_{234}}\textcolor{blue}{p_{456}},\\[.2ex]
p_{246}Y-t_{16}p_{124}p_{256}p_{346}-t_{1}t_{2}t_{4}t_{5}^2t_{6}t_{9}t_{11}t_{13}t_{14}\textcolor{blue}{p_{126}}\textcolor{blue}{p_{234}}\textcolor{blue}{p_{456}},\\[.2ex]
p_{256}X-t_{7}t_{14}\textcolor{blue}{p_{156}}p_{236}p_{245}-t_{1}t_{5}t_{8}t_{9}t_{11}\textcolor{blue}{p_{126}}p_{235}\textcolor{blue}{p_{456}},\\[.2ex]
p_{146}Y-t_{7}t_{16}p_{124}\textcolor{blue}{p_{156}}p_{346}-t_{1}t_{5}t_{9}t_{11}t_{13}\textcolor{blue}{p_{126}}p_{134}\textcolor{blue}{p_{456}},\\[.2ex]
p_{346}X-t_{12}t_{14}p_{146}p_{236}\textcolor{blue}{p_{345}}-t_{2}t_{5}t_{9}t_{10}t_{11}p_{136}\textcolor{blue}{p_{234}}\textcolor{blue}{p_{456}},\\[.2ex]
p_{136}Y-t_{1}t_{11}\textcolor{blue}{p_{126}}p_{134}p_{356}-t_{6}t_{7}t_{14}t_{15}t_{16}\textcolor{blue}{p_{123}}\textcolor{blue}{p_{156}}p_{346},\\[.2ex]
p_{356}X-t_{5}t_{9}p_{136}p_{235}\textcolor{blue}{p_{456}}-t_{3}t_{7}t_{12}t_{14}t_{16}\textcolor{blue}{p_{156}}p_{236}\textcolor{blue}{p_{345}},\\[.2ex]
p_{245}Y-t_{12}t_{16}p_{124}p_{256}\textcolor{blue}{p_{345}}-t_{2}t_{5}t_{6}t_{9}t_{11}p_{125}\textcolor{blue}{p_{234}}\textcolor{blue}{p_{456}},\\[.2ex]
p_{135}X-t_{5}p_{136}p_{145}p_{235}-t_{3}t_{7}t_{8}t_{10}t_{11}t_{12}t_{14}t_{15}t_{16}^2\textcolor{blue}{p_{123}}\textcolor{blue}{p_{156}}\textcolor{blue}{p_{345}},\\[.2ex]
p_{135}Y-t_{11}p_{125}p_{134}p_{356}-t_{4}t_{5}t_{6}t_{7}t_{12}t_{13}t_{14}^2t_{15}t_{16}\textcolor{blue}{p_{123}}\textcolor{blue}{p_{156}}\textcolor{blue}{p_{345}}.
\end{gather*}
The lifts of $f$, $g$ (the elements of $\mathcal G_C(I^{\rm ex})$ that do not correspond to exchange relations) are:
\begin{gather*}
\tilde{f}= p_{135}p_{246} - t_{2}t_{3}t_{4}t_{5}t_{6}t_{7}t_{10}t_{11}t_{14}t_{16}\textcolor{blue}{p_{156}}\textcolor{blue}{p_{234}} - t_{3}t_{8}t_{10}t_{11}t_{16}Y
\\ \hphantom{\tilde{f}=}
{}- t_{5}t_{6}t_{8}t_{9}t_{10}t_{11}t_{13}t_{14}t_{15}t_{16}\textcolor{blue}{p_{123}}\textcolor{blue}{p_{456}}
- t_{4}t_{5}t_{6}t_{13}t_{14}X - t_{1}t_{3}t_{4}t_{5}t_{8}t_{11}t_{12}t_{13}t_{14}t_{16}\textcolor{blue}{p_{126}}\textcolor{blue}{p_{345}},
\\
\tilde{g}= XY - t_{7}t_{12}t_{14}t_{15}t_{16}\textcolor{blue}{p_{123}}\textcolor{blue}{p_{156}} p_{246}\textcolor{blue}{p_{345}} - t_{1}t_{2}t_{5}t_{9}t_{11}\textcolor{blue}{p_{126}}p_{135}\textcolor{blue}{p_{234}}\textcolor{blue}{p_{456}} \\ \hphantom{\tilde{g}=}
{}- t_{1}t_{2}t_{3}t_{4}t_{5}t_{7}t_{11}t_{12}t_{14}t_{16}\textcolor{blue}{p_{126}} \textcolor{blue}{p_{156}}\textcolor{blue}{p_{234}}\textcolor{blue}{p_{345}}
- t_{2}t_{5}t_{6}t_{7}t_{9}t_{10}t_{11}t_{14}t_{15}t_{16}\textcolor{blue}{p_{123}} \textcolor{blue}{p_{156}}\textcolor{blue}{p_{234}}\textcolor{blue}{p_{456}}
\\ \hphantom{\tilde{g}=}
{}- t_{1}t_{5}t_{8}t_{9}t_{11}t_{12}t_{13}t_{14}t_{15}t_{16}\textcolor{blue}{p_{123}} \textcolor{blue}{p_{126}}\textcolor{blue}{p_{345}}\textcolor{blue}{p_{456}}.
\end{gather*}

\subsection*{Acknowledgements}
This work was partially supported by CONACyT grant CB2016 no. 284621.
L.B. was supported by ``Programa de Becas Posdoctorales en la UNAM 2018'' Instituto de Matem\'aticas, Universidad Nacional Aut\'onoma de M\'exico.
F.M. thanks the Instituto de Matem\'aticas, UNAM Unidad Oaxaca for their hospitality during this project and also acknowledges partial supports by the EPSRC Early Career Fellowship EP/R023379/1, the Starting Grant of Ghent University BOF/STA/201909/038, and the FWO grants (G023721N and G0F5921N).
L.B. would like to thank Kiumars Kaveh for providing an opportunity to present this work, and further Nathan Ilten for his insightful comments.
A special thanks goes to Christopher Manon who patiently explained his joint work with Kaveh to us. The connection between cluster mutation and Escobar--Harada's flip map was first observed in a discussion with Megumi Harada during the MFO-Workshop on Toric geometry in 2019, see~\cite{BHM}.
We would like to thank the organizers of the meeting and in particular Megumi Harada for explaining us the results of~\cite{EH-NObodies}.
We are grateful to Brenda Policarpo Sibaja and Nathan Ilten for pointing out misprints in an earlier version of this paper.

\pdfbookmark[1]{References}{ref}
\LastPageEnding

\end{document}